\documentclass[11pt,twoside]{preprint}
\usepackage[a4paper,innermargin=1.2in,outermargin=1.2in,
bottom=1.5in,marginparwidth=1in,marginparsep
=3mm]{geometry}
\usepackage{amsmath,amsthm,amssymb,enumerate}
\usepackage{hyperref}
\usepackage{mathrsfs}
\usepackage{breakurl}
\usepackage[nameinlink,capitalize]{cleveref}
\usepackage{secdot}
\usepackage{tikz}
\usepackage{bbm}

\usepackage{xcolor}
\usepackage{mhequ}
\usepackage{comment}
\definecolor{bigpart}{rgb}{0.9,0.3,0.4}
\definecolor{ccorner}{rgb}{0.45,0.3,0.85}
\definecolor{ggreen}{rgb}{0.3,0.85,0.3}

\definecolor{Brown}{rgb}{.75,.5,.25}
\definecolor{DGreen}{rgb}{0,0.55,0}
\definecolor{Olive}{rgb}{0.41,0.55,0.13}

\renewcommand{\theequation}{\thesection.\arabic{equation}}

\newtheorem{theorem}{Theorem}[section]
\newtheorem{lemma}[theorem]{Lemma}
\newtheorem{proposition}[theorem]{Proposition}
\newtheorem{corollary}[theorem]{Corollary}

\theoremstyle{definition}

\newtheorem{example}[theorem]{Example}

\newtheorem{innercustomthm}{Assumption}
\newenvironment{assumption}[1]
  {\renewcommand\theinnercustomthm{#1}\innercustomthm}
  {\endinnercustomthm}

\theoremstyle{remark}
\newtheorem{remark}[theorem]{Remark}

\def\db#1{\| #1 \|}

\newcommand{\ddnew}[4]{[ #1 ]_{{ \mathbf{C}}^{#2}_{#3}|\mathcal{G},#4}}

\newcommand{\vn}[1]{{\vert\kern-0.23ex\vert\kern-0.23ex\vert #1 
    \vert\kern-0.23ex\vert\kern-0.23ex\vert}}
\usepackage{mathtools}

\allowdisplaybreaks

\newcommand{\A}{\mathcal{A}}
\newcommand{\C}{\mathcal{C}}
\newcommand{\F}{\mathcal{F}}
\newcommand{\bF}{\mathbb{F}}

\renewcommand{\L}{\mathcal{L}}
\newcommand{\bP}{\mathbb{P}}

\newcommand{\cE}{\mathcal{E}}

\newcommand{\cS}{\mathcal{S}}

\newcommand{\G}{\mathcal{G}}
\newcommand{\N}{\mathbb{N}}
\newcommand{\R}{\mathbb{R}}
\newcommand{\Z}{\mathbb{Z}}

\newcommand{\nn}{\nonumber}

\def\cG{\mathcal{G}}
\newcommand{\eps}{\varepsilon}
\def\E{\hskip.15ex\mathbb{E}\hskip.10ex}
\renewcommand{\P}{\bP}

\renewcommand{\Re}{\operatorname{Re}}
\renewcommand{\d}{\partial}
\renewcommand{\phi}{\varphi}

\newcommand{\one}{\mathbbm{1}}

\newcommand{\cP}{\mathcal{P}}

\newcommand{\cF}{\mathcal{F}}
\newcommand{\cA}{\mathcal{A}}
\newcommand{\cC}{\mathcal{C}}
\newcommand{\scC}{\mathscr{C}}

\DeclareMathOperator{\Law}{Law}

\newcommand{\wt}{\widetilde}
\newcommand{\nalpha}{\wt\alpha}

\def\half{\frac{1}{2}}

\def\({\left(}
\def\){\right)}

\begin{document}
\title{Strong rate of convergence of the Euler scheme for SDEs with irregular drift driven by L\'evy noise}
\author{Oleg Butkovsky\thanks{Weierstrass Institute, Mohrenstra\ss e 39, 10117 Berlin, Germany$\qquad$\url{oleg.butkovskiy@gmail.com}}\,, Konstantinos Dareiotis\thanks{University of Leeds, Woodhouse, LS2 9JT Leeds, United Kingdom$\qquad$\url{k.dareiotis@leeds.ac.uk}}, and M\'at\'e Gerencs\'er\thanks{TU Wien, Wiedner Hauptstra\ss e 8-10, 1040 Vienna, Austria$\qquad$\url{mate.gerencser@tuwien.ac.at}}}
\maketitle
\begin{abstract}
We study the strong rate of convergence of the Euler--Maruyama scheme for a multidimensional stochastic differential equation (SDE)
$$
dX_t = b(X_t) \, dt + dL_t,
$$
with irregular $\beta$-H\"older drift, $\beta > 0$, driven by a L\'evy process with exponent $\alpha \in (0, 2]$. For $\alpha \in [2/3, 2]$, we obtain strong $L_p$ and almost sure convergence rates in the entire range $\beta > 1 - \alpha/2$, where the SDE is known to be strongly well-posed. This significantly improves the current state of the art, both in terms of convergence rate and the range of~$\alpha$. Notably, the obtained convergence rate does not depend on $p$, which is a novelty even in the case of smooth drifts.
As a corollary of the obtained moment-independent error rate, we show that the Euler--Maruyama scheme for such SDEs converges almost surely and obtain an explicit convergence rate. Additionally, as a byproduct of our results, we derive strong $L_p$ convergence rates for approximations of nonsmooth additive functionals of a L\'evy process.
Our technique is based on a new extension of stochastic sewing arguments and L\^e's quantitative John-Nirenberg inequality.
\end{abstract}
\section{Introduction}
We consider the stochastic differential equation 
\begin{equation}\label{eq:main-SDE}
	dX_t=b(X_t)\,dt+dL_t,\quad t\ge0,\quad X_0=x_0,
\end{equation}
driven by a $d$-dimensional L\'evy process $L$. Here the coefficient $b$ is a measurable function $\R^d\to\R^d$, and the initial condition $x_0\in\R^d$. Throughout the article the dimension $d\in\N=\{1,2,\ldots\}$ is arbitrary.

The `strength' of a L\'evy process can often be characterised by a single parameter $\alpha\in(0,2]$ called the \emph{stable index} (for various examples see \cref{Sec:ex}), with $\alpha=2$ corresponding to the usual Brownian motion.
This parameter can be used to describe the regularisation provided by $L$. Indeed, assuming a natural nondegeneracy condition on the jump measure of $L$, 
\eqref{eq:main-SDE} has a unique strong solution whenever $b$ belongs to the H\"older space $\C^\beta(\R^d)$, where $\beta$ satisfies
\begin{equ}\label{eq:exponents-WP}
	\beta>1-\frac{\alpha}{2}.
\end{equ}
This solution theory was developed in \cite{Priola1,Priola_flow,Priola2,Chen32}.

Once the well-posedness is understood, it is natural to investigate basic discretisations of the equation. The most classical Euler-Maruyama approximation of \eqref{eq:main-SDE} reads as
\begin{equation}\label{eq:EMsde}
	dX^n_t=b(X^n_{\kappa_n(t)})\,dt+dL_t,\quad X_0^n=x_0^n,
\end{equation}
with $\kappa_n(t):=\lfloor nt\rfloor n^{-1}$.
For Lipschitz $b$ the convergence  of $X^n$ to $X$ can be seen by very straightforward arguments. The $L_p$-rate of convergence obtained this way, however, deteriorates for large $p$ (see Remark \ref{rem:explanation} (ii) below for some details). Even though our main focus will be on non-Lipschitz coefficients, it is noteworthy and perhaps surprising that our method coincidentally also solves this moment issue and therefore gives new results even in the Lipschitz case.

One of the first results regarding convergence of $X^n$ to $X$ in the case of irregular $b$ is \cite{PT}. It is shown there that
\begin{equation}\label{PamenTaguchi}
	\Bigl\|\sup_{t\in[0,1]}|X_t-X^n_t|\Bigr\|_{L_p(\Omega)} \le N n^{-(\frac\beta2\wedge\frac1p)},\quad n\in\N,	
\end{equation}
for the case where $L$ is a truncated symmetric $\alpha$--stable process, $\alpha\in(1,2)$, $\beta>2-\alpha$; here $N=N(\alpha,\beta,p,d)$ is a certain positive constant.
This result was improved in \cite{KSEM2019} in three directions: first, the condition on $\beta$ is relaxed to $\beta>2/\alpha-1$; second, the rate of convergence in \eqref{PamenTaguchi}  is improved and is  $\frac\beta\alpha\wedge\frac1p$; third, the class of considered L\'evy processes  is significantly extended and additionally includes  standard isotropic stable processes, tempered stable processes, relativistic stable processes, and others.
The same rate $\frac\beta\alpha\wedge\frac1p$ is derived \cite{MX} even for multiplicative noise, under the further relaxed condition $\beta>1-\alpha/2$, under which strong solutions are known to exist.
Very recently, for the case of standard isotropic stable processes,  \cite{LevySobolev} showed that the strong $L_p$ rate of  $\frac\beta\alpha\wedge\frac1p$ holds in the whole range $\beta>1-\alpha/2$ even when $\beta$ denotes regularity only in a certain Sobolev scale. A standard example of a coefficient that possesses Sobolev but not H\"older regularity is one with discontinuity of the first kind, in this case scalar SDEs driven by a Brownian motion and a finite activity Poisson process were studied in \cite{PSz21}.

From the discussion above, the reader may notice the following gaps in the literature. 
All of the works mentioned above consider the case $\alpha\in[1,2]$; recall however that the strong well-posedness of \eqref{eq:main-SDE} is known in the whole range $\alpha\in(0,2]$. Further, the aforementioned moment issue is still present: the convergence rate becomes arbitrary slow for very large~$p$. Note that this also has the consequence that one can not deduce almost sure convergence of $|X_t-X_t^n|$ as $n\to\infty$. Indeed, to show this one has to prove the bound $\E |X_t-X_t^n|^p\le N n^{-1-\eps}$ for some $p>0$, $\eps>0$, while the best available bound is $\E |X_t-X_t^n|^p\le N n^{-1}$, which is not sufficient.

The present paper closes these gaps; the novelties  can be summarised as follows.
First, our methods are completely free from the moment issue.
As alluded to above and detailed in Remark \ref{rem:explanation} (ii) below, this makes our results new even in the smooth drifts, where the strong $L_p$ rate, which one trivially gets, namely, $\frac1\alpha\wedge\frac1p$, $p\geq 1$, is improved to $1$.
Another factor that may provide poor rate in the previous results is if $\beta/\alpha$ is small. This issue is also not present here, the obtained rate is always strictly above $1/2$.
We also obtain almost sure convergence (and rate) of the Euler-Maruyama scheme, to our knowledge for the first time for SDEs of the type \eqref{eq:main-SDE}.

Second, we are not restricted to $\alpha\in[1,2]$. In the regime $\alpha\in[2/3,2]$ our assumption on $\beta$ coincides with the optimal condition \eqref{eq:exponents-WP}. 
We cover some (but not the optimal) range of irregular drifts when $\alpha\in(1/2,2/3)$.
For $\alpha\in(0,1/2]$ we require $b$ to be regular (i.e. more than Lipschitz), the contribution in this case is the handling of high moments.

Third, we have also shown that the solution of \eqref{eq:main-SDE} is the limit of the corresponding Picard’s successive approximations. This extends to the non-smooth case the corresponding results from \cite{Yam81,Tan92,NO22}.

Fourth,  the class of considered driving L\'evy processes  is fairly large. It is similar to \cite{KSEM2019}, so it includes not only the standard stable processes but also their different ``relatives''. The conditions on the L\'evy process, see \ref{A:1}--\ref{A:3} below, are rather general and are easy to verify for various examples, see  \cref{Sec:ex}.

Finally, as a byproduct of our proofs, we obtain significant improvements in a different but related problem of numerical stochastic analysis: the question of approximating additive functionals of stochastic processes given high-frequency observations. To wit, consider an $\mathbb{R}^d$-valued stochastic process $Y$, a measurable function $f: \mathbb{R}^d \to \mathbb{R}^d$, and the 'occupation time functional'
\begin{equation}
	\Gamma = \int_0^1 f(Y_s) \, ds.
\end{equation}
Approximation of $\Gamma$ is an interesting question in itself (for a detailed overview of the literature, we refer to \cite[Section~1]{Alt}). A natural estimation scheme is given by
\begin{equation}
	\Gamma^n = \int_0^1 f(Y^n_{\kappa_n(s)}) \, ds = \frac{1}{n}\sum_{i=0}^n f(Y^n_{i/n}),
\end{equation}
where $Y^n$ is the process approximating $Y$. If the increments of $Y$ can be simulated directly, then one can, of course, take $Y^n:=Y$. For the case where $Y$ is a Markov process, whose density satisfies certain estimates (e.g., $Y$ is an $\alpha$-stable process), $Y^n=Y$, and $f$ is a bounded or H\"older continuous function, $L_p$ error bounds for $\Gamma-\Gamma^n$ were obtained in \cite{GK14, GKK15}. For the case $p=2$, Altmeyer \cite{Alt} improved the convergence rate to the one which was shown to be optimal in some setups \cite{AltG, Mark22}.

In this article, we take the best of both worlds. Namely, we obtain the $L_p$ convergence rate for the error $\Gamma-\Gamma^n$ as in \cite{Alt} but for general $p\ge2$. Furthermore, we provide $L_p$ bounds on the convergence rate of the error for the case $Y^n\neq Y$. This is relevant if  the process $Y$ cannot be simulated directly (this is often the case if $Y$ is a solution to SDE \eqref{eq:main-SDE}) and we have access only to its approximation $Y^n$ (which can be its Euler approximation \eqref{eq:EMsde}).

Our approach is rather different from the one used in the papers \cite{PT,MX,KSEM2019,LevySobolev} discussed above, as we do not rely on any form of Zvonkin transformation or It\^o-Tanaka trick (in fact, It\^o's formula is not even once applied for either $X$ or $X^n$).
Instead, we employ stochastic sewing, which originates from the work of L\^e \cite{Khoa} and has been developed for discretisation problems in \cite{ButDarGEr, BDG-SPDE, KML,le2021taming}.
In \cite{ButDarGEr} the Euler-Maruyama scheme for fractional Brownian motion-driven SDEs is studied. One may hope due to the scaling correspondance between stable indices and Hurst parameters ($\alpha\leftrightarrow 1/H$) that the methods therein translate  easily to the L\'evy case. This is unfortunately not the case, for several reasons. First, the usual regime $H<1$ corresponds to $\alpha>1$. To consider $\alpha<1$, one needs tools from the $H>1$ case, in particular the shifted stochastic sewing lemma \cite{Mate20}. Second, high moments of L\'evy processes do not scale (or they do not even exist for several examples).
This is related to the aforementioned moment issue yielding poor rate for large moments in preceding literature. Overcoming this challenge relies on the recently obtained quantitative John-Nirenberg inequality \cite{Le22}.
Finally, there is no useful form of Girsanov theorem to remove the drift at any part of the analysis.

This article deals only with the strong rate of convergence. Clearly, the weak rate of convergence is at least as good as the strong rate of convergence; already with this simple observation our results imply weak rates which are better than the ones available in the literature \cite{Platen-weak,Mikulevicius-weak} in the range of parameters that we cover. An interesting and challenging question is whether these weak rates can be improved further, and whether the range of $\beta$ can be upgraded to $\beta>1-\alpha$ (in this range SDE \eqref{eq:main-SDE} is weakly well--posed \cite{Kulik19, LZ22}).
We leave this for the future work.

\medskip
The rest of the paper is organized as follows. Our main result concerning the $L_p$ convergence of the numerical scheme is formulated in \cref{S:main}. Examples of L\'evy processes satisfying the assumptions of the convergence theorem are given in \cref{Sec:ex}. A number of technical tools needed for the proofs are collected in \cref{S:Prel}. The main results are proved in \cref{S:proofs}, whilst the proofs of some technical auxiliary statements are placed in the Appendix.

\medskip
\textbf{Acknowledgements.}
The authors are grateful to Randolf Altmeyer and Mark Podolskij for very helpful discussions regarding approximation of additive  functionals of a L\'evy process, and to Khoa L\^e for many useful conversations and for bringing to our attention the articles \cite{GK14,GKK15}. We would like to thank the referees for their helpful comments and feedback.
OB has received funding from the DFG Research Unit FOR~2402 and is funded by the Deutsche Forschungsgemeinschaft (DFG, German Research Foundation) under Germany's Excellence Strategy --- The Berlin Mathematics Research Center MATH+ (EXC-2046/1, project ID: 390685689, sub-project EF1-22).
MG was funded by the Austrian Science Fund (FWF) Stand-Alone programme P 34992.
The project has been conceived during the stay of the authors at the Hausdorff
Research Institute for Mathematics (HIM), Bonn. Significant progress on the project has been achieved during the visits of the authors to TU Wien,  Mathematisches Forschungsinstitut Oberwolfach (mini-workshop 2207c), and  Universit\`a degli Studi di Torino. We would like to thank all these institutions and their staff for providing excellent working conditions, support, and hospitality. 

\section{Main results}\label{S:main}

We begin by introducing the basic notation. 
For $\beta\in(0,1)$ and a Borel subset $Q$ of $\R^k$, $k\in\N$, let $\C^\beta(Q)$ be the corresponding H\"older space, that is, the set of functions $f\colon Q\to\R$ such that
\begin{equ}
	\|f\|_{\C^\beta(Q)}:=\sup_{x\in Q}|f(x)|+[f]_{\C^\beta(Q)}:=\sup_{x\in Q}|f(x)|+\sup_{x\neq y\in Q}\frac{|f(x)-f(y)|}{|x-y|^\beta}<\infty.
\end{equ}
With a slightly unconventional notation we set $\C^0(Q)$ to be the set of bounded measurable  functions (not necessarily continuous) equipped with the supremum norm.
The definition of the analogous spaces for $\R^d$-valued functions is simply understood coordinate-wise.
For $\beta\in[1,\infty)$ we denote by $\C^\beta(Q)$ the set of functions whose weak derivatives of order $0,1,\ldots,\lfloor \beta\rfloor$ all have representatives belonging to $\C^{\beta-\lfloor\beta\rfloor}(Q)$. In the particular case $Q = \R^d$ sometimes we will use a shorthand and write $\C^\beta$ instead of $\C^\beta(\R^d)$ in order not to overcrowd the notation.

Fix a probability space $(\Omega, \F, \P)$ and on it a $d$-dimensional L\'evy process $L$, equipped with a right-continuous, complete filtration $\bF=(\cF_t)_{t\in[0,1]}$.
The conditional expectation given $\F_t$ will be denoted by $\E^t$.
The Markov transition semigroup associated to the process $L$ is denoted by $\cP=(\cP_t)_{t\geq 0}$, and the generator of $\cP$ is denoted by $\L$.

We fix $\alpha\in(0,2]$,  $d\in\N$ and impose the following assumptions on $\cP$ and $\L$. 

\begin{assumption}{H1}[gradient type bound on the semigroup]\label{A:1}
	There exists  a constant $M$ such that for any $f\in\C^0(\R^d)$, one has 
	\begin{equation}\label{eq:K1newA}
		\|\nabla\cP_t f\|_{\C^0(\R^d)}\leq M t^{-1/\alpha} \|f\|_{\C^0(\R^d)},\quad 0<t\le1.
	\end{equation}
\end{assumption}

\begin{assumption}{H2}[action of the generator]\label{A:2}
	For $\delta=0,1$, and any $\eps>0$ there exists  a constant $M=M(\delta,\eps)$ such that for any $f\in\C^{\alpha+\delta+\eps}(\R^d)$ vanishing at infinity, one has 
	\begin{equation}\label{eq:K2newA}
		\|\L f  \|_{\C^\delta(\R^d)}\leq M \|f\|_{\C^{\alpha+\delta+\eps}(\R^d)}.
	\end{equation}
\end{assumption}

\begin{assumption}{H3}[moment conditions]\label{A:3}
	For any $p\in(0,\alpha)$, $\eps>0$, there exists a constant $M=M(p,\eps)$ such that
	\begin{equation*}
		\E [|L_t|^p\wedge1] \le M  t^{\frac{p}{\alpha}-\eps},\quad 0<t\le1.
	\end{equation*}
\end{assumption}
With some abuse of notation, in the sequel when we refer to the parameter $M$ given a process $L$ satisfying the above assumptions, we understand the collection of all of the $M$-s in H1-H3.

We provide a long list of examples of processes satisfying H1-H3 in \cref{Sec:ex}, let us here just briefly mention three of the most standard examples.
First let $\alpha\in(0,2)$ and $L$ be the standard $d$-dimensional $\alpha$--stable process. That is, $L$ is a L\'evy process whose characteristic function is 
$$
\E e^{i\langle\lambda, L_t\rangle}=e^{-tc_\alpha |\lambda|^\alpha},\quad\lambda\in\R^d,\, t\ge0,
$$ 
for some constant $c_\alpha>0$. In this case H1-H3 are satisfied.
A similar example is the  $d$-dimensional cylindrical $\alpha$--stable process, that is, a process $L$ whose coordinates are $d$ independent $1$-dimensional standard $\alpha$--stable processes. Its characteristic funtion is given by 
$$
\E e^{i\langle\lambda, L_t\rangle}=e^{-t\widetilde  c_\alpha \sum_{i=1}^d |\lambda_i|^\alpha},\quad\lambda\in\R^d,\, t\ge0,
$$
for some constant $\widetilde c_\alpha>0$. In this case H1-H3 are also satisfied.
Finally the most standard of the most standard examples is the $d$--dimensional Brownian motion, which satisfies H1-H3 with $\alpha=2$.

\medskip\textbf{Convention on the operator $\wedge$.}
The expression of the form $c_1+c_2\wedge c_3$, where $c_i\in\R$, will quite often appear in the paper. We will always mean that in this expression the minimum is taken first and then the addition, thus it equals to $c_1+(c_2\wedge c_3)=c_1+\min(c_2,c_3)$.

\medskip
\textbf{Convention on constants.} 
Throughout the paper $N$ denotes a positive constant whose value
may change from line to line; its dependence is always specified in the corresponding statement.

\medskip
We begin with the well-posedness of \eqref{eq:main-SDE}.  
As mentioned in the introduction, this is essentially known, but since none of the available results cover the whole range of exponents and generality of driving processes considered herein, we provide a short proof in the appendix. This is done for the sake of presentation as well as to highlight the usefulness of stochastic sewing for well-posedness of SDEs with jumps; our method for obtaining strong well-posedness is very different from  \cite{Priola1,Priola_flow,Priola2,Chen32}. Another advantage of our proof strategy is that we get ``for free'' that the solution of \eqref{eq:main-SDE} is the limit of Picard's successive approximations. While this fact is known for Lipschitz or essentially Lipschitz drifts  \cite[Theorem~1]{Yam81}, \cite[Theorem~2]{Tan92}, \cite[Lemma~6.2 and Theorem 6.1]{NO22}, to the best of our knowledge, this result is new for H\"older continuous drifts.

Define successively a sequence of approximations $Y^{(0)}(t):=\eta+L_t$, $t\in[0,1]$,
$$
Y^{(n+1)}(t)=\eta+\int_0^t b(Y^{(n)}(s))\,ds+L_t,\quad t\in[0,1],\, n\in\Z_+.
$$ 
\begin{theorem}\label{L:SEU}
	Suppose that $L$ satisfies \ref{A:1}--\ref{A:3}. Let $\eta$ be a $\F_0$--measurable random vector taking values in $\R^d$. Suppose additionally that
	\begin{equation}\label{eq:main-exponent}
		\beta>\Big(1-\frac{\alpha}{2}\Big)\vee\Big(2-2\alpha\Big)
	\end{equation}
	and let $b\in \cC^\beta(\R^d,\R^d)$. 
	Then equation \eqref{eq:main-SDE}  with the initial condition $X_0=\eta$ has a unique strong solution. 	
	Furthermore, this solution $X$ is the limit of the Picard iterations, namely
	\begin{equation*}
		\|X-Y^{(n)}\|_{\C^0([0,1])}\to0\quad \text{as $n\to\infty$\,\, a.s. and in $L_p(\Omega)$, $p\ge1$}. 
	\end{equation*}
\end{theorem}
Now we are ready to present our main results: the $L_p$ convergence of the Euler--Maruyama scheme with an explicit rate. In the statements below, $X$ is a solution to \eqref{eq:main-SDE} with the initial condition $x_0\in \mathbb{R}^d$, and $X^n$, $n\in \mathbb{Z}_+$, is its Euler approximation with the initial condition $x_0^n\in \mathbb{R}^d$, which solves \eqref{eq:EMsde}.

\begin{theorem}\label{T:main} 
	Suppose that $L$ satisfies \ref{A:1}--\ref{A:3} and that \eqref{eq:main-exponent} holds. 
	Let $b\in \cC^\beta(\R^d,\R^d)$, $p>2$, $\eps>0$. Then there exists a constant $N=N(d,\alpha,\beta,p,\eps,M,\|b\|_{\cC^\beta})$ such that for all $n\in\N$ the following bound holds:
	\begin{equation}\label{eq:main-rate}
		\Big\|\|X-X^n\|_{\C^{0}([0,1])}\Big\|_{L_p(\Omega)}\leq N n^{-\big(\frac{1}{2}+\frac{\beta}{\alpha}\wedge\frac12\big)+\eps}+N|x_0-x^n_0|.
	\end{equation}
\end{theorem}

One can also show the almost sure convergence of $X^n$ to $X$. 

\begin{corollary}\label{C:main}  Suppose that $L$ satisfies \ref{A:1}--\ref{A:3} and $\beta$ satisfies \eqref{eq:main-exponent}. Let  $b\in \cC^\beta(\R^d,\R^d)$. Take $x_0^n=x_0$ for all $n\in\N$. Then for any $\eps>0$, 
	there exists an a.s. finite random variable $\eta$ such that for any $n\in\N$, $\omega\in\Omega$
	\begin{equation*}
		\|X(\omega)-X^n(\omega)\|_{\C^{0}([0,1])} \leq \eta(\omega) n^{-\big(\frac{1}{2}+\frac{\beta}{\alpha}\wedge\frac12\big)+\eps}.
	\end{equation*}	
\end{corollary}

The next theorem gives strong $L_p$ rate of convergence of approximations of additive functionals of a L\'evy process.

\begin{theorem}\label{t:quadr}
	Suppose that $L$ satisfies \ref{A:1}--\ref{A:3}. 
	Let $p\in[1,\infty)$, $\eps>0$, $\theta\in[0,1]$, $f\in \C^\theta(\R^d,\R^d)$. Then there exists a constant $N=N(d,\alpha,p,\theta,\eps,M)$ such that for all $n\in\N$ the following bound holds:
	\begin{equation}\label{eq:levy-rate}
		\Bigl\|\sup_{t\in[0,1]}\Bigl|\int_0^t f(L_r)\,dr -\int_0^t f(L_{\kappa_n(r)})\,dr\Bigr|\,\Bigr\|_{L_p(\Omega)}\leq N\|f\|_{\C^\theta} n^{-\big(\frac{1}{2}+\frac{\theta}{\alpha}\wedge\frac12\big)+\eps}.
	\end{equation}	
\end{theorem}	

In the case $p=2$ this rate is proven in \cite[Theorem~11]{Alt}. For general $p$, a worse rate $\frac{1}{2}+\frac{\theta}{2\alpha}\wedge\frac{1}{2}-\eps$ is proved in \cite[Theorem~2.2,~Section 2.4]{GK14}. For special cases of $f$ (indicator function or Dirac-$\delta$) even central limit theorems are available in the literature, see e.g. \cite[Theorem~2.3]{Mark22}, and also \cite{AltG}.

The next theorem concerns approximations of additive functionals of the solution $X$ of an SDE driven by a L\'evy process. In this case the underlying process $X$ can not be simulated directly, and one has to use the corresponding Euler approximation $X^n$. We were not able to find any relevant results in this direction in the literature.

\begin{theorem}\label{t:altm}
	Suppose that $L$ satisfies \ref{A:1}--\ref{A:3} and that \eqref{eq:main-exponent} holds. Let $b\in \cC^\beta(\R^d,\R^d)$, $f\in \cC^\theta(\R^d,\R^d)$, where
	\begin{equation}\label{eq:main-exponentth}
		\theta>\Big(1-\frac{\alpha}{2}\Big)\vee\Big(2-2\alpha\Big).
	\end{equation}
	Let $p\in[1,\infty)$, $\eps>0$. Then there exists a constant $N=N(d,\alpha,\beta,\theta, p,\eps,M,\|b\|_{\cC^\beta})$ such that for all $n\in\N$ the following bound holds:
	\begin{equation}\label{eq:altmeyer-rate}
		\Bigl\|\sup_{t\in[0,1]}\Bigl|\int_0^t f(X_r)\,dr -\int_0^t f(X^n_{\kappa_n(r)})\,dr\Bigr|\,\Bigr\|_{L_p(\Omega)}\leq N\|f\|_{\C^\theta} n^{-\big(\frac{1}{2}+\frac{\beta\wedge\theta}{\alpha}\wedge\frac12\big)+\eps}+N\|f\|_{\C^\theta}|x_0-x^n_0|.
	\end{equation}
\end{theorem}

The proofs of these statements are given in \cref{S:proofs}.

\begin{remark}\label{r:shift}	
	For any $\kappa\in\R^d$, equations  \eqref{eq:main-SDE} and \eqref{eq:EMsde} can be rewritten as:
	$$
	dX_t= \widetilde b(X_t)\,dt +d \widetilde L_t;\qquad
	dX^n_t= \widetilde b(X^n_{\kappa_n(t)})\,dt +d \widetilde L_t,
	$$
	where $\widetilde b(x):=b (x)-\kappa$, $\widetilde L_t:=L_t+\kappa t$. Clearly, if $b\in\C^\beta$, then $\wt b \in\C^\beta$.  Therefore, if for some  $\kappa\in\R^d$  the process $(L_t+\kappa t)_{t\in[0,1]}$ satisfies \ref{A:1}--\ref{A:3}, then \cref{T:main,C:main,t:altm} hold (provided that all the other conditions on $\beta$, $b$, $p$ are met).	
\end{remark}

\begin{remark}\label{rem:explanation}
	\Cref{fig} shows the region where \cref{T:main} guarantees the convergence of the Euler scheme. Let us give some context for some of the different regimes of interest for the exponents.
	\begin{enumerate}[(i)]
		\item In the regime $\alpha\in[2/3,2]$, \eqref{eq:main-exponent} coincides with the well-known condition \eqref{eq:exponents-WP}, and thus \cref{T:main} establishes strong convergence in the optimal range of $\beta$.
		We also remark that the threshold $2/3$ appears in the theory of L\'evy driven SDEs from time to time, e.g., in \cite{Chen32, Menozzi_McKean}. 
		We are unsure whether there is some connection between these appearances, or if this is just an instance of the ``\emph{law of small numbers}''.

		\item The result is new even in the case of smooth drift. In the regime $\beta\geq 1$ the drift is regular enough to make the strong well-posedness of \eqref{eq:main-SDE} trivial. Furthermore, it is easy to get \emph{some} rate of convergence from elementary arguments: assuming for simplicity $x_0=x_0^n$, by Gronwall's lemma one has almost surely
		\begin{equ}
			\sup_{t\in[0,1]}|X_t-X^n_t|\leq e^{[b]_{\cC^1}}\|b\|_{\cC^1}\int_0^1\big(|X^n_t-X^n_{\kappa_n(t)}|\wedge 1\big)\,dt.
		\end{equ}
		This yields for any $p\geq 1$, $\eps>0$,
		\begin{equ}\label{eq:Lipschitz error}
			\big\|\sup_{t\in[0,1]}|X_t-X^n_t|\big\|_{L_p(\Omega)}\leq N n^{-\big(\frac{1}{\alpha}\wedge\frac{1}{p}\big)+\eps},
		\end{equ}
		with some constant $N=N(\alpha,p,\eps,M,\|b\|_{\cC^1})$.
		This provides little control for high moments of the error; further this does not allow to get almost sure rate of convergence. 
		This is markedly improved by \cref{T:main} and \cref{C:main}.

		\item Whenever $\beta\geq\alpha/2$ (which is enforced by \eqref{eq:main-exponent} for all $\alpha\leq 1$), the minimum in $\frac{\beta}{\alpha}\wedge\frac{1}{2}$ is the second term, and so in this case the expression for the $L_p$-rate simplifies to $1-\eps$. Thus in this regime we recover the best possible (up to~$\eps$) rate for an Euler-type approximation.
		
	\end{enumerate}
\end{remark}

\begin{figure}[h]
	\begin{center}
			\begin{tikzpicture}[scale=0.65]
					\foreach \i in {6,6.1,...,12}
					\foreach \j in {0,0.1,...,6}
					{\pgfmathsetmacro{\redd}{0.9*(0.5+min(0.5,\j/\i))}
							\pgfmathsetmacro{\bluee}{1.3-\redd}
							\definecolor{puty}{rgb}{\redd,0.3,\bluee}
							\fill[puty] (\i,\j)--(\i+0.1,\j)--(\i+0.1,\j+0.1)--(\i,\j+0.1);}
					
				\fill[bigpart] (0,13)--(0,12)--(4,4)--(6,3)--(12,6)--(12,13);
				\fill[white] (0,0)--(6,3)--(12,0);
				
				\draw[ggreen, line width=1.2mm,line cap=round] (0,12) -- (4,4) -- (12,0);
				\draw[very thick, dashed] (0,6)--(12,0);
				\draw[dotted] (0,6)--(12,6);
				\draw[thick] (0,0)--(13,0);
				\draw[thick] (0,0)--(0,13);
				\draw (3,0.1)--(3,-0.1);
				\draw (4,0.1)--(4,-0.1);
				\draw (6,0.1)--(6,-0.1);
				\draw (12,0.1)--(12,-0.1);
				\draw (0.1,6)--(-0.1,6);
				\draw (0.1,12)--(-0.1,12);
				\node at (-0.5,12) {\scriptsize $2$};
				\node at (-0.5,6) {\scriptsize $1$};
				\node at (3,-0.5) {\scriptsize $1/2$};
				\node at (4,-0.5) {\scriptsize $2/3$};
				\node at (6,-0.5) {\scriptsize $1$};
				\node at (12,-0.5) {\scriptsize $2$};
				\node at (6,-1.5) {$\alpha$};
				\node at (-1.5,6) {$\beta$};
			\end{tikzpicture}
	\end{center}
	\caption{Convergence rates. \protect\tikz[baseline=-3]\protect
		\draw[ggreen, line width=1.2mm,line cap=round] (0,0)--(0.8,0);
		is the required lower bound for $\beta$;
		\protect\tikz[baseline=-3]\protect\draw[very thick, dashed] (0,0)--(0.8,0);
		is the classical condition \eqref{eq:exponents-WP};
		shading indicates rate of $L_p$-convergence from $1/2$
		(\protect\tikz[baseline=-3]\protect
		\fill[rounded corners, ccorner] (0,-0.2)--(0,0.2)--(0.8,0.2)--(0.8,-0.2)--(0,-0.2)--(0,0.2);)
		to $1$ (\protect\tikz[baseline=-3]\protect
		\fill[rounded corners,bigpart] (0,-0.2)--(0,0.2)--(0.8,0.2)--(0.8,-0.2)--(0,-0.2)--(0,0.2);).
	}\label{fig}
\end{figure}
In the case when $L$ is a standard Brownian motion, \cref{T:main} is consistent with the results of \cite{ButDarGEr} in the case $H=1/2$.

An implementation of the Euler--Maruyama  scheme \eqref{eq:EMsde} would also require simulation of the driving L\'evy process $L$. This is a well--studied problem, and  many techniques and tricks are available in the literature. We do not discuss this problem here but rather refer the reader to a concise collection of methods for simulating $L$  provided in \cite[Section~14]{Papapa}.

\subsection{Examples of L\'evy processes satisfying the main assumptions}\label{Sec:ex}

Let us present some examples of L\'evy processes satisfying \ref{A:1}--\ref{A:3}. 
Denote by $\Phi$ the characteristic exponent (symbol) of $L$, that is, 
\begin{equation*}
\E e^{i \langle \lambda,L_t\rangle}=e^{-t \Phi(\lambda)},\quad \lambda\in\R^d,\,t\ge0.
\end{equation*}
Recall that (see, e.g., \cite[Corollary~2.4.20]{Iphone}) $\Phi$ can be written in the form
$$
\Phi(\lambda)=-i\langle a,\lambda\rangle+\frac12 \langle \lambda, Q\lambda\rangle+\int_{\R^d}(1-e^{i\langle \lambda,y\rangle}+i\langle \lambda,y\rangle\one_{|y|\le1})\,\nu(dy),\quad \lambda\in\R^d,
$$
where $a\in\R^d$, $Q\in\R^{d\times d}$ is a positive semidefinite matrix, and $\nu$ is a $\sigma$-finite measure on $\R^d$ such that $\nu(\{0\}) = 0$ and $\int_{\R^d} (1\wedge|y|^2)\,\nu(dy) <\infty$. It is common to refer to $(a,Q,\nu)$  as the generating triplet of $L$.
We begin with general sufficient conditions on $\Phi$, and then move on to the specific examples. By $\Re\Phi$ we mean the real part of $\Phi$.

\begin{proposition}\label{P:RePhibound}	Assume that for some $\alpha\in(0,2)$, $c_1,c_2,N>0$  the symbol $\Phi$ satisfies 
\begin{equation}\label{phibound}
	c_1|\lambda|^\alpha\le \Re\Phi(\lambda)\le c_2|\lambda|^\alpha,\quad \text{when $|\lambda|>N$}.
\end{equation}
Then the following hold:
\begin{enumerate}[$($i$)$]
	\item if $\alpha\in[1,2)$, then  \ref{A:1} and \ref{A:2} are satisfied for the process $L$;
	\item if $\alpha\in(0,1)$, then $\int_{|y|\le1}|y|\nu(dy)<\infty$ and \ref{A:1} and \ref{A:2} are satisfied for the process $\wt L_t:=L_t+\kappa t$, where $\kappa=-a+\int_{|y|\le1}y\nu(dy)$.
\end{enumerate} 
\end{proposition}
Recall that thanks to \cref{r:shift}
it is sufficient in \cref{T:main} to verify  Assumptions \ref{A:1} and \ref{A:2} for the shifted process $\wt L$. The proof of Proposition \ref{P:RePhibound} is provided in the Appendix. To verify \ref{A:3} the following result can be applied.
\begin{proposition}[{\cite[Theorem~3.1(c)]{Schmoments}}]\label{P:ReH3}	Assume that for some $\alpha\in(0,2)$, $C>0$  the symbol $\Phi$ satisfies 
\begin{equation}\label{phiboundH3}
	|\Phi(\lambda)|\le C|\lambda|^\alpha,\quad \lambda\in\R^d.
\end{equation}
Then for any $p\in(0,\alpha)$, there exists $N=N(\kappa,\alpha,C,d)$ such that
$$
\E|L_t|^p\le N t^{p/\alpha},\quad t\in(0,1].
$$
\end{proposition}
One can also derive an explicit formula for the moments of $L$ in terms of $\Phi$ \cite[p.~3865]{Schmoments}; this is also very useful in verifying \ref{A:3}. Namely, for any $p\in(0,2)$, there exists a constant $N=N(p,d)$ such that for any $t>0$
$$
\E|L_t|^p=N\int_{\R^d}(1-\Re e^{-t\Phi(\lambda)})|\lambda|^{-p-d}\,d\lambda.
$$

Now let us give an extensive list of examples of L\'evy processes satisfying \ref{A:1}--\ref{A:3}. This list is inspired by \cite[p.~425]{Priola1}, \cite[Example~6.2]{Priola2}, \cite[p. 1147]{SSW12} and \cite[Section~4]{Chen32}. All the corresponding proofs are placed in the Appendix.

\begin{example}[General non-degenerate  $\alpha$--stable process, $\alpha\in(0,2)$]\label{E:NDAS}
Take $Q=0$, $a=0$, and
$$
\nu(D) =\int_0^\infty \int_{\mathbb{S}} r^{-1-\alpha}\one_D(r\xi)\mu(d\xi)dr,\quad D\in\mathcal{B}(\R^d),
$$
where $\mu$ is  a finite non-negative measure concentrated on the unit sphere $\mathbb{S}:=\{y\in\R^d: |y|=1\}$ that is non-degenerate, i.e. its support is not contained in a proper linear subspace of $\R^d$.
Then there exists $\kappa\in\R^d$ such that the shifted process $\wt L_t:=L_t+\kappa t$ satisfies \ref{A:1}--\ref{A:3} (and, thus, \cref{T:main} holds by \cref{r:shift}). If $\alpha\in[1,2)$, then $\kappa=0$; if $\alpha\in(0,1)$,
then $\kappa=\int_{|y|\le1}y\nu(dy)<\infty$
\end{example}

\begin{example}[Standard isotropic $d$-dimensional $\alpha$--stable process, $\alpha\in(0,2)$]\label{E:SDAS}
This is a special case of \cref{E:NDAS} with $\mu$ being the uniform measure on $\mathbb{S}$. One can check that in this case  $\nu(D)=c_\alpha \int_D |y|^{-d-\alpha}\,dy$, $D\in\mathcal{B}(\R^d)$ and $\kappa=0$ for all values of $\alpha\in(0,2)$.
\end{example}

\begin{example}[Cylindrical $\alpha$--stable process, $\alpha\in(0,2)$]\label{E:DCAS}
Let $L$ be a $d$-dimensional process whose coordinates are independent standard $1$-dimensional $\alpha$-stable processes.
Then $L$ satisfies \ref{A:1}--\ref{A:3}. 
\end{example}

\begin{example}[{$\alpha$-stable-type process, $\alpha\in(0,2)$, see \cite[p. 1146]{SSW12}}]\label{E:alphasttype}
Take $Q=0$, $a=0$, and
\begin{equation}\label{nuformula}
	\nu(D) =\int_0^\infty \int_{\mathbb{S}} r^{-1-\alpha}\rho(r)\one_D(r\xi)\mu(d\xi)dr,\quad D\in\mathcal{B}(\R^d),
\end{equation}
where $\mu$ is a symmetric (that is $\mu(D)=\mu(-D)$ for any Borel set $D$) non-degenerate finite non-negative measure concentrated on  $\mathbb{S}$ and $\rho\colon(0,\infty)\to\R_+$ is a measurable function such that for some constants $C,C_1,C_2>0$ one has
\begin{equation}\label{rhoprop}
	\one_{[0,C]}(r)\le C_1 \rho(r)\le C_2.
\end{equation}	
Then $L$ satisfies \ref{A:1}--\ref{A:3}.
\end{example}

\begin{example}[$\alpha$-stable tempered process, $\alpha\in(0,2)$]\label{E:Tempered}
This is a special case of \cref{E:alphasttype} with $\rho(r)= e^{- c r}$ with some $c>0$.
\end{example}

\begin{example}[Truncated $\alpha$-stable  process, $\alpha\in(0,2)$]\label{E:Truncated}
This is a special case of \cref{E:alphasttype} with $\rho_r=c\one_{[0,1]}$ with some $c>0$ and $\mu$ being the uniform measure on $\mathbb{S}$. One can check that in this case $\nu(D)=c_{\alpha}\int_{D\cap \{y:|y|\le 1\}} |y|^{-d-\alpha}\,dy$, $D\in\mathcal{B}(\R^d)$
\end{example}

\begin{example}[{Relativistic  $\alpha$-stable process, $\alpha\in(0,2)$, \cite{Carmona, chakraborty2020relativistic}}]\label{E:relativistic}
Take $L$ with symbol $\Phi(\lambda)=(|\lambda|^2+C^{2/\alpha})^{\alpha/2}-C$, with some parameter $C>0$. It satisfies \ref{A:1}--\ref{A:3}.
\end{example}

\begin{example}[Brownian motion]\label{E:BM}
If $L$ is the standard $d$--dimensional Brownian motion, then it satisfies \ref{A:1}--\ref{A:3} with $\alpha=2$.
\end{example}

\begin{example}[Linear combinations]\label{E:sum}
\begin{enumerate}[(i)]  
	\item Let $L^{(1)}$, $L^{(2)}$ be two independent L\'evy processes. Assume that the process $L^{(i)}$, $i=1,2$, satisfies  \ref{A:1}--\ref{A:3} with $\alpha=\alpha_i$. Then the process $L^{(1)}+L^{(2)}$ satisfies  \ref{A:1}--\ref{A:3} with $\alpha=\alpha_1\vee\alpha_2$.
	\item In particular, if $L$ is a sum of a $d$--dimensional Brownian motion and the standard $\alpha$--stable process, then the rate in \eqref{eq:main-rate} is $n^{-(\frac12+\frac\beta2\wedge\frac12)+\eps}$ for $b\in\C^\beta$, $\beta>0$. 
	\item 
	In general, if $L$ has a non-degenerate diffusion part (that is, $Q$ is positive definite) and for some $\gamma>0$ we have $\int_{\{|y|>1\}}|y|^\gamma \nu (dy)<\infty$, then  the
	statement of \cref{T:main} is valid with   $\alpha=2$. This corresponds to the rate  $n^{-(\frac12+\frac\beta2\wedge\frac12)+\eps}$. This significantly improves \cite[Remark~2.3]{KSEM2019}.
\end{enumerate}
\end{example}

\begin{remark}
Let $d=1$ and take $Q=0$, $\nu(D)=c_\alpha\int_D y^{-1-\alpha}\one(y>0)\,dy$, $D\in\mathcal{B}(\R)$, $\alpha\in(0,1)$, ${a=\int_{|y|\le1} y\nu(dy)<\infty}$. Then the process $L$ satisfies \ref{A:1}--\ref{A:3} (since it is a special case of \eqref{E:NDAS}) and is increasing (\cite[Theorem~1.3.15]{Iphone}). Thus, regularization by noise can occur for monotone drivers as well.
\end{remark}

Finally, let us present a simple example of a class of L\'evy processes for which regularization by noise cannot occur. Suppose that $a=0$, $Q=0$, $\nu(\R^d)<\infty$. Then, the corresponding process $L$ is a pure jump process and it has only finitely many jumps on the interval $[0,1]$ \cite[Theorem~21.3]{Sato-san}. Denoting its first jumping time by $T(\omega)$, we see that on the time interval $[0,T(\omega)]$, equation \eqref{eq:main-SDE} becomes $dX_t=b(X_t) dt$. Clearly, if $b\in\C^\beta$, $\beta<1$, this equation might have infinitely many or no solutions.  Thus, for such L\'evy noises the original equation \eqref{eq:main-SDE} is not well-posed.

\section{Preliminaries}\label{S:Prel}

Before we proceed to the proofs of our main results, let us collect a number of useful methods and bounds  which we are going to apply later. 

For a random variable $\xi$, a sub-$\sigma$-algebra  $\mathcal{G}\subset\F$, and $p\ge1$  we introduce the quantity
\begin{equation}\label{condnorm}
\|\xi\|_{L_p(\Omega)|\mathcal{G}}:=\left(\E[|\xi|^p|\mathcal{G}]\right)^{\frac1p},
\end{equation}
which is a $\mathcal{G}$-measurable non-negative random variable. 
It is clear that
\begin{equation*} 
\|\xi\|_{L_p(\Omega)}=\|\|\xi\|_{L_p(\Omega)|\mathcal{ G}}\|_{L_p(\Omega)}.
\end{equation*}
Note that if  $p\ge1$, $\mathcal{G}\subset \mathcal{H}$ are $\sigma$-algebras, then the following simple bounds hold almost surely 
\begin{equation}\label{standardbound}
\|\E [\xi|\mathcal{H}]\|_{L_p(\Omega)}\le \|\xi\|_{L_p(\Omega)};\qquad 
\|\E [\xi|\mathcal{H}]\|_{L_p(\Omega)|\mathcal{G}}\le \|\xi\|_{L_p(\Omega)|\mathcal{G}}.
\end{equation}
These quantities are not norms, but rather $\cG$-measurable nonnegative random variables. To simplify the presentation, any inequality between such expressions is understood in the almost sure sense. 

\subsection{Conditional shifted stochastic sewing lemma}
An important tool to obtain \cref{T:main} is an adjusted version of \cite[Lemma 2.2]{Mate20}, which in turn is based on L\^e's stochastic sewing lemma \cite{Khoa}.
We need the following notation. For $0\le S\le T$ we denote a modified simplex 
\begin{equation}\label{ST21}
\Delta_{[S,T]}:=\{(s, t) : S \le  s < t \le T, s - (t - s) \ge S\}.
\end{equation}
For a function  $f\colon \Delta_{[S,T]}\to\R^d$ and a  triplet of times $(s,u,t)$ such that $S\le s\le u\le t\le T$, we denote
\begin{equation*}
\delta f_{s,u,t}:=f_{s,t}-f_{s,u}-f_{u,t}.
\end{equation*}
The conditional expectation given $\F_s$ is denoted by $\E^s$.

\begin{lemma}\label{lem:shiftedmodifiedSSL}
Let $0\leq S<T\leq1$, $p\in[2,\infty)$ and let $(A_{s,t})_{(s,t)\in\Delta_{[S,T]}}$ be a family of random variables in $L_p(\Omega,\R^d)$ such that $A_{s,t}$ is $\F_t$-measurable. Let $\G\subset \F_S$ be a $\sigma$-algebra. 	Suppose that for some $\eps_1,\eps_2,\eps_3>0$ and $\G$-measurable random variables $\Gamma_1, \Gamma_2, \Gamma_3\ge0$ the bounds
\begin{align}
	&\|A_{s,t}\|_{L_p(\Omega)|\G}  \leq \Gamma_1|t-s|^{1/2+\eps_1}\,,\label{SSL1}
	\\
	&\|\E^{s-(t-s)}\delta A_{s,u,t}\|_{L_p(\Omega)|\G}  \leq \Gamma_2 |t-s|^{1+\eps_2}+\Gamma_3 |t-s|^{1+\eps_3}\label{SSL2}
\end{align}
hold for all $(s,t)\in\Delta_{[S,T]}$ and $u=(s+t)/2$.

Further, suppose that there exists a process $\A=\{\A_t:t\in[S,T]\}$ such that for any $S\le s\le t\le T$  one has
\begin{equation}\label{limpart}
	\A_t-\A_s=\lim_{m\to\infty} \sum_{i=1}^{m-1} A_{s+i\frac{t-s}m,s+(i+1)\frac{t-s}m}\,\, \text{in probability.}
\end{equation}
Then there exist deterministic constants $K_1,K_2>0$, which depend only on $\eps_1,\eps_2$, $\eps_3$, $p$, and $d$ such that for any $S\le s\le t\le T$ we have
\begin{equation}\label{SSL3 cA}
	\|\A_t-\A_s\|_{L_p(\Omega)|\G}  \leq  K_1\Gamma_1 |t-s|^{1/2+\eps_1}+K_2\Gamma_2 |t-s|^{1+\eps_2}+K_3\Gamma_3 |t-s|^{1+\eps_3}.
\end{equation}
\end{lemma}

The proof of \cref{lem:shiftedmodifiedSSL} is given in the appendix.

\subsection{Weighted John-Nirenberg inequality}

\begin{proposition}\label{prop:vmo_john_nirenberg}
Let $0\le S\le T$ and let $\A\colon\Omega\times[S,T]\to\R^d$, $\xi\colon\Omega\times[S,T]\to\R_+$ be stochastic processes  adapted to the filtration $(\F_t)_{t\in[S,T]}$.  Assume additionally that $\A$ is continuous and $\cA_t\in L_1(\Omega)$ for all $t\in[S,T]$.
Suppose that for any $S\le s\le t\le T$ one has
\begin{equation}\label{eq:vmo_gamma}
	\E^s  |\A_t-\A_s|     \le \xi_s \quad a.s.
\end{equation}
Then for any $p\ge1$ there exists a constant $N=N(p)$ such that 
for any $S\le s\le t\le T$
\begin{equation}\label{jnres}
	\|\sup_{r\in[s,t]}|\A_r-\A_s|\,\|_{L_p(\Omega)|\F_s}\le N	\|\sup_{r\in[s,t]}\xi_r\,\|_{L_p(\Omega)|\F_s}.
\end{equation}
\end{proposition}

The above proposition is very close to \cite[Theorem 1.3]{Le22}, and the only difference is that the condition \eqref{eq:vmo_gamma} is imposed there for all stopping times $s(\omega), t(\omega)\in[S,T]$. On the other hand, we only assume that \eqref{eq:vmo_gamma} holds for deterministic $s, t\in[S,T]$; this  will be crucial later in the proofs of \cref{T:main,t:quadr,t:altm}. For the case when $\xi$ is a constant,  \cite[Proposition 2.2]{Le22b} shows that if \eqref{eq:vmo_gamma} is satisfied for deterministc $s,t$, then it is also satisfied for $s,t$ being stropping times. In general case, this seems to be not true. Therefore, for the proof of \cref{prop:vmo_john_nirenberg}, we adapt the argument from \cite[Theorem 1.3]{Le22}. This is done in the appendix.

\subsection{Heat kernel and related bounds}
Recall that $\cP$ is the Markov transition semigroup associated to the process $L$.
Quite often we use the following simple observation. For any measurable bounded function $f\colon\R^d\to\R^d$, $0\leq s\leq t$, and any $\F_s$-measurable random vector $\xi$ one has
\begin{equation}\label{very-very-basic}
\E^s f(L_t+\xi)=\cP_{t-s} f(L_s+\xi).
\end{equation}
We now formulate some consequences of \ref{A:1}-\ref{A:3} in the form that they are actually used in the proofs.
\begin{proposition}\label{Prop:HKbounds}
Suppose that \ref{A:1} and \ref{A:2} hold. Then for every $\eps>0$, $\beta\ge0$, $\rho$ as below, $\mu\in(\eps,1+\eps]$, $\mu\ge(\beta-\rho)/\alpha$
there exists  a constant $N=N(\alpha,\beta,\mu,\eps,\rho)$ such that the following holds for any $f\in\C^\beta$:
\begin{align}
	&\|\cP_t f\|_{\cC^\rho}\leq N \|f\|_{\cC^\beta}t^{\frac{(\beta-\rho)\wedge0}\alpha}, \quad 0<t\le1,\,\,\rho\ge0;\label{eq:K2A}\\
	&\|\cP_t f -\cP_s f\|_{\cC^\rho}\le N \| f\|_{\C^\beta} s^{\frac{\beta-\rho}\alpha-\mu} (t-s)^{\mu-\eps}, \quad 0\le s<t\le1,\,\, \rho=0,1.\label{eq:K1A}
\end{align}
\end{proposition}
The proof of this proposition is mostly technical and is provided in the appendix. 

Assuming that $L$ satisfies \ref{A:3}, it is immediate to see that this implies that for any 
$m>0$, $p\ge1$, $\eps>0$ there exists a constant $N=N(p,m,\eps,\alpha)$ such that 
\begin{equation}\label{eq:basic moments}
\| |L_t|^m \wedge 1\|_{L_p(\Omega)}\le 
N t^{(\frac m\alpha\wedge\frac{1}{p})-\eps},\quad t\in[0,1].
\end{equation}

We further recall two elementary inequalities:
\begin{align}
&|f(x_1)-f(x_2)|\leq |x_1-x_2|\|f\|_{\C^1}\label{eq:C1}\\
&|f(x_1)-f(x_2)-f(x_3)+f(x_4)|\leq |x_1-x_2-x_3+x_4|\|f\|_{\C^1}+|x_1-x_2||x_1-x_3|\|f\|_{\C^2}\label{eq:C2}
\end{align}
for any $x_1,x_2,x_3,x_4\in\R^d$ and any $f$ from $\C^1$ or $\C^2$, respectively.

\section{Proof of the main results}\label{S:proofs}
We will denote 
\begin{equation}\label{phiphin}
\varphi:=X-L,\quad  \varphi^n=X^n-L,\qquad n\in\Z_+.
\end{equation}
We consider the following decomposition of the difference between the additive functional of the process and its estimate.  For $0\le s\le t \le 1$, $n\in\Z_+$, $f\in\C^\beta$ we write
\begin{align}
\int_s^t f(X_r)\,dr-\int_s^t f(X^n_{\kappa_n(r)})\,dr
&=\int_s^t f(L_r+\varphi_r)-f(L_r+\varphi^n_r)\,dr
\nn\\
&\phantom{=}+\int_s^t f(L_r+\varphi^n_r)-f(L_r+\varphi^n_{\kappa_n(r)})\,dr
\nn\\
&\phantom{=}+\int_s^t f(L_r+\varphi^n_{\kappa_n(r)})-f(L_{\kappa_n(r)}+\varphi^n_{\kappa_n(r)})\,dr
\nn\\
&=:\cE^{f,n,1}_{s,t}+\cE^{f,n,2}_{s,t}+\cE^{f,n,3}_{s,t}.\label{eq:error-decomposition}
\end{align}
Clearly, for $f=b$ we get the increment of the difference between the process and its Euler approximaiton:
\begin{equation}\label{errordec2}
(X_t-X^n_t)-(X_s-X^n_s)=\int_s^t b(X_r)\,dr-\int_s^t b(X^n_{\kappa_n(r)})\,dr=
\cE^{b,n,1}_{s,t}+\cE^{b,n,2}_{s,t}+\cE^{b,n,3}_{s,t}.
\end{equation}

\begin{remark}
This decomposition differs from the one in e.g. \cite{ButDarGEr}: therein, $\cE^{b,n,2}+\cE^{b,n,3}$ can be treated as one term, and by Girsanov's theorem, the perturbation $\varphi^n$ can in fact be transformed away. Such trick is not available in the L\'evy case due to the lack of an appropriate Girsanov's theorem.
\end{remark}

Our  goal is to bound the $L_p(\Omega)$ norm of the left-hand side of \eqref{errordec2}.
However, doing this directly for the term $\cE^{b,n,3}_{s,t}$ would lead to a rate that deteriorates for large $p$, see \cref{r:restr} below.
Therefore, we bound a conditional $L^2$-``norms'' instead and eventually after buckling apply John-Nirenberg inequality.
First, in \cref{sec:apriori} we prove certain a priori bounds for $\varphi,\varphi^n$. Then in \cref{sec:E12} we produce general bounds for conditional ``norms'' of integrals of irregular functions along the perturbed L\'evy process, which are then applied to bound $\cE_{s,t}^{f,n,i}$ in \cref{sec:E3}.
Finally, these bounds are combined for the proofs of the main theorems in \cref{sec:mainproof}.

\subsection{A priori bounds}\label{sec:apriori}

Let $f\colon[0,1]\times\Omega\to\R^d$ be a measurable bounded function adapted to the filtration $\bF$. Let $\gamma\in(0,1]$, $p\ge2$, $0\le S\le T\le1$, let $\G\subset\F_S$ be a $\sigma$-algebra. We consider the following quantities of $f$: \begin{align}
&\db{f}_{\scC^0_p|\G,[S,T]}:=\sup_{t\in[S,T]}\|f(t)\|_{L_p(\Omega)|\G};\nn\\
&[f]_{\scC^\gamma_p|\G,[S,T]}:=\sup_{s,t\in[S,T]}\frac{\|f(t)-f(s)\|_{L_p(\Omega)|\G}}{| t-s|^\gamma};\label{newnormH}\\
&\ddnew{f}{\gamma}{p}{[S,T]}:=\sup_{s,t\in[S,T]}\frac{\big\|\|f(t)-\E^s f(t)\|_{L_1(\Omega)|\cF_s}\big\|_{L_p(\Omega)|\G}}{| t-s|^\gamma}.\label{newnorm}
\end{align}

Part (iii) of the following lemma is the main a priori estimate on the ``stochastic regularity'' of $\varphi$, $\varphi^n$.
\begin{lemma}\label{lem:useful-lemma}
\begin{enumerate}[{\rm(i)}]
	\item Let $q\ge1$. Let $\cG\subset \F$ be a $\sigma$-algebra. Let $Y,Z\in L_q(\Omega)$ be random variables and suppose that $Z$ in $\cG$--measurable. Then  
	\begin{equation}\label{YZcond}
		\|Y-\E [Y|\cG]\|_{L_q(\Omega)|\cG}	\le 2\|Y-Z\|_{L_q(\Omega)|\cG},\quad\text{a.s.}
	\end{equation}
	\item For any $0\le s\le t\le 1$, $q\ge1$, measurable function $f\colon[s,t]\times\Omega\to\R^d$ adapted to the filtration $\bF$, $\gamma\in(0,1)$, $\sigma$-algebra $\G\subset\F_s$ one has
	\begin{equation}\label{normbound}
		\ddnew{f}{\gamma}{q}{[s,t]}\le2[f]_{\scC^{\gamma}_q|\G,[s,t]}.
	\end{equation}
	\item Let $q\ge1$,  $\eps>0$. Assume that $\beta>1-\alpha$ and that \ref{A:3} holds. There exists a constant $N=N(\beta,\|b\|_{\cC^\beta},\alpha,\eps, q,M)$ such that for $0\leq s\leq t\leq 1$ one has a.s.
	\begin{align}\label{eq:useful-bound}
		&\|\varphi_t-\E^s\varphi_t\|_{L_q(\Omega)|\F_s}\leq N  |t-s|^{1+\big(\frac{\beta\wedge1}\alpha\wedge\frac{1}{q}\big)-\eps};\\
		\label{eq:useful-bound-n}
		&\|\varphi_t^n-\E^s\varphi_t^n\|_{L_q(\Omega)|\F_s}\leq N  |t-s|^{1+\big(
			\frac{\beta\wedge1}\alpha\wedge\frac{1}{ q}\big)-\eps}.
	\end{align}
\end{enumerate}
\end{lemma}
\begin{proof}
(i) We have 
\begin{align*}
	\|Y- \E [Y|\cG]\|_{L_q(\Omega)|\cG}	&\le \|Y-Z\|_{L_q(\Omega)|\cG}
	+\|\E [Y|\cG]-Z\|_{L_q(\Omega)|\cG}\nn\\
	&= \|Y-Z\|_{L_q(\Omega)|\cG}+\|\E [Y-Z|\cG]\|_{L_q(\Omega)|\cG}\nn\\
	&\le 2\|Y-Z\|_{L_q(\Omega)|\cG},
\end{align*}
where the last inequality follows from \eqref{standardbound}. 

(ii) By part (i) of the lemma and Jensen's inequality we have for any $s\le s'\le t'\le t$
\begin{equation*}
	\big\|\|f(t')-\E^{s'} f(t')\|_{L_1(\Omega)|\cF_{s'}}\big\|_{L_q(\Omega)|\G}\le 2
	\big\|\|f(t')- f(s')\|_{L_1(\Omega)|\cF_{s'}}\big\|_{L_q(\Omega)|\G}\le 2
	\|f(t')- f(s')\|_{L_q(\Omega)|\G},
\end{equation*}
where we used that $\G\subset \F_s\subset \F_{s'}$. The desired result follows now from the definitions of the seminorms \eqref{newnormH} and \eqref{newnorm}.

(iii) Without loss of generality, we can assume that $\beta\le1$. 
Suppose that \eqref{eq:useful-bound} holds for some $m\geq 0$ in place of ${1+(
	\frac{\beta\wedge1}\alpha\wedge\frac1{q})-\eps}$. This is certainly true for $m=0$ thanks to the fact that $b$ is bounded; we proceed now by induction on $m$.
We apply \eqref{YZcond} with $\cG=\F_s$, $Y=\phi_t$, $Z=\phi_s+\int_s^t b(L_s+\E^s\varphi_r)\,dr$. We get
\begin{align*}
	\|\varphi_t-\E^s\varphi_t\|_{L_q(\Omega)|\F_s}&\leq 2\Big\|\varphi_t-\varphi_s-\int_s^t b(L_s+\E^s\varphi_r)\,dr\Big\|_{L_q(\Omega)|\F_s}
	\\
	&=2\Big\|\int_s^t\big( b(L_r+\varphi_r)-b(L_s+\E^s\varphi_r)\big)\,dr\Big\|_{L_q(\Omega)|\F_s}
	\\
	&\le N \Big\|\int_s^t \big(|L_r-L_s|^{\beta}+|\varphi_r-\E^s\varphi_r|^{\beta}\big)\wedge1\,dr\Big\|_{L_q(\Omega)|\F_s}\\
	&\le N\int_s^t \big(\| |L_r-L_s|^{\beta} \wedge1\|_{L_{ q}(\Omega)|\F_s}+\|\varphi_r-\E^s\varphi_r\|^{\beta}_{L_{ q}(\Omega)|\F_s}\big)\,dr.
\end{align*}
Using \eqref{eq:basic moments}, the independence of $L_r-L_s$ from $\F_s$, and the induction hypothesis, we get a.s.
\begin{equation*}
	\|\varphi_t-\E^s\varphi_t\|_{L_q(\Omega)|\F_s}\leq N|t-s|^{1+(\frac\beta\alpha\wedge\frac{1}{ q}\wedge \beta m)-\eps}.
\end{equation*}
It is elementary to see that if $\eps>0$ is small enough, then the recursion $m_0=0$, $m_{i+1}=1+(\frac\beta\alpha\wedge\frac{1}{ q}\wedge \beta m_i)-\eps$ reaches $1+(\frac\beta\alpha\wedge\frac{1}{q})-\eps$ in finitely many steps (recall that  $\alpha>1-\beta$ and thus $\alpha>(1-\beta)/(1-\eps)$ for small enough $\eps>0$). Recalling our initial assumption $\beta\le1$, we get  \eqref{eq:useful-bound}. 

Inequality \eqref{eq:useful-bound-n} is obtained by a similar argument, though one has to be a bit more careful because now $L_{\kappa_n(r)}-L_{\kappa_n(s)}$ is not independent of $\F_s$. For fixed $s\in[0,1]$, define $s'$ to be the smallest  grid point which is bigger or equal to $s$, that is, $s':=\lceil ns \rceil n^{-1}$. It is crucial to note that $\phi^n_{s'}$ is $\F_s$ measurable.

We proceed by induction as before and assume that \eqref{eq:useful-bound-n} holds for some $m\geq 0$.
If $s\le t <s'$, then $\phi_t^n$ is $\F_s$--measurable. Hence $\varphi_t^n=\E^s\varphi_t^n$ and the left--hand side of \eqref{eq:useful-bound-n} is zero. Therefore it remains to consider the case $t\ge s'$. In this case, using again \eqref{YZcond} with $\cG=\F_s$, $Y=\phi^n_t$, $Z=\phi^n_{s'}+\int_{s'}^t b(L_s+\E^s\phi_{\kappa_n(r)}^n)\,dr$, we deduce
\begin{align*}
	\|\varphi_t^n-\E^s\varphi_t^n\|_{L_q(\Omega)|\F_s}&\le 2\Big\|\varphi_t^n-\varphi^n_{s'}-\int_{s'}^t b(L_s+\E^s\phi_{\kappa_n(r)}^n)\,dr\Big\|_{L_q(\Omega)|\F_s}
	\\
	&=2\Big\|\int_{s'}^t \big(b(L_{\kappa_n(r)}+\varphi^n_{\kappa_n(r)})-b(L_{s}+\E^s\phi^n_{\kappa_n(r)})\big)\,dr\Big\|_{L_q(\Omega)|\F_s}
	\\
	&\le N\Big\|\int_{s'}^t \big(|L_{\kappa_n(r)}-L_{s}|^\beta+|\varphi^n_{\kappa_n(r)}-\E^s\varphi^n_{\kappa_n(r)}|^\beta\big)\wedge1\,dr\Big\|_{L_q(\Omega)|\F_s}.
\end{align*}
Note that for $r\geq s'$, $\kappa_n(r)\geq s'\geq s$, and therefore $L_{\kappa_n(r)}-L_s$ is independent of $\cF_s$. From here we obtain \eqref{eq:useful-bound-n} exactly as before.
\end{proof}
\begin{remark}
The reason for the non-standard portion of our main assumption \eqref{eq:main-exponent} (condition $\beta>2-2\alpha$) and the strange threshold $2/3$ in \cref{rem:explanation} is the appearance of  $1/q$ in \eqref{eq:useful-bound}-\eqref{eq:useful-bound-n}.
\end{remark}

\subsection{General bounds}\label{sec:E12}

Since  $\cE^{f,n,1}$ and $\cE^{f,n,2}$ in decomposition \eqref{eq:error-decomposition} have similar forms (difference of averages of $f$ along $L$ with two different perturbations), we begin with the following  general bound that can be applied to both. 

\begin{lemma}\label{L:32} Let $p\in[2,\infty)$.  Assume \ref{A:1}-\ref{A:3}. Let	 $\tau,\gamma,\eps_0\in(0,1]$, $\theta>0$, be constants satisfying 
\begin{equation}\label{taugamma}
	\theta\in(1-\frac\alpha2,2],\quad \frac{\theta-2}{\alpha}+\tau>0,\qquad
	\gamma+\frac{\theta-1}{\alpha}>0.
\end{equation}
Let $f\in\C^\theta$. Let $g,h\colon[0,1]\times\Omega\to\R^d$ be bounded, adapted, measurable functions and suppose that there exist $C_g,C_h>0$ such that  for any $0\le s \le t \le 1$ one has
a.s.
\begin{align}\label{a:g}
	&\E^{s}|g_t-\E^{s}g_t|\le C_g |t-s|^{\tau},\\
	&\|h_t-\E^{s}h_t\|_{L_1(\Omega)}\le C_h |t-s|^{\eps_0}\label{a:h}.
\end{align} 
Then there exists a constant $N=N(\alpha,\theta,p,d,\gamma,\tau)$ such that for any $0\le S \le T \le 1$ and any $\sigma$-algebra $\G\subset\F_S$ one has the bound
\begin{align}\label{L32bound}
	\Bigl\|\int_S^T f(L_r+g_r)-f(L_{r}+h_r)\,dr\Bigr\|_{L_p(\Omega)|\G}  &\le  N\|f\|_{\C^\theta}(T-S)^{1+\frac{(\theta-1)\wedge0}{\alpha}}\db{g-h}_{\scC^0_p|\G,[S,T]}\nn\\
	&\phantom{\le}+ N\|f\|_{\C^\theta} (T-S)^{\gamma+\frac{(\theta-1)\wedge0}{\alpha}+1}\ddnew{g-h}{\gamma}{p}{[S,T]}\nn\\
	&\phantom{\le}+ N\|f\|_{\C^\theta} C_g (T-S)^{1+\frac{\theta-2}{\alpha}+\tau}\db{g-h}_{\scC^0_p|\G,[S,T]}.
\end{align}
\end{lemma}
\begin{remark}\label{rem:spoiler}
It is pivotal that the seminorm appearing in the right--hand side of \eqref{L32bound} is $\ddnew{g-h}{\gamma}{p}{[S,T]}$ rather than a much less precise seminorm  $[g-h]_{\scC^\gamma_p|\G,[S,T]}$ (recall \cref{lem:useful-lemma}(ii) and the definitions of the seminorms in \eqref{newnorm} and \eqref{newnormH}). This will be crucial for bounding $\cE^{f,n,2}$, see \cref{rem:WhyImportant} below.
\end{remark}

\begin{proof}
Fix $0\le S\le T\le 1$. Put 
\begin{align*}
	&A_{s,t}:=\E^{s-(t-s)}\int_s^t f(L_r+\E^{s-(t-s)}g_r)-f(L_{r}+\E^{s-(t-s)}h_r)\,dr,\qquad(s,t)\in \Delta_{[S,T]};\\
	&\A_{t}:=\int_0^t f(L_r+g_r)-f(L_{r}+h_r)\,dr,\qquad t\in[S,T].
\end{align*}
Let us verify that the processes $A$, $\A$ satisfy all the conditions of the stochastic sewing lemma (\cref{lem:shiftedmodifiedSSL}).

Let $(s,t)\in\Delta_{[S,T]}$. Then recalling \eqref{very-very-basic} and \eqref{eq:C1}, we see that
\begin{align*}
	|A_{s,t}|&\le \int_s^t |\cP_{r-(s-(t-s))} f(L_{s-(t-s)}+\E^{s-(t-s)}g_r)-\cP_{r-(s-(t-s))} f(L_{s-(t-s)}+\E^{s-(t-s)}h_r)|\,dr\\
	& \le\int_s^t \|\cP_{r-(s-(t-s))} f\|_{\C^1}|\E^{s-(t-s)}(g_r-h_r)|\,dr.
\end{align*}
Thus, by \eqref{eq:K2A} (applied with $\rho=1$ and $\beta=\theta$) and \eqref{standardbound}, we have
\begin{align*}
	\|A_{s,t}\|_{L_p(\Omega)|\G}& \le\|f\|_{\C^\theta}\int_s^t (r-s)^{\frac{(\theta-1)\wedge0}{\alpha}}\|\E^{s-(t-s)}(g_r-h_r)\|_{L_p(\Omega)|\G}\,dr\\
	& \le N\|f\|_{\C^\theta} (t-s)^{1+\frac{(\theta-1)\wedge0}{\alpha}} \sup_{r\in[S,T]}\|\E^{s-(t-s)}(g_r-h_r)\|_{L_p(\Omega)|\G}\\
	& \le N\|f\|_{\C^\theta} (t-s)^{1+\frac{(\theta-1)\wedge0}{\alpha}} \sup_{r\in[S,T]}\|g_r-h_r\|_{L_p(\Omega)|\G}\\
	& = N\|f\|_{\C^\theta} (t-s)^{1+\frac{(\theta-1)\wedge0}{\alpha}}\db{g-h}_{\scC^0_p|\G,[S,T]}.
\end{align*}
Here in the penultimate line we used that $\G\subset\F_S\subset\F_{s-(t-s)}$.
Note that by the assumption $\theta>1-\alpha/2$, we have $1+\frac{\theta-1}{\alpha}>1/2$. Therefore, condition \eqref{SSL1} is satisfied with $\Gamma_1= N\|f\|_{\C^\theta}\db{g-h}_{\scC^0_p|\G,[S,T]}$.

Now let us verify condition \eqref{SSL2}. As required, we take $(s,t)\in\Delta_{[S,T]}$, and $u:=(t+s)/2$. It will be convenient to denote $s_1:=s-(t-s)$, $s_2:=s-(u-s)$, $s_3:=s$, $s_4:=u$, $s_5:=t$. One has $s_1\le s_2\le s_3\le s_4\le s_5$.
Then we deduce
\begin{align}\label{step1}
	&\E^{s-(t-s)}\delta A_{s,u,t}\nn\\
	&\quad=\E^{s_1}\delta A_{s_3,s_4,s_5}\nn\\
	&\quad=\E^{s_1}\int_{s_3}^{s_4}(f(L_r+\E^{s_1}g_r)-f(L_{r}+\E^{s_1}h_r)-f(L_r+\E^{s_2}g_r)-f(L_{r}+\E^{s_2}h_r))\,dr\nn\\
	&\qquad+\E^{s_1} \int_{s_4}^{s_5}(f(L_r+\E^{s_1}g_r)-f(L_{r}+\E^{s_1}h_r)-f(L_r+\E^{s_3}g_r)-f(L_{r}+\E^{s_3}h_r))\,dr\nn\\
	&\quad=:I_1+I_2.
\end{align}
Here in the term $I_2$ we used the identity $u-(t-s)=s=s_3$.
We begin with the analysis of $I_1$. Recalling \eqref{very-very-basic}, we obviously have
\begin{align*}
	&I_1=\E^{s_1}\E^{s_2} \int_{s_3}^{s_4}(f(L_r+\E^{s_1}g_r)-f(L_{r}+\E^{s_1}h_r)-f(L_r+\E^{s_2}g_r)-f(L_{r}+\E^{s_2}h_r))\,dr\nn\\
	&\quad=\E^{s_1}\int_{s_3}^{s_4}\big(\cP_{r-s_2}f(L_{s_2}+\E^{s_1}g_r)-\cP_{r-s_2}f(L_{s_2}+\E^{s_1}h_r)\\
	&\hskip13ex-\cP_{r-s_2}f(L_{s_2}+\E^{s_2}g_r)-
	\cP_{r-s_2}f(L_{s_2}+\E^{s_2}h_r)\big)\,dr.
\end{align*}
Applying \eqref{eq:C2} and  \eqref{eq:K2A} we see that
\begin{align}\label{step2}
	|I_1|&\le \|f\|_{\C^\theta}\int_{s_3}^{s_4}(r-s_2)^{\frac{(\theta-1)\wedge0}{\alpha}}\E^{s_1}|\E^{s_1}(g_r-h_r)-\E^{s_2}(g_r-h_r)|\,dr\nn\\
	&\quad+ \|f\|_{\C^\theta}\int_{s_3}^{s_4}(r-s_2)^{\frac{\theta-2}{\alpha}}|\E^{s_1}(g_r-h_r)|\,\E^{s_1}|\E^{s_2}g_r-\E^{s_1}g_r|\,dr.
\end{align}
Using conditional Jensen's inequality and the assumption \eqref{a:g}, we see that a.s.
\begin{equation}\label{step21}
	\E^{s_1}|\E^{s_2}g_r-\E^{s_1}g_r|=
	\E^{s_1}|\E^{s_2}(g_r-\E^{s_1}g_r)|\le 
	\E^{s_1}|g_r-\E^{s_1}g_r|\le C_g |r-s_1|^{\tau}.
\end{equation} 
Similarly,
\begin{align}\label{step22}
	\E^{s_1}|\E^{s_2}(g_r-h_r)-\E^{s_1}(g_r-h_r)|&=\E^{s_1}|\E^{s_2}\big((g_r-h_r)-\E^{s_1}(g_r-h_r)\big)|\nn\\
	&\le \E^{s_1}|(g_r-h_r)-\E^{s_1}(g_r-h_r)|.
\end{align}
Combining 
\eqref{step2}, \eqref{step21}, \eqref{step22}, and using the Minkowski inequality together with \eqref{standardbound} and the fact that $\G\subset\F_{s_1}$, we finally get
\begin{align}\label{I1fin}
	\|I_1\|_{L_p(\Omega)|\G}&\le N\|f\|_{\C^\theta} (s_4-s_1)^{\gamma+\frac{(\theta-1)\wedge0}{\alpha}+1}\ddnew{g-h}{\gamma}{p}{[S,T]}\nn\\
	&\phantom{\le}+N\|f\|_{\C^\theta}C_g(s_3-s_2)^{\frac{\theta-2}{\alpha}}(s_4-s_3)(s_4-s_1)^{\tau}
	\db{g-h}_{\scC^0_p|\G,[S,T]}\nn\\
	&\le N\|f\|_{\C^\theta} (t-s)^{\gamma+\frac{(\theta-1)\wedge0}{\alpha}+1}\ddnew{g-h}{\gamma}{p}{[S,T]}\nn\\
	&\phantom{\le}+N\|f\|_{\C^\theta}C_g(t-s)^{\frac{\theta-2}{\alpha}+1+\tau}\db{g-h}_{\scC^0_p|\G,[S,T]},
\end{align}
where the last inequality follows from the fact $s_3-s_2=s_4-s_3=(t-s)/2$ and $s_4-s_1=(u-s)+(t-s)=\frac32 (t-s)$. 
By exactly the same argument (we just need to take $s_3$ in place of $s_2$, $s_4$ in place of $s_3$ and $s_5$ in place of $s_4$), we get
the exact same bound for $I_2$, and then by \eqref{step1}, for $\E^{s-(t-s)}\delta A_{s,u,t}$ as well.
Since by the assumptions of the lemma $\gamma+\frac{(\theta-1)\wedge0}{\alpha}+1>1$ and $\frac{\theta-2}{\alpha}+1+\tau>1$, we see that condition \eqref{SSL2} is satisfied with $\Gamma_2=N\|f\|_{\C^\theta} \ddnew{g-h}{\gamma}{p}{[S,T]}$ and $\Gamma_3=N\|f\|_{\C^\theta}C_g\db{g-h}_{\scC^0_p|\G,[S,T]}$.

It remains to verify condition \eqref{limpart}. Let $s,t\in[S,T]$, $s<t$. Fix $m\in \N$. Denote 
$t_i:=s+i\frac{t-s}m$, $i=0,\ldots,m$. Note that $t_{i}-(t_{i+1}-t_i)=t_{i-1}$. 
Then we have
\begin{align}
	\bigl|\A_t-\A_s-\sum_{i=1}^{m-1} A_{t_i,t_{i+1}}\bigr|
	&\le \sum_{i=1}^{m-1}  \int_{t_i}^{t_{i+1}} \bigl|f(L_r+g_r) -\cP_{r-t_{i-1}} f(L_{t_{i-1}}+\E^{t_{i-1}} g_r)\,\bigr|dr\nn\\
	&\phantom{\le}+ \sum_{i=1}^{m-1}  \int_{t_i}^{t_{i+1}} \bigl|f(L_r+h_r) -\cP_{r-t_{i-1}} f(L_{t_{i-1}}+\E^{t_{i-1}} h_r)\,\bigr|dr\nn\\
	&\phantom{\le}+ \int_{t_0}^{t_1} \big(|f(L_r+g_r)|+|f(L_r+h_r)|\big)\,dr\nn\\
	&=:I_{m,1}+I_{m,2}+I_{m,3}.\label{integraldiff}
\end{align}
Using \eqref{eq:K1A} (with $\rho=0$, $\theta\wedge\alpha$ in place of $\beta$, and $\mu=\tfrac{\theta\wedge\alpha}{\alpha}$) we easily deduce that for any $\eps>0$
\begin{align*}
	&\bigl|f(L_r+g_r) -\cP_{r-t_{i-1}} f(L_{t_{i-1}}+\E^{t_{i-1}} g_r)\,\bigr|\\
	&\quad\le
	\|f\|_{\C^{\theta\wedge 1}} (|L_r-L_{t_{i-1}}|^{\theta\wedge 1}\wedge1)+\|f\|_{\C^{\theta\wedge 1}} |g_r -\E^{t_{i-1}} g_r|^{\theta\wedge 1} 
	+\|f\|_{\C^{\theta\wedge\alpha}} |r-{t_{i-1}}|^{\frac{\theta\wedge\alpha}{\alpha}-\eps}.
\end{align*}
This together with \eqref{eq:basic moments}, \eqref{a:g} and the Minkowski inequality yields 
\begin{equation}\label{I1fin+}
	\|I_{m,1}\|_{L_1(\Omega)}\le N (1+C_g) \|f\|_{\C^\theta} (t-s)\frac1{m^{\big(\frac{\theta\wedge 1}{\alpha}\wedge ((\theta\wedge1) \tau)\wedge1\big)-\eps}}.
\end{equation}
Similarly, with the help of \eqref{a:h} we bound 
\begin{equation}\label{I2fin}
	\|I_{m,2}\|_{L_1(\Omega)}\le N (1+C_h)\|f\|_{\C^\theta} (t-s)\frac1{m^{\big(\frac{\theta\wedge 1}{\alpha}\wedge ((\theta\wedge1) \eps_0)\wedge1\big)-\eps}}.
\end{equation}
Finally, it is obvious that $|I_{m,3}|\le N m^{-1}\|f\|_{\C^0}$.
Therefore, substituting this, \eqref{I1fin+} and \eqref{I2fin} into \eqref{integraldiff} and choosing $\eps>0$ sufficiently small, we see that the  sum 
$\sum_{i=1}^{m-1} A_{t_i,t_{i+1}}$ converges to $\A_t-\A_s$ in $L_1(\Omega)$ and hence in probability as $m\to\infty$. Hence, \eqref{limpart} holds.

Thus, all the conditions of \cref{lem:shiftedmodifiedSSL} are satisfied. The claimed bound \eqref{L32bound} follows now from~\eqref{SSL3 cA}.
\end{proof}

\begin{remark}
We now understand why it was essential to use the shifted stochastic sewing lemma rather than the usual stochastic sewing lemma. Indeed, the exponent in the second term in \eqref{step2}, $\frac{\theta - 2}{\alpha}$, can be less than $-1$. Had we applied the usual stochastic sewing lemma,  we would have been required to impose $\frac{\theta - 2}{\alpha} > -1$ to ensure integrability. Later, when we apply \cref{L:32} for $f=b$ and $\theta=\beta$, this would have led to the suboptimal condition $\beta > 2 - \alpha$ rather than $\beta > 1 - \frac{\alpha}{2}$. This issue is effectively resolved through the use of shifting.
\end{remark}

Next, we obtain a general quadrature estimate. It will be crucial for bounding the third error term   $\cE^{f,n,3}$ in decomposition \eqref{eq:error-decomposition} as well as for the proof of
\cref{t:quadr}.
An analogue of such a bound in the case of fractional Brownian motion in place of $L$ and $0$ in place of $g$ is obtained in \cite[Lemma~4.1]{ButDarGEr}, with rate that is consistent with \eqref{Lkappabound} below.

\begin{lemma}\label{l:approx}
Assume \ref{A:1}-\ref{A:3}. Let $g\colon[0,1]\times\Omega\to\R^d$ be  a bounded measurable function, let $f\in\C^\theta$, $\theta\in[0,1]$. Suppose that the following holds:

\begin{enumerate}[$($i$)$]
	\item there exist constants 
	$\tau>0$, $C_g>0$ such that 
	\begin{align}\label{a:tau}
		&\tau>\frac12 +\frac1\alpha+\frac\theta\alpha\wedge\frac12 -\frac{\theta}\alpha;\\
		&\E^{s}|\E^t g_r-\E^{s}g_r|\le C_g |r-s|^{\tau},\quad 0\le s \le r \le 1,\,\, t\in[s,1]\label{a:gnew};
	\end{align}
	\item for some $n\in\N$ and all $t\in[0,1]$
	\begin{equation}\label{a:measure}
		g_{\kappa_n(t)}\text{ is $\F_{(\kappa_n(t)-\frac1n)\vee0}$ measurable}.
	\end{equation}
	\item\label{condtheorem} $g\equiv0$ \textbf{or} $\theta>0$.
\end{enumerate}

Then for any $\eps\in(0,1/2)$ there exists a constant $N=N(\alpha,\theta,d,\tau,\eps,M)$ independent of $n$ such that for any $\{0\le S \le T \le 1\}$, and any $\sigma$-algebra $\G\subset\F_{\kappa_n(S)}$ the following  holds:
\begin{align}\label{Lkappabound}
	&\Bigl\|\int_S^T (f(L_r+g_{\kappa_n(r)})-f(L_{\kappa_n(r)}+g_{\kappa_n(r)}))\,dr\Bigr\|_{L_2(\Omega)|\G}\nn\\
	  &\qquad\le  N\|f\|_{\C^\theta}(1+ C_g) n^{-\bigl(\frac12 +\frac\theta\alpha\wedge\frac12 \bigr)+2\eps}|T-S|^{\half+\eps}.
\end{align}
\end{lemma}
\begin{proof}
To simplify notations, set $\psi_s:=g_{\kappa_n(s)}$. We fix $0\le S\le T\le 1$ and
apply \cref{lem:shiftedmodifiedSSL}  for the processes
\begin{align}\label{eq:A-def-third}
	A_{s,t}&:=\E^{s-(t-s)}\int_s^t (f(L_r+\E^{s-(t-s)}\psi_r)-f(L_{\kappa_n(r)}+\E^{s-(t-s)}\psi_r))\,dr,\,\, (s,t)\in\Delta_{[S,T]};\\
	\A_{t}&:=\int_0^t (f(L_r+\psi_r)-f(L_{\kappa_n(r)}+\psi_r))\,dr,\quad t\in[S,T]\nn.
\end{align}
First we verify \eqref{SSL1}. If $(s,t)\in \Delta_{[S,T]}$ and $s\le t\le s+2/n$, then we have from \eqref{eq:basic moments} (with $m=\theta$ and $p=2$)
\begin{align}\label{smallgap}
	\|A_{s,t}\|_{L_2(\Omega)|\G}&\le
	\int_s^t \|f(L_r+\E^{s-(t-s)}\psi_r)-f(L_{\kappa_n(r)}+\E^{s-(t-s)}\psi_r)\|_{L_2(\Omega)|\G}\,dr\nn\\
	&\le \|f\|_{\C^\theta}\int_s^t \||L_r-L_{\kappa_n(r)}|^{\theta}\wedge1\|_{L_2(\Omega)|\G}\,dr\nn\\
	&\le N\|f\|_{\C^\theta}|t-s| n^{-\big(\frac{\theta}{\alpha}\wedge\frac12\big)+\eps}\nn\\
	&\le N \|f\|_{\C^\theta}|t-s|^{\half+\eps}n^{-\big(\half+\frac{\theta}{\alpha}\wedge\frac12\big)+2\eps},
\end{align}
where we used in the penultimate inequality that $\G\subset\F_{\kappa_n(S)}$ and thus $L_r-L_{\kappa_n(r)}$ is independent of $\G$. The last inequality follows from the fact that $t-s\le 2/n$.

Now consider the case $t\geq s+2/n$. Then we note that $r\geq s$ implies $\kappa_n(r)\geq s-(t-s)$, and in fact $\kappa_n(r)-( s-(t-s))\geq (t-s)/2$. Recalling \eqref{very-very-basic}, we see that
\begin{equ}
	A_{s,t}=\int_s^t\big(\cP_{r-(s-(t-s))}-\cP_{\kappa_n(r)-(s-(t-s))}\big)f(L_{s-(t-s)}+\E^{s-(t-s)}\psi_r)\,dr.
\end{equ}
Applying \eqref{eq:K1A} with $\rho=0$, $\theta\wedge((1-\eps)\alpha)$ in place of $\beta$,  $\mu=(\frac\theta\alpha\wedge\frac12)+\frac12-\eps$, we get (note that all the assumptions of \cref{Prop:HKbounds} are satisfied with such choice of parameters)
\begin{align*}
	|A_{s,t}|&\le N \|f\|_{\C^\theta}\int_s^tn^{-\bigl(\frac12 +\frac\theta\alpha\wedge\frac12 \bigr)+2\eps}
	|t-s|^{(\frac{\theta}{\alpha}\wedge(1-\eps))-(\frac{\theta}{\alpha}\wedge\frac12)-\frac12+\eps}\,dr\\
	&\leq  N \|f\|_{\C^\theta} n^{-\bigl(\frac12 +\frac\theta\alpha\wedge\frac12 \bigr)+2\eps}|t-s|^{\half+\eps},
\end{align*}
which implies 
\begin{equation*}
	\|A_{s,t}\|_{L_2(\Omega)|\G}\le  N \|f\|_{\C^\theta} n^{-\bigl(\frac12 +\frac\theta\alpha\wedge\frac12 \bigr)+2\eps}|t-s|^{\half+\eps}.
\end{equation*}
Recalling \eqref{smallgap}, we see that the condition \eqref{SSL1} is satisfied with $\Gamma_1=N\|f\|_{\C^\theta}n^{-\bigl(\frac12 +\frac\theta\alpha\wedge\frac12 \bigr)+2\eps}$ and $\eps_1=\eps$.

Moving on to the condition \eqref{SSL2}, take $(s,t)\in\Delta_{[S,T]}$ and $u:=(t+s)/2$. As before, denote $s_1:=s-(t-s)$, $s_2:=s-(u-s)$, $s_3:=s$, $s_4:=u$, $s_5:=t$.
We need to bound $\E^{s-(t-s)}\delta A_{s,u,t}=\E^{s_1}\delta A_{s_3,s_4,s_5}$.
By a standard computation we see that
\begin{align}\label{ivanitwo}
	\E^{s_1}\delta A_{s_3,s_4,s_5}&=\E^{s_1}\E^{s_2}\int_{s_3}^{s_4} f(L_r+\E^{s_1}\psi_r)-f(L_{\kappa_n(r)}+\E^{s_1}\psi_r)
	\nn\\
	&\hspace{3cm}-f(L_r+\E^{s_2}\psi_r)+f(L_{\kappa_n(r)}+\E^{s_2}\psi_r)\,dr
	\nn\\
	&\quad+\E^{s_1}\E^{s_3}\int_{s_4}^{s_5} f(L_r+\E^{s_1}\psi_r)-f(L_{\kappa_n(r)}+\E^{s_1}\psi_r)
	\nn\\
	&\hspace{3cm}-f(L_r+\E^{s_3}\psi_r)+f(L_{\kappa_n(r)}+\E^{s_3}\psi_r)\,dr\nn\\
	&=:I_1+I_2.
\end{align}
The two terms are treated in exactly the same way, so we only discuss bounding the first one.
When $|t-s|\geq 4 /n$, then for $r\geq s_3$ we have 
\begin{equ}\label{eq:trivial}
	\kappa_n(r)-s_2\geq s_3-1/n-s_2\geq (t-s)/4>0.
\end{equ}
Therefore we first write
\begin{equs}
	I_1&=\E^{s_1}\int_{s_3}^{s_4}\big(\cP_{r-s_2}-\cP_{\kappa_n(r)-s_2}\big)f(L_{s_2}+\E^{s_1}\psi_r\big)
	\\
	&\hspace{2cm}-\big(\cP_{r-s_2}-\cP_{\kappa_n(r)-s_2}\big)f(L_{s_2}+\E^{s_2}\psi_r\big)\,dr.
\end{equs}
Applying \eqref{eq:K1A} with $\rho=1$, $\theta$ in place of $\beta$,  $\mu=\frac12 +\frac\theta\alpha\wedge\frac12$, and using \eqref{eq:trivial} yields (we see again  that all the assumptions of \cref{Prop:HKbounds} are satisfied with such choice of parameters)
\begin{align}\label{i1cor3b}
	|I_1|&\leq\int_{s_3}^{s_4}\|\big(\cP_{r-s_2}-\cP_{\kappa_n(r)-s_2}\big)f\|_{\cC^1}
	\E^{s_1}|\E^{s_1}\psi_r-\E^{s_2}\psi_r|\,dr\nn\\
	&\le N\|f\|_{\C^\theta}\int_{s_3}^{s_4}n^{-{\bigl(\frac12 +\frac\theta\alpha\wedge\frac12 \bigr)}+\eps}|t-s|^{-\big(\frac12 +\frac\theta\alpha\wedge\frac12+\frac1\alpha -\frac{\theta}\alpha\big)}	|r-s_1|^\tau\,dr\nn\\
	&\le N\|f\|_{\C^\theta} C_g n^{-{\bigl(\frac12 +\frac\theta\alpha\wedge\frac12 \bigr)}+\eps}|t-s|^{1+\widetilde\eps},
\end{align}
where we used \eqref{a:gnew} in the second inequality and put $\wt \eps:=\tau-\big(\frac12 +\frac\theta\alpha\wedge\frac12+\frac1\alpha -\frac{\theta}\alpha\big)>0$ by \eqref{a:tau}. By a similar argument, \eqref{i1cor3b} holds for $|I_2|$.  Therefore, taking $L_2(\Omega)|\G$ norm and recalling \eqref{ivanitwo} we can conclude that in the case $|t-s|\geq 4 /n$ we have 
\begin{equ}\label{eq:lastdelta}
	\|\E^{s-(t-s)}\delta A_{s,u,t}\|_{L_2(\Omega)|\G}\le N \|f\|_{\C^\theta}C_g n^{-\bigl(\frac12 +\frac\theta\alpha\wedge\frac12 \bigr)+\eps}|t-s|^{1+\widetilde\eps}.
\end{equ}

Next, consider the case $|t-s|\leq 4 /n$. By assumption, $\psi_r=g_{\kappa_n(r)}$ is  $\F_{(\kappa_n(r)-\frac1n)\vee0}$ measurable. Therefore, if $r\in[\kappa_n(s_1),\kappa_n(s_1)+\frac2n)$, then $\psi_r$ is $\F_{\kappa_n(s_1)}$--measurable. Since $\kappa_n(s_1)\le s_1\le s_2$, one has $\E^{s_1}\psi_r=\E^{s_2}\psi_r=\psi_r$ and thus the integrand in $I_1$ is zero. Hence, we can concentrate on the case $r\ge \kappa_n(s_1)+\frac2n$. In this case, $\kappa_n(r)-1/n\ge  s_1$. Thus we get 
\begin{align*}
	I_1&=\E^{s_1}\int_{[s_3,s_4]\,\cap\,[\kappa_n(s_1)+\frac2n,1]}\E^{\kappa_n(r)-\frac1n}\Big( f(L_r+\E^{s_1}\psi_r)-f(L_{\kappa_n(r)}+\E^{s_1}\psi_r)
	\nn\\
	&\hspace{3cm}-f(L_r+\E^{s_2}\psi_r)+f(L_{\kappa_n(r)}+\E^{s_2}\psi_r)\Big)\,dr\\
	&=\E^{s_1}\int_{[s_3,s_4]\,\cap\,[\kappa_n(s_1)+\frac2n,1]}\cP_{\frac1n}
	f(L_{\kappa_n(r)-\frac1n}+\E^{s_2}\psi_r)-\cP_{\frac1n}
	f(L_{\kappa_n(r)-\frac1n}+\E^{s_1}\psi_r)\\  &\hspace{2cm}+\cP_{r-\kappa_n(r)+\frac1n}f(L_{\kappa_n(r)-\frac1n}+\E^{s_1}\psi_r)-\cP_{r-\kappa_n(r)+\frac1n}f(L_{\kappa_n(r)-\frac1n}+\E^{s_2}\psi_r)\,dr,
\end{align*}
where we used again that $\psi_r$ is $\F_{\kappa_n(r)-\frac1n}$-measurable.
Applying \eqref{eq:K2A} with $\rho=1$ and recalling \eqref{a:gnew}, we immediately deduce
\begin{align*}
	|I_1|&\le N \|f\|_{\C^\theta}C_g
	\int_{[s_3,s_4]\,\cap\,[\kappa_n(s_1)+\frac2n,1]}
	(\|\cP_{\frac1n}f\|_{\cC^1}+\|\cP_{r-\kappa_n(r)+\frac1n}f\|_{\cC^1})\E^{s_1}|\E^{s_1}\psi_r-\E^{s_2}\psi_r|\,dr
	\\
	&\le N\|f\|_{\C^\theta}C_g|t-s|^{1+\tau}n^{-\frac{\theta-1}\alpha}
	\\
	&\le N\|f\|_{\C^\theta}C_g|t-s|^{1+\eps}n^{-\tau-\frac{\theta-1}\alpha+\eps}\\
	&\le N\|f\|_{\C^\theta}C_g|t-s|^{1+\eps}n^{-\bigl(\frac12 +\frac\theta\alpha\wedge\frac12 \bigr)+\eps},
\end{align*}	
where the penultimate inequality follows from the fact that $|t-s|\le 4/n$, and in the last inequality we used \eqref{a:tau}. By the same argument, exactly the same bound holds also for $|I_2|$. Recalling now \eqref{ivanitwo} and \eqref{eq:lastdelta}, we can therefore conclude that  \eqref{SSL2} is satisfied with
with $\Gamma_2=N\|f\|_{\C^\theta} C_g n^{-\bigl(\frac12 +\frac\theta\alpha\wedge\frac12 \bigr)+\eps}$ and $\eps_2=\eps\wedge\widetilde \eps$.

It remains to verify that the process $\A_t$ satisfies \eqref{limpart}. 	Fix now $m\in \N$. Denote $t_i:=s+i\frac{t-s}m$, $i=0,\ldots,m$. We get 
\begin{align}
	\Bigl\|\A_t-\A_s-\sum_{i=1}^{m-1} A_{t_i,t_{i+1}}\Bigr\|_{L_1(\Omega)}
	&\le \|\A_{t_1}-\A_{t_0}\|_{L_1(\Omega)}\nn\\
	&\phantom{\le}+	\Bigl\|\sum_{i=1}^{m-1}\A_{t_{i+1}}-\A_{t_{i}}-\E^{t_{i}}
	(\A_{t_{i+1}}-\A_{t_{i}})\Bigr\|_{L_2(\Omega)}\nn\\
	&\phantom{\le}+\Bigl\|\sum_{i=1}^{m-1}\E^{t_{i}}(\A_{t_{i+1}}-\A_{t_{i}})-\E^{t_{i-1}}
	(\A_{t_{i+1}}-\A_{t_{i}})\Bigr\|_{L_2(\Omega)}\nn\\
	&\phantom{\le}+\Bigl\|\sum_{i=1}^{m-1}A_{t_i,t_{i+1}}-\E^{t_{i-1}}
	(\A_{t_{i+1}}-\A_{t_{i}})\Bigr\|_{L_1(\Omega)}\nn\\
	&:=I_1+I_2+I_3+I_4.\label{manyiterms}
\end{align}
Since $f$ is bounded, we clearly have
\begin{equation}\label{I1fbnd}
	I_1\le 2\|f\|_{\C^0}m^{-1}.
\end{equation}
Next, we note that the sequence $\bigl(\A_{t_{i+1}}-\A_{t_{i}}-\E^{t_{i}}
(\A_{t_{i+1}}-\A_{t_{i}})\bigr)_{i=1,\ldots,m-1}$ is a martingale difference sequence with respect to the filtration $(\F_{t_{i+1}})_{i=1,\ldots,m-1}$. Therefore, the Burkholder-Davis-Gundy inequality implies 
\begin{equation}\label{I2fbnd}
	I_2^2\le \sum_{i=1}^{m-1}\Bigl\|\A_{t_{i+1}}-\A_{t_{i}}-\E^{t_{i}}
	(\A_{t_{i+1}}-\A_{t_{i}})\Bigr\|_{L_2(\Omega)}^2\le 4\|f\|_{\C^0}^2m^{-1}.
\end{equation}
Similarly, the sequence $\bigl(\E^{t_{i}}
(\A_{t_{i+1}}-\A_{t_{i}})-\E^{t_{i-1}}
(\A_{t_{i+1}}-\A_{t_{i}})\bigr)_{i=1,\ldots,m-1}$ is a martingale difference sequence with respect to the filtration $(\F_{t_{i}})_{i=1,\ldots,m-1}$, and we get 
\begin{equation}\label{I3fbnd}
	I_3^2\le \sum_{i=1}^{m-1}\Bigl\|\E^{t_{i}}
	(\A_{t_{i+1}}-\A_{t_{i}})-\E^{t_{i-1}}
	(\A_{t_{i+1}}-\A_{t_{i}})\Bigr\|_{L_2(\Omega)}^2\le 16\|f\|_{\C^0}^2m^{-1}.
\end{equation}
Finally, if $g\equiv0$, then $I_4=0$. If $g\not\equiv0$, then by condition \eqref{condtheorem} of the theorem $\theta>0$. Therefore, using \eqref{standardbound} we derive for any $i=1,\ldots,m-1$
\begin{align*}
	&\|\E^{t_{i-1}}(\A_{t_{i+1}}-\A_{t_{i}})-A_{t_{i},t_{i+1}}\|_{L_1(\Omega)}\\
	&\qquad\leq \Big\|\int_{t_{i}}^{t_{i+1}} f(L_r+g_{\kappa_n(r)})-f(L_r+\E^{t_{i-1}}g_{\kappa_n(r)}) \Big\|_{L_1(\Omega)}\\
	&\qquad\phantom{\le}+\Big\|\int_{t_{i}}^{t_{i+1}}  f(L_{\kappa_n(r)}+g_{\kappa_n(r)})+f(L_{\kappa_n(r)}+\E^{t_{i-1}}g_{\kappa_n(r)})\,dr\Big\|_{L_1(\Omega)}
	\\
	&\qquad\leq 2\|f\|_{\C^\theta}\int_{t_{i}}^{t_{i+1}}\|g_{\kappa_n(r)}-\E^{t_{i-1}}g_{\kappa_n(r)}\|_{L_1(\Omega)}^\theta\,dr\\
	&\qquad\leq 2\|f\|_{\C^\theta}(1+C_g)(t_{i+1}-t_{i-1})^{1+\theta\tau}\le 
	N\|f\|_{\C^\theta}(1+C_g)m^{-1-\theta\tau},
\end{align*}
where the penultimate inequality follows from  the fact that if $\kappa_n(r)\le t_{i-1}$, then $g_{\kappa_n(r)}=\E^{t_{i-1}}g_{\kappa_n(r)}$ and if $\kappa_n(r)\ge t_{i-1}$, then \eqref{a:gnew} is applicable and $\kappa_n(r)-t_{i-1}\le t_{i+1}-t_{i-1}$. Recalling the definition of $I_4$ in \eqref{manyiterms}, we get 
\begin{equation}\label{I4fbnd}
	I_4\le N\|f\|_{\C^\theta}(1+C_g)m^{-\theta\tau} \one_{\theta>0}.
\end{equation}
Collecting together bounds 	\eqref{I1fbnd}--\eqref{I4fbnd} and substituting them into \eqref{manyiterms},
we get 
\begin{equation*}
	\Bigl\|\A_t-\A_s-\sum_{i=1}^{m-1} A_{t_i,t_{i+1}}\Bigr\|_{L_1(\Omega)}\leq N\|f\|_{\C^\theta}(1+C_g)m^{-\theta\tau} \one_{\theta>0}+N\|f\|_{\C^0}m^{-\frac12}. 	
\end{equation*}
which implies \eqref{limpart}. The claimed bound \eqref{Lkappabound} is therefore given by \eqref{SSL3 cA}.
\end{proof}

\begin{remark}\label{r:restr}

The reader might observe two additional constraints introduced in \cref{l:approx} compared to \cref{L:32}. Specifically, we assumed that $\G\subset \F_{\kappa_n(S)}$ and that $p=2$. Both of these conditions are employed to derive the bound~\eqref{smallgap}. With a generic $p>2$, the convergence rate becomes suboptimal, given by $n^{-\bigl(\frac12 +\frac\theta\alpha\wedge\frac1p\bigr)+\eps}$. Without the assumption $\G\subset \F_{\kappa_n(S)}$, there is no deterministic bound on conditional moments of $L_r-L_{\kappa_n(r)}$ given $\G$ for $r\in[S,T]$.

These two restrictions substantially complicate the proof of the main results. The condition $p=2$ is the reason why we must employ the John--Nirenberg machinery at all and bound conditional expectations, instead of merely bounding $L_p(\Omega)$ moments directly as done, for example, in \cite{{ButDarGEr}}. The limitation $\G\subset \F_{\kappa_n(S)}$  leads to additional challenges in the buckling part of the proof of \cref{T:main}.
\end{remark}	

\subsection{Bounds  on $\cE^{f,n,1}$, $\cE^{f,n,2}$, $\cE^{f,n,3}$}\label{sec:E3}

In this part of the paper we apply generic bounds from \cref{sec:E12} to the error terms in decompositions \eqref{eq:error-decomposition} and \eqref{errordec2}.
\begin{corollary}\label{En1En2}
Assume that all the conditions of \cref{t:altm}  are satisfied. 
Then for any $\eps>0$ there exists a constant $N=N(\alpha, \beta,\theta, p,d,\|b\|_{\C^\beta},\eps,M)$ such that for any $0\leq s\leq t\leq 1$, $\sigma$-algebra $\G\subset\F_s$ and all $n\in\N$ the following  holds:
\begin{align}\label{eq:En1-bound}
	\|\cE_{s,t}^{f,n,1}\|_{L_p(\Omega)|\G}&\leq 
	N\|f\|_{\C^\theta} (t-s)^{1+\frac{(\theta-1)\wedge0}{\alpha}}\db{\phi-\phi^n}_{\scC^0_p|\G,[s,t]}\nn\\
	&\qquad + N\|f\|_{\C^\theta} (t-s)^{\frac32+\frac{(\theta-1)\wedge0}{\alpha}}[\phi-\phi^n]_{\scC^{1/2}_p|\G,[s,t]}
	;\\
	\|\cE_{s,t}^{f,n,2}\|_{L_p(\Omega)|\G}&\leq N\|f\|_{\C^\theta}(t-s)^{1+\frac{(\theta-1)\wedge0}{\alpha}}n^{-1}+ N\|f\|_{\C^\theta} (t-s)^{\frac32+\frac{(\theta-1)\wedge0}{\alpha}}n^{-\frac12-(\frac\beta\alpha\wedge\frac12)+\eps} .\label{eq:En2-bound}
\end{align}
\end{corollary}
\begin{proof}
Recall that $\beta>2-2\alpha$ and $\beta>1-\alpha/2$. Without loss of generality, we can assume that $\beta\le2$ and $\theta\le2$: indeed, if one of them is larger than $2$ then we replace the corresponding constant by $2$, this will not affect neither conditions \eqref{eq:main-exponent} and \eqref{eq:main-exponentth}, nor the bounds on the right-hand side of \eqref{eq:En1-bound} and \eqref{eq:En2-bound}. Thus, till the end of the proof choose arbitrary $\delta>0$ small enough so that
\begin{equation}\label{alphabeta}
	\theta\wedge\beta>\max(2-2\alpha, 1-\alpha/2)+\delta\alpha.
\end{equation}

To establish \eqref{eq:En1-bound} and \eqref{eq:En2-bound}, we will apply \cref{L:32} with  $\tau={1+\big(
	\frac{\beta\wedge1}\alpha\wedge 1\big)-\delta}$, $\gamma=1/2$. Thanks to \eqref{alphabeta}, we see that condition \eqref{taugamma}
holds.

First, let us obtain \eqref{eq:En1-bound}. We apply \cref{L:32} with $g=\phi$, $h=\phi^n$. It follows from \eqref{eq:useful-bound} with $q=1$ that \eqref{a:g} holds with $C_g\leq N$. Similarly, \eqref{eq:useful-bound-n} and \eqref{standardbound} imply that \eqref{a:h} holds. Thus all the assumptions of \cref{L:32} are satisfied. Taking into account \eqref{normbound} and the fact that $1+\frac{(\theta-1)\wedge0}{\alpha}\le 1< 1+\frac{\theta-2}{\alpha}+\tau$, we see that \eqref{L32bound} yields \eqref{eq:En1-bound}. 

Now let us prove \eqref{eq:En2-bound}. We will again apply \cref{L:32} with $g_t=\phi^n_t$, $h_t=\phi^n_{\kappa_n(t)}$ and the same $\tau$, $\gamma$ as above. Thanks to \eqref{eq:useful-bound-n} with $q=1$ we see that \eqref{a:g} holds with $C_g\leq N$. Further, if $0\le \kappa_n(t)\le s <t\le 1$ then 
\begin{equation*}
	\|\phi^n_{\kappa_n(t)}-\E^s \phi^n_{\kappa_n(t)}\|_{L_1(\Omega)}=0.
\end{equation*}
If $0\le s < \kappa_n(t) \le t\le 1$, then applying \eqref{eq:useful-bound-n} (in a very rough way) yields
\begin{equation*}
	\|\phi^n_{\kappa_n(t)}-\E^s \phi^n_{\kappa_n(t)}\|_{L_1(\Omega)}\le N |\kappa_n(t)-s|\le N(t-s).
\end{equation*}
Hence, \eqref{a:h} holds. Thus all the assumptions of \cref{L:32} are satisfied. Therefore \eqref{L32bound} implies
\begin{align}\label{almostdonetwo}
	\|\cE_{s,t}^{f,n,2}\|_{L_p(\Omega)|\G}&=\Bigl\|\int_s^t f(L_r+\varphi^n_r)-f(L_s+\varphi^n_{\kappa_n(r)})\,dr\Bigr\|_{L_p(\Omega)|\G}\nn\\
	&\le  N\|f\|_{\C^\theta}(t-s)^{1+\frac{(\theta-1)\wedge0}{\alpha}}\db{\phi^n-\phi^n_{\kappa_n(\cdot)}}_{\scC^0_p|\G,[s,t]}\nn\\
	&\phantom{\le}+ N\|f\|_{\C^\theta} (t-s)^{\frac32+\frac{(\theta-1)\wedge0}{\alpha}}\ddnew{\phi^n-\phi^n_{\kappa_n(\cdot)}}{1/2}{p}{[s,t]}.
\end{align}
Note that for any $r\in[0,1]$, we clearly have
\begin{equation}\label{nkappan}
	\phi^n_r-\phi^n_{\kappa_n(r)}=\big(r-\kappa_n(r)\big)b(L_{\kappa_n(r)}+\phi^n_{\kappa_n(r)}).
\end{equation}
Since $b$ is bounded, this implies 
\begin{equation}\label{AD21}
	\db{\phi^n_\cdot-\phi^n_{\kappa_n(\cdot)}}_{\scC^0_p|\G,[s,t]}\le \|b\|_{\C^0}\sup_{r\in[s,t]}|r-\kappa_n(r)|\le N n^{-1}.
\end{equation}
Let now $s\le r'\le r\le t$. If $\kappa_n(r)\le r'$, then  both $\varphi^n_r$ and $\varphi^n_{\kappa_n(r)}$ are $\cF_{r'}$-measurable, so trivially
\begin{align}
	&\phi^n_{r}-\phi^n_{\kappa_n(r)}- \E^{r'}[\phi^n_{r}-\phi^n_{\kappa_n(r)}]=0
	\label{zerozerozero}
\end{align}
Otherwise if $s\le r'<\kappa_n(r)\le r\le t$, then by \eqref{nkappan} and \eqref{YZcond},
\begin{align*}
	&\E^{r'}\big|\phi^n_{r}-\phi^n_{\kappa_n(r)}- \E^{r'}[\phi^n_{r}-\phi^n_{\kappa_n(r)}]\big|\\
	&\qquad=(r-\kappa_n(r)) \E^{r'}\big|b(L_{\kappa_n(r)}+\phi^n_{\kappa_n(r)})- \E^{r'}b(L_{\kappa_n(r)}+\phi^n_{\kappa_n(r)})\big|\\
	&\qquad\le 2(r-\kappa_n(r)) \E^{r'}\big|b(L_{\kappa_n(r)}+\phi^n_{\kappa_n(r)})- b(L_{r'}+\phi^n_{r'})\big|\\
	&\qquad\le N(r-\kappa_n(r)) \E^{r'}\big(\big|L_{\kappa_n(r)}-L_{r'}\big|^{\beta\wedge1}\wedge1
	+|\phi^n_{\kappa_n(r)}-\phi^n_{r'}|^{\beta\wedge1}\big)\\
	&\qquad\le N(r-\kappa_n(r))\bigl(|r-r'|^{\big(\frac{\beta\wedge1}\alpha\wedge1\big)-\eps}+|r-r'|^{\beta\wedge1}\bigr),
\end{align*}
where in the last inequality we used that $L_{\kappa_n(r)}-L_{r'}$ is independent of $\cF_{r'}$, \eqref{eq:basic moments}, and boundedness of $b$ (which implies Lipschitzness of $\varphi^n$). Note that if $\alpha\ge1$, then clearly $\beta\wedge1\ge\frac{\beta\wedge1}\alpha$. Further, if $\alpha<1$, then $\beta>1/2$ thanks to \eqref{eq:main-exponent}. Thus in both cases $\beta\wedge1\ge
\frac{\beta\wedge1}\alpha\wedge\frac12=\frac{\beta}\alpha\wedge\frac12$, since $\alpha\le2$. We continue the above inequality by taking $\|\cdot\|_{L_p(\Omega|\G}$ norms and writing
\begin{align*}
	\big\|\E^{r'}\big|\phi^n_{r}-\phi^n_{\kappa_n(r)}- \E^{r'}[\phi^n_{r}-\phi^n_{\kappa_n(r)}]\big|\big\|_{L_p(\Omega)|\G}
	&\le N(r-\kappa_n(r))|r-r'|^{\big(\frac{\beta}\alpha\wedge\frac12\big)-\eps}\\
	&\le Nn^{-\frac12-(\frac{\beta}\alpha\wedge\frac12)+\eps}|r-r'|^{1/2},
\end{align*}
where in the last inequality we used that $r-\kappa_n(r)\le n^{-1}\wedge |r-r'|$. This together with \eqref{zerozerozero} yields
\begin{equation}\label{goodnormbound}
	\ddnew{\phi^n_\cdot-\phi^n_{\kappa_n(\cdot)}}{1/2}{p}{[s,t]}\le Nn^{-\frac12-(\frac{\beta}\alpha\wedge\frac12)+\eps}.
\end{equation}
Substituting this and \eqref{AD21} into \eqref{almostdonetwo}, we finally get
\begin{equation*}
	\|\cE_{s,t}^{f,n,2}\|_{L_p(\Omega)}\le  N\|f\|_{\C^\theta}(t-s)^{1+\frac{(\theta-1)\wedge0}{\alpha}}n^{-1}+ N\|f\|_{\C^\theta} (t-s)^{\frac32+\frac{(\theta-1)\wedge0}{\alpha}}n^{-\frac12-(\frac{\beta}\alpha\wedge\frac12)+\eps}. \qedhere
\end{equation*} 
\end{proof}

\begin{remark}\label{rem:WhyImportant}
We see now why it was important in the derivation of \eqref{eq:En2-bound} that the seminorm $\ddnew{\cdot}{\gamma}{p}{[s,t]}$ rather than $[\cdot]_{\scC^\gamma_p|\G,[s,t]}$ appeared in \eqref{L32bound}, recall \cref{rem:spoiler}. Indeed, by taking $r\in[s,t]$ such that $r=\kappa_n(r)$ and $s\le r'\le r\le t$ it is easy to see that one has
\begin{equation*}
	\|\phi^n_{r}-\phi^n_{\kappa_n(r)}- \phi^n_{r'}-\phi^n_{\kappa_n(r')}]\|_{L_p(\Omega)|\G}
	=(r'-\kappa_n(r'))\|b(L_{\kappa_n(r')}+\phi^n_{\kappa_n(r')})\|_{L_p(\Omega)|\G}.
\end{equation*}
This implies that  $[\phi^n_\cdot-\phi^n_{\kappa_n(\cdot)}]_{\scC^\gamma_p,[s,t]}=\infty$ for any $\gamma>0$, and is obviously much worse than \eqref{goodnormbound}.
Note however that for \eqref{eq:En1-bound} we simply bounded $\ddnew{\cdot}{\gamma}{p}{[s,t]}$ by $[\cdot]_{\scC^\gamma_p|\G,[s,t]}$.
\end{remark}

\begin{corollary}\label{cor:2}
Assume that all the conditions of \cref{t:altm}  are satisfied. Then for any $\eps\in(0,1/2)$  there exists a constant $N=N(\alpha,\beta,\theta,d,\|b\|_{\C^\beta},\eps,M)$ such that for any $0\leq s\leq t\leq 1$, $n\in\N$, and any $\sigma$-algebra $\G\subset\F_{\kappa_n(s)}$ the following  holds:
\begin{equation}\label{eq:E3n-bound}
	\|\cE_{s,t}^{f,n,3}\|_{L_2(\Omega)|\G}\leq N \|f\|_{\C^\theta}n^{-\bigl(\frac12 +\frac\theta\alpha\wedge\frac12 \bigr)+2\eps}|t-s|^{\half+\eps}.
\end{equation}
\end{corollary}
\begin{proof}
Choose $\delta>0$ small enough so that
$$
\theta\wedge \beta\wedge 1>1-\frac\alpha2+\delta\alpha.
$$
We apply \cref{l:approx} with $g=\phi^n$ and $\theta\wedge1$ in place of $\theta$.
Note that the rate provided in \eqref{Lkappabound} is consistent with \eqref{eq:E3n-bound}, since $\frac{\theta}{\alpha}\wedge\frac12=\frac{\theta\wedge1}{\alpha}\wedge\frac{1}{2}$. Therefore it remains to verify the conditions of \cref{l:approx}.

For any $0\le s\le t \le 1$, $r\in [s,1]$ we have by 	\eqref{eq:useful-bound-n} with $q=1$
\begin{align*}
	\E^{s}|\E^{s}\varphi^n_r-\E^t\varphi^n_r|
	=\E^{s}|\E^{t}[\E^{s}\varphi^n_r-\varphi^n_r]|\leq\E^{s}|\E^{s}\varphi^n_{r}-\varphi^n_{r}|	\le N |r-s|^{1-\delta+\frac{\beta\wedge\alpha\wedge 1}{\alpha}}.
\end{align*}
Therefore, condition \eqref{a:gnew} is satisfied with $\tau=1-\delta+\frac{\beta\wedge\alpha\wedge 1}{\alpha}$ and $C_g=N$.
We note that if $ \beta\wedge\alpha\wedge1=\beta\wedge1$, then using that $\frac{\theta}{\alpha}\wedge\frac12\le\frac{\theta\wedge1}{\alpha}$ and $\frac{1}{2}-\frac{1}{\alpha}+\frac{\beta\wedge1}{\alpha}>\delta$, we get 
$$
\tau=1-\delta+\frac{\beta\wedge 1}{\alpha}>\frac12+\frac1\alpha\ge \frac12+\frac1\alpha +\frac{\theta}{\alpha}\wedge\frac12-\frac{\theta\wedge1}{\alpha}
$$
and condition \eqref{a:tau} holds. Alternatively, if $\beta\wedge\alpha\wedge 1=\alpha$, then we use that  $\frac12\ge\frac{\theta}{\alpha}\wedge\frac12$ and $1-\frac1\alpha+\frac{\theta\wedge1}{\alpha}>\delta$ to get  
$$
\tau=2-\delta>\frac12+\frac1\alpha+\frac{\theta}{\alpha}\wedge\frac12-\frac{\theta\wedge1}{\alpha}.
$$
Hence also in this case  condition \eqref{a:tau}. We also see that $\phi^n_{\kappa_n(t)}$ is $\F_{(\kappa_n(t)-\frac1n)\vee0}$-measurable for $t\in[0,1]$. Thus all the conditions of \cref{l:approx} are satisfied and we get \eqref{eq:E3n-bound}.
\end{proof}

Gathering the error bounds of \cref{En1En2} and \cref{cor:2}, we finally derive the following crucial conditional quadrature estimate. 

\begin{corollary}\label{c:mainc}
Assume that all the conditions of \cref{t:altm}  are satisfied. 
Then for any $\eps\in(0,\frac12+\frac{(\theta-1)\wedge0}\alpha)$ there exists a constant $N_0=N_0(\alpha, \beta,\theta, d,\|b\|_{\C^\beta},\eps,M)$ such that for any $n\in\N$, $k\in\{0,1,...n\}$, $0\leq k/n\le s\leq t\leq 1$,   the following  holds:
\end{corollary}
\begin{align}\label{boundcor}
&\Bigl\|\int_s^t f(X_r)\,dr-\int_s^t f(X^n_{\kappa_n(r)})\,dr\Bigr\|_{L_2(\Omega)|\F_{\frac{k}n}}\nn\\
&\qquad\le 
N_0\|f\|_{\C^\theta}|t-s|^{\frac12+\eps}\big(n^{-\bigl(\frac12 +\frac{\beta\wedge\theta}\alpha\wedge\frac12 \bigr)+2\eps}+\|\varphi-\varphi^n\|_{\scC^0_2|\F_{\frac{k}n},[s,t]}+[\varphi-\varphi^n]_{\scC^{1/2}_2|\F_{\frac{k}n},[s,t]}\big).
\end{align}
\begin{proof}
Fix $k/n\le s\le t\le 1$.  We use \eqref{eq:En1-bound}, \eqref{eq:En2-bound}, and \eqref{eq:E3n-bound} with $\G=\F_{\frac{k}n}$ to bound each term in \eqref{eq:error-decomposition}.
Since $k/n\le s$, we see that $k/n\le \kappa_n(s)$. Therefore, $\F_{\frac{k}n}\subset \F_{\kappa_n(s)}$ and the conditions of \cref{cor:2} hold.
Using that $1+\frac{(\theta-1)\wedge0}{\alpha}>\frac12+\eps$ and $\bigl(\frac12 +\frac{\beta\wedge\theta}\alpha\wedge\frac12 \bigr)-\eps<1$, we get the desired bound.
\end{proof}

\subsection{Main proofs}\label{sec:mainproof}

\begin{proof}[Proof of \cref{T:main}]
Fix $\eps>0$ small enough so that 
$$	
\gamma_0:=\frac12 +\frac{\beta}\alpha\wedge\frac12-2\eps>0.
$$

In the proof we will apply \eqref{boundcor} with $f=b$, $\theta=\beta$. We take now $\Delta>0$ small enough so that 
\begin{equation}\label{deltaeq}
	\Delta^{\eps} N_0\|b\|_{\C^\beta}\leq 1/4,
\end{equation}
where $N_0$ is the constant from \eqref{boundcor}. Recall the decomposition \eqref{phiphin}.  Fix $n\in\N$ and put
\begin{equation}\label{adefhard}
	\A_t^n:=X_t-X^n_t=\phi_t-\phi_t^n=\int_0^t (b(X_r)-b(X^n_{\kappa_n(r)})\,dr+x_0-x_0^n,\quad t\in[0,1].
\end{equation}	

\textbf{Step 1}.
We claim that for any $S\in\{0,\frac1n,\frac2n,\ldots,1\}$,  $T\in[S,1\wedge(S+\Delta)]$ we have
\begin{equation}\label{eq:phieq}
	\,[\A^n]_{\scC^{1/2}_2|\F_{S},[S,T]}\leq n^{-{\gamma_0}}+|\A^n_{S}|.	
\end{equation}

Indeed, fix $S\in\{0,\frac1n,\frac2n,\ldots,1\}$, $T\in[S,1\wedge(S+\Delta)]$. Note that for any $S\le s\le t \le T$
$$
\A^n_t-\A^n_s=(\phi_t-\phi_s)-(\phi^n_t-\phi^n_s)=\int_s^t b(X_r)\,dr-\int_s^t b(X^n_{\kappa_n(r)})\,dr.
$$
By taking in \eqref{boundcor} $k/n=S$, we get from the above identity
\begin{align}\label{boundcor2}
	\bigl\|\A^n_t-\A^n_s\bigr\|_{L_2(\Omega)|\F_{S}}&\le 
	N_0\|b\|_{\C^\beta}|t-s|^{\frac12+\eps}\big(n^{-\gamma_0}+\|\varphi-\varphi^n\|_{\scC^0_2|\F_{S},[S,T]}+[\varphi-\varphi^n]_{\scC^{1/2}_2|\F_S,[S,T]}\big)\nn\\
	&\le	N_0\|b\|_{\C^\beta}|t-s|^{\frac12+\eps}\big(n^{-\gamma_0}+|\phi_{S}-\phi^n_{S}|+2[\varphi-\varphi^n]_{\scC^{1/2}_2|\F_{S},[S,T]}\big)	,
\end{align}
where we used  that if $s,t\in [S,T]$, then trivially both (semi)norms on the right-hand side of  equation \eqref{boundcor} can be replaced by the ones on $[S,T]$.  
Dividing  \eqref{boundcor2} by $|t-s|^{1/2}$ and taking supremum over $s,t\in[S,T]$, $s\le t$, we get
\begin{equation*}
	\,[\A^n]_{\scC^{1/2}_2|\F_{S},[S,T]}\le N_0\|b\|_{\C^\beta} \Delta^{\eps}\big(n^{-\gamma_0}+|\A^n_{S}|+2[\A^n]_{\scC^{1/2}_2|\F_{S},[S,T]}\big),
\end{equation*}
where we also used the definition of $\A^n$ in \eqref{adefhard} and the inequality $t-s\le T-S\le \Delta$.
By our choice of $\Delta$ \eqref{deltaeq}, the estimate buckles and by putting $2N_0\|b\|_{\C^\beta}\Delta^{\eps}[\A^n]_{\scC^{1/2}_2|\F_{S},[S,T]}$ to the left-hand side, we get
\eqref{eq:phieq}.

\textbf{Step 2}. We get rid of the assumption that $S$ is a gridpoint and  claim that for any $0\le S< T\le 1$, $T\in[S,1\wedge(S+\Delta)]$ we have
\begin{equation}\label{claim}
	\|\A^n_T-\A^n_S\|_{L_2(\Omega)|\F_{S}}\leq \Delta^{\frac12}( n^{-\gamma_0}+|\A^n_{S}|)+4\|b\|_{\C^0}n^{-1}.	
\end{equation}

Let $0\le S<T\le 1$, $T\le S+\Delta$.  If $T-S\le \frac1n$, there is nothing to prove: \eqref{claim} follows from the fact that $|\A^n_T-\A^n_S|\le 2\|b\|_{\C^0}|T-S|\le 2\|b\|_{\C^0}n^{-1}$. Therefore, we assume from now on that $T>S+\frac1n$. Denote $S':=\kappa_n(S)+\frac1n$; that is, $S'$ is the smallest gridpoint strictly bigger than $S$. We see that $S'>S$ and therefore    we have
\begin{align}\label{step121}
	\|\A_T^n-\A_S^n\|_{L_2(\Omega)|\F_{S}}&\le \|\A_T^n-\A_{S'}^n\|_{L_2(\Omega)|\F_{S}}+\|\A_{S'}^n-\A_S^n\|_{L_2(\Omega)|\F_{S}}\nn\\
	&= \bigl(\E\bigl[\E(|\A^n_T-\A^n_S|^2|\F_{S'})\bigl|\F_S\bigr]\bigr)^{\frac12}+\|\A_{S'}^n-\A_S^n\|_{L_2(\Omega)|\F_{S}}.
\end{align}	
Next, since $S'$ is a gridpoint we have $S'\in \{0,\frac1n,\frac2n,...,1\}$ and $T\le S'+\Delta$.
Therefore, the results of Step~1 are applicable and we get from  \eqref{eq:phieq}
\begin{equation*}
	\|\A^n_T-\A^n_{S'}\|_{L_2(\Omega)|\F_{S'}}\leq \Delta^{\frac12}( n^{-\gamma_0}+|\A^n_{S'}|)\le \Delta^{\frac12}( n^{-\gamma_0}+|\A^n_{S}|)+|\A^n_S-\A^n_{S'}|.
\end{equation*}	
Substituting this back into \eqref{step121} and using again that $|\A^n_S-\A^n_{S'}|\le 	2\|b\|_{\C^0} |S-S'|\le 2\|b\|_{\C^0} n^{-1}$, we get  \eqref{claim}.

\textbf{Step 3}. Now with \eqref{claim} in hand we apply the weighted John-Nirenberg inequality \cref{prop:vmo_john_nirenberg} to the process $\A^n$ introduced in \eqref{adefhard}. Setting $\xi^n_t:=\Delta^{\frac12}( n^{-\gamma_0}+|\A^n_{t}|)+4\|b\|_{\C^0}n^{-1}$, $t\in[0,1]$, we see that both processes $\A^n$ and $\xi^n$ are continuous and \eqref{eq:vmo_gamma} holds thanks to \eqref{claim}. Therefore \eqref{jnres} implies that for any $p\ge1$ there exists $N_1=N_1(p)$ independent of $n$ such that for any $0\le S<T\le 1$, $T\le S+\Delta$
\begin{align*}
	\|\sup_{r\in[S,T]}|\A^n_r-\A^n_S|\,\|_{L_p(\Omega)|\F_S}&\le 	 N_1\Delta^{\frac12}n^{-\gamma_0}+N_1\Delta^{\frac12}\|\sup_{r\in[S,T]}|\A^n_r|\,\|_{L_p(\Omega)|\F_S}+4N_1\|b\|_{\C^0}n^{-1}\\
	&\le 	 N_1\Delta^{\frac12}n^{-\gamma_0}+N_1\Delta^{\frac12}|A_S|\\
	&\phantom{\le}+N_1\Delta^{\frac12}\|\sup_{r\in[S,T]}|\A^n_r-A_S|\,\|_{L_p(\Omega)|\F_S}+4N_1\|b\|_{\C^0}n^{-1}.
\end{align*}
Take now $\Delta$ small enough so that in addition to our standing assumption 
\eqref{deltaeq} we have also
\begin{equation*}
	N_1\Delta^{\frac12} \leq 1/2.
\end{equation*}
Then the estimate buckles and we get
\begin{equation*}
	\|\sup_{r\in[S,T]}|\A^n_T-\A^n_S|\,\|_{L_p(\Omega)}\le 	 n^{-\gamma_0}+\|A_S\|_{L_p(\Omega)}+8N_1\|b\|_{\C^0}n^{-1},
\end{equation*}
whenever $0\le S<T\le 1$, $T-S<\Delta$. Iterating this bound $\lceil \Delta^{-1}\rceil$ times yields
\begin{equation*}
	\|\sup_{r\in[0,1]}|\A^n_r|\,\|_{L_p(\Omega)}\le 	 Nn^{-\gamma_0}+N|x_0-x_0^n|,
\end{equation*}
for some $N=N(d,\alpha.\beta,p,\eps,M,\|b\|_{\C^\beta})$, which is the claimed error estimate \eqref{eq:main-rate}.
\end{proof}

\begin{proof}[Proof of \cref{C:main}]
The statement follows from \cref{T:main}  by the standard argument. Namely, we fix 	$\eps>0$ and choose $p>1/\eps$. Set
$$
\eta(\omega):=\sup_{n\in\N}(n^{\frac12 +\frac{\beta}\alpha\wedge\frac12-2\eps}\|X(\omega)-X^n(\omega)\|_{\C^0([0,1])}).
$$
Then by \cref{T:main},
\begin{align*}
	\E \eta^p&=\E\bigl[\sup_{n\in\N}(n^{p(\frac12 +\frac{\beta}\alpha\wedge\frac12-2\eps)}\|X(\omega)-X^n(\omega)\|_{\C^0([0,1])}^p)\bigr]
	\\
	&\le \E\bigl[\sum_{n\in\N}n^{p(\frac12 +\frac{\beta}\alpha\wedge\frac12-2\eps)}\|X(\omega)-X^n(\omega)\|_{\C^{0}([0,1])}^{p}\bigr]\\
	&\le N \sum_{n\in\N}n^{-p\eps}<\infty.
\end{align*}
Thus $\eta<\infty$ a.s. which completes the proof. 
\end{proof}

\begin{proof}[Proof of \cref{t:altm}]
We argue as in the proof of \cref{c:mainc}, but we  skip the buckling step.
Take $\eps>0$ small enough so that 
$$	
\gamma_{f}:=\frac12 +\frac{\beta\wedge\theta}\alpha\wedge\frac12-2\eps>0.
$$
Fix $n\in\N$ and put
\begin{equation}\label{afdef}
	\A_t^{n,f}:=\int_0^t (f(X_r)-f(X^n_{\kappa_n(r)}))\,dr,\quad t\in[0,1].
\end{equation}	

\textbf{Step~1}. Take $\Delta=\Delta(\alpha,\beta,d,\|b\|_{\C^\beta},\eps,M)$ as in \eqref{deltaeq}. Then \cref{c:mainc} and  \eqref{eq:phieq} imply for any $S\in\{0,\frac1n,\frac2n,...,1\}$,  $T\in[S,1\wedge(S+\Delta)]$ 
\begin{align}\label{altmbound}
	\|\A_T^{n,f}-\A_S^{n,f}\|_{L_2(\Omega)|\F_S}&\le N\|f\|_{\C^\theta} (n^{-\gamma_f}+\|\varphi-\varphi^n\|_{\scC^0_2|\F_{S},[S,T]}+[\varphi-\varphi^n]_{\scC^{1/2}_2|\F_{S},[S,T]})\nn\\
	&\le N\|f\|_{\C^\theta} (n^{-\gamma_f}+|\varphi_S-\varphi^n_S|+2[\varphi-\varphi^n]_{\scC^{1/2}_2|\F_{S},[S,T]})\nn\\
	&\le N\|f\|_{\C^\theta} (n^{-\gamma_f}+|\varphi_S-\varphi^n_S|)
\end{align}
for $N=N(\alpha,\beta,\theta, d,\|b\|_{\C^\beta},\eps,M)$. 

\textbf{Step~2}. We remove the restriction in \eqref{altmbound} that $S$ is a grid point. Let $0\le S<T\le 1$, $T\le S+\Delta$. We note that again that if $T-S\le \frac1n$, then $|\A_T^{n,f}-\A_S^{n,f}|\le 2 \|f\|_{\C^0}n^{-1}$ and  \eqref{altmbound} holds. Otherwise, if $T>S+\frac1n$, we put $S':=\kappa_n(S)+\frac1n$. Since $S'$ is a gridpoint, we have from Step~1
\begin{align}\label{aboveeq}
	\|\A^{n,f}_T-\A^{n,f}_{S'}\|_{L_2(\Omega)|\F_{S'}}&\leq  N\|f\|_{\C^\theta} (n^{-\gamma_f}+|\varphi_{S'}-\varphi^n_{S'}|)\nn\\
	&\le N\|f\|_{\C^\theta} (n^{-\gamma_f}+|\varphi_{S}-\varphi^n_{S}|)+N\|f\|_{\C^\theta}n^{-1},
\end{align}	
where we used that $|\phi_S-\phi_{S'}|\le \|b\|_{\C^0}n^{-1}$ and 
$|\phi^n_S-\phi^n_{S'}|\le \|b\|_{\C^0}n^{-1}$.
Therefore, using that $\F_S\subset\F_{S'}$ and the boundedness of $f$, we get from \eqref{aboveeq} 
for any $0\le S<T\le 1$, $T\le S+\Delta$
\begin{align}\label{fjn}
	\|\A_T^{n,f}-\A_S^{n,f}\|_{L_2(\Omega)|\F_{S}}&\le \|\A_T^{n,f}-\A_{S'}^{n,f}\|_{L_2(\Omega)|\F_{S}}+\|\A_{S'}^{n,f}-\A_S^{n,f}\|_{L_2(\Omega)|\F_{S}}\nn\\
	&\le N\|f\|_{\C^\theta} (n^{-\gamma_f}+|\varphi_{S}-\varphi^n_{S}|),
\end{align}	
where  $N=N(\alpha,\beta,\theta,d,\|b\|_{\C^\beta},\eps,M)$.

\textbf{Step~3}. Now we apply the  weighted John-Nirenberg inequality \cref{prop:vmo_john_nirenberg} to the process $\A^{n,f}$. Thanks to the boundedness of $f$ and $b$ the processes $\A^{n,f}$ and $\phi-\phi^n$ are continuous. Therefore, by \cref{prop:vmo_john_nirenberg}, inequality \eqref{fjn} yields for any $0\le S<T\le 1$, $T\le S+\Delta$, $p\ge1$
\begin{align*}
	\|\sup_{r\in[S,T]}|\A^{n,f}_r-\A^{n,f}_S|\,\|_{L_p(\Omega)}&\le 	 N\|f\|_{\C^\theta} n^{-\gamma_f}+N\|f\|_{\C^\theta}\|\sup_{r\in[S,T]}|\phi_r-\phi_r^n|\,\|_{L_p(\Omega)}\\
	&\le 	N\|f\|_{\C^\theta} (n^{-\gamma_f} + |x_0-x_0^n|),
\end{align*}
where the last inequality follows from \cref{T:main} and $N=N(\alpha,\beta,\theta,d,\|b\|_{\C^\beta},\eps,M,p)$ is independent of $S,T$. Iterating the above bound $\lceil \Delta^{-1}\rceil$ times, we get \eqref{eq:levy-rate}.
\end{proof}

\begin{proof}[Proof of \cref{t:quadr}]
We will use the same proof strategy as in the proofs of \cref{T:main,t:altm}: namely, we first obtain the conditional bound at grid points, then remove the grid points restriction, and finally, apply the John-Nirenberg inequality.

We note that the theorem does not follow directly from \cref{t:altm} because the function $f$ is allowed to be of any non-negative regularity (or just bounded if $\theta=0$). We stress that we do not impose the restriction $\theta > 1 - \frac{\alpha}{2}$ as before.

Take $\eps>0$ small enough so that 
$$	
\gamma_{L,f}:=\frac12 +\frac\theta\alpha\wedge\frac12-2\eps>0.
$$
Fix $n\in\N$ and put
\begin{equation*}
	\A_t^{n,L,f}:=\int_0^t (f(L_r)-f(L^n_{\kappa_n(r)}))\,dr,\quad t\in[0,1].
\end{equation*}		

\textbf{Step~1}.
Let $S\in\{0,\frac1n,\frac2n,\ldots,1\}$,  $T\in[S,1]$ and  apply \cref{l:approx} with $g\equiv0$.	We get 
\begin{equation}\label{altmboundL}
	\|\A_T^{n,L,f}-\A_S^{n,L,f}\|_{L_2(\Omega)|\F_S}\le N\|f\|_{\C^\theta} n^{-\gamma_{L,f}}
\end{equation}
for $N=N(\alpha,\theta,d,\eps,M)$. 

\textbf{Step~2}. Now we take arbitrary $S\in[0,1]$. Let $T\in[S,1]$. Let $S':=\kappa_n(S)+\frac1n$. As before, if $T-S\le1/n$, then there is nothing to prove: one have 	$|\A_T^{n,L,f}-\A_S^{n,L,f}|\le 2\|f\|_{\C^0} n^{-1}$. Alternatively, we have $T\ge S'$ and therefore by \eqref{altmboundL} and boundedness of $f$
\begin{align*}
	\|\A_T^{n,L,f}-\A_S^{n,L,f}\|_{L_2(\Omega)|\F_{S'}}&\le	\|\A_T^{n,L,f}-\A_{S'}^{n,L,f}\|_{L_2(\Omega)|\F_{S'}}+\|\A_{S'}^{n,L,f}-\A_{S}^{n,L,f}\|_{L_2(\Omega)|\F_{S'}}\\
	&\le  N\|f\|_{\C^\theta} n^{-\gamma_{L,f}}.
\end{align*}
This implies that 	\eqref{altmboundL} holds for any $0\le S\le T\le 1$.

\textbf{Step~3}. Since the process $\A^{n,L,f}$ is continuous and \eqref{altmboundL} holds for any $0\le S\le T\le 1$, we see that all the conditions of the  John-Nirenberg inequality \cref{prop:vmo_john_nirenberg} are satisfied. Therefore for any $p\ge1$ there exists a constant $N=N(\alpha,\theta,d,\eps,p,M)>0$ such  that
\begin{equation*}
	\|\sup_{r\in[0,1]}|\A^{n,L,f}_r|\,\|_{L_p(\Omega)}\le N\|f\|_{\C^\theta} n^{-\gamma_{L,f}},
\end{equation*}
which is the desired bound \eqref{eq:levy-rate}.
\end{proof}

\begin{appendix}
\section{Proofs of the well-posedness of SDE \eqref{eq:main-SDE}}	

To show that SDE \eqref{eq:main-SDE} is strongly well-posed we will run a fixed point argument. Throughout the Appendix we will assume that assumptions \ref{A:1}--\ref{A:3} holds. We will assume without loss of generality that $\beta\le2$.  For a  measurable bounded function $f\colon[0,1]\times \Omega\to\R^d$, $\gamma\in(0,1]$, $p\ge1$, $T\in(0,1]$ put
\begin{equation*}
	\|f\|_{{\scC^{\gamma}_p,[0,T]}}:=\sup_{t\in[0,T]}\|f(t)\|_{L_p(\Omega)}+
	\sup_{s,t\in[0,T]}\frac{\|f(t)-f(s)\|_{L_p(\Omega)}}{(t-s)^\gamma}.
\end{equation*}
and consider a mapping 
$$
\cS f(t):=\eta+\int_0^t b(L_s+f_s)\, ds,\quad 0\le t\le 1,
$$
where $\eta\in\R^d$ is a $\F_0$--measurable vector.
We claim the following contraction bound.
\begin{lemma}\label{L:cboundSEU}
	Let $\phi,\psi\colon[0,1]\times\Omega\to\R^d$ be bounded, adapted, measurable functions. Suppose that $1-\alpha/2<\beta<2$. Assume that there exist constants $\tau\in(0,1]$, $\tau>(2-\beta)/\alpha$, $C_\phi,C_\psi>0$, $T_0\in(0,1]$, $S\in[0,1-T_0]$ such that 
	for any $S\le s \le t \le S+T_0$ one has
	a.s.
	\begin{equation}\label{phipsicond}
		\E^{s}|\phi_t-\E^{s}\phi_t|\le C_\phi |t-s|^{\tau};\quad 
		\E^{s}|\psi_t-\E^{s}\psi_t|\le C_\psi |t-s|^{\tau}.
	\end{equation} 
	Then for any $p\ge2$ there exist constants $N=N(\alpha,\beta,\tau,p,d,\|b\|_{\C^\beta})$, $\eps=\eps(\alpha,\beta,\tau)>0$ such that for any $T\in(0,T_0]$ and $S\in[0,1-T]$ one has
	\begin{equation}\label{contbound}
		\|\cS\phi-\cS\psi\|_{{\scC^{\frac12}_p,[S,S+T]}}\le N (1+C_\phi)T^\eps \db{\phi-\psi}_{\scC^{\frac12}_p,[S,S+T]}+\|\cS\phi(S)-\cS\psi(S)\|_{L_p(\Omega)}.
	\end{equation}
\end{lemma}
\begin{proof}
	Fix $T\in(0,T_0]$.
	We will apply \cref{L:32} with $g=\phi$, $h=\psi$, $\gamma=1/2$, $\theta=\beta$, $f=b$, and $\mathcal{G}=\{\emptyset,\Omega\}$. We see that all the conditions of the Lemma are satisfied and it follows from \eqref{L32bound} that for any $S\le s\le t\le S+T$
	\begin{align}\label{wpstep1}
		\|(\cS\phi(t)-\cS\psi(t))-(\cS\phi(s)-\cS\psi(s))\|_{L_p(\Omega)}&
		=\Bigl\|\int_s^t (b(L_s+\phi_s)-b(L_s+\psi_s))\,ds\Bigr\|_{L_p(\Omega)}\nn\\
		&\le N (1+C_\phi) |t-s|^{\frac12+\eps}	\db{\phi-\psi}_{\scC^{\frac12}_p,[0,T]},
	\end{align}
	where $\eps:=\frac12+\frac{(\beta-1)\wedge0}{\alpha}>0$ thanks to \eqref{eq:main-exponent}.
	Therefore, by taking in the above inequality $s=S$, we get
	\begin{equation*}
		\|\cS\phi(t)-\cS\psi(t)\|_{L_p(\Omega)}\le N T^{\eps} (1+C_\phi) \db{\phi-\psi}_{\scC^{\frac12}_p,[0,T]}+		\|\cS\phi(S)-\cS\psi(S)\|_{L_p(\Omega)}.
	\end{equation*}
	Combining this with \eqref{wpstep1}, we get the desired bound \eqref{contbound}.
\end{proof}

\begin{proof}[Proof of \cref{L:SEU}: strong uniqueness]
	Let $X$, $\wt X$ be two strong solutions to \eqref{eq:main-SDE} with the initial condition $\eta$ adapted to the same filtration $\F$; here $\eta$ is $\F_0$--measurable vector in $\R^d$. Define $\phi:=X-L$, $\wt \phi:=\wt X -L$. Since $X,\wt X$ solve \eqref{eq:main-SDE}, we obviously have $\cS\phi=\phi$, $\cS\wt \phi=\wt\phi$.
	
	It follows from \eqref{eq:main-exponent} that one can choose $\delta>0$ small enough so that 
	\begin{equation}\label{betacondsmall}
		\beta>\max(2-2\alpha+\delta\alpha, 1-\alpha/2+\delta\alpha).
	\end{equation}
	We apply \cref{L:cboundSEU} with $S=0$ to the functions $\phi$, $\wt \phi$. By \cref{lem:useful-lemma}(iii), condition \eqref{phipsicond} holds with $T_0=1$,  $\tau=1+(\frac{\beta\wedge\alpha\wedge1}{\alpha})-\delta$ and $C_\phi=C_\psi=N$. It is easy to see that with such choice of $\tau$, one has $\tau>(2-\beta)/\alpha$, and thus all the conditions of \cref{L:cboundSEU} hold. Therefore, \eqref{contbound} implies
	for any $T\in(0,1]$
	\begin{equation}\label{contraction2sol}
		\|\phi-\wt \phi\|_{{\scC^{\frac12}_2,[0,T]}}=
		\|\cS\phi-\cS\wt \phi\|_{{\scC^{\frac12}_2,[0,T]}}\le N T^\eps \db{\phi-\wt\phi}_{\scC^{\frac12}_2,[0,T]}.
	\end{equation}
	Note that $\db{\phi-\wt\phi}_{\scC^{\frac12}_2,[0,T]}<\infty$ regardless of the moment conditions imposed on $\eta$ (we did not impose any). Indeed,
	\begin{equation*}
		\db{\phi-\wt\phi}_{\scC^{\frac12}_2,[0,T]}=\Bigl\|\int_0^\cdot (b(L_s+\phi_s)-b(L_s+\wt\phi_s))\,ds\Bigr\|_{\scC^{\frac12}_2,[0,T]}\le 2\|b\|_{\C^0}.
	\end{equation*}
	Since $N$ in \eqref{contraction2sol} does not depend on $T$, by taking in \eqref{contraction2sol} $T$ small enough such that $NT^\eps\le1/2$, we have $\phi=\wt \phi$  on $[0,T]$, which implies $X=\wt X$ on $[0,T]$. Repeating this procedure $\lceil 1/T\rceil$ times by starting at time $iT$ instead of $0$, $i=1,\ldots,\lceil 1/T\rceil-1$,   (note that we have not imposed any moment conditions on $\eta$, so there is no problem that, e.g., $X_T$ does not have a second moment) we get strong uniqueness on the interval $[0,1]$.
\end{proof}

To establish strong existence, we consider a sequence $\phi^{(0)}(t):=\eta$, $0\le t\le 1$, $\phi^{(n)}:=\cS^n \phi^{(0)}$, $n\in\N$. Clearly, for any $n\in\Z_+$, $t\in[0,1]$
\begin{equation}\label{phistep}
	\phi^{(n+1)}(t)=\eta+\int_0^t b(L_s+\phi^{(n)}(s))\,ds.
\end{equation}
\begin{lemma}\label{L:A2}
	For any $\eps>0$, there exist constants  $N=N(\alpha,\beta,\eps,\|b\|_{\C^\beta})>0$, 
	$T_0=T_0(\alpha, \beta,\eps,\|b\|_{\C^\beta})\in(0,1]$ such that for any $n\in\Z_+$, $S\le s\le t\le S+ T_0$ 
	\begin{equation}\label{statementA2} 
		\E^s|\phi^{(n)}_t-\E^s\phi^{(n)}_t|\le N |t-s|^{1+\frac {\beta\wedge1\wedge\alpha}\alpha-\eps}.
	\end{equation}
\end{lemma}
\begin{proof}
	Fix $\eps>0$. Denote $m:=1+\frac {\beta\wedge1\wedge\alpha}\alpha-\eps$. First we note that 	
	\begin{equation}\label{thetacond}
		1+(\beta\wedge1)m-m>0.
	\end{equation}	
	Indeed, if $\beta\ge1$ this is obvious; otherwise since $\beta>1-\alpha$ we have
	\begin{equation*}
		m< 1+\frac{\beta}{\alpha}<1+\frac{\beta}{1-\beta}=\frac{1}{1-\beta},
	\end{equation*}
	which implies \eqref{thetacond}. Fix now $T_0\in(0,1]$ small enough such that
	\begin{equation}\label{Tcond}
		(8 \|b\|_{\C^\beta})^{\beta\wedge1} T_0^{1+(\beta\wedge1)m-m}\le 1.
	\end{equation}
	
	We will prove by induction over $n$ that \eqref{statementA2} holds with $8 M^*\|b\|_{\C^\beta}$ in place of $N$, where
	$$
	M^*:=M(\beta\wedge(\alpha-\eps\alpha/2),\eps/2)+1;
	$$
	recall the definition of the function $M=M(p,\eps)$ in \ref{A:3}.

	The case $n=0$ is obvious. Assume that the lemma holds for some $n\in\Z_+$, and let us prove it for $n+1$. Arguing similar to the proof of \cref{lem:useful-lemma} and using \eqref{YZcond}, we derive for any $S\le s\le t\le S+T_0$ 
	\begin{align}
		\E^s|\phi^{(n+1)}_t-\E^s\phi^{(n+1)}_t|&\le 2 \E^s\Bigl|\phi^{(n+1)}_t-\phi^{(n+1)}_s-\int_s^t b(L_s+\E^s\phi^{(n)}_r)\,dr\Bigr|\nn\\
		&=2\E^s\Bigl|\int_s^t (b(L_r+\phi^{(n)}_r)-b(L_s+\E^s\phi^{(n)}_r))\,dr\Bigr|\nn\\
		&\le 4 \|b\|_{\C^\beta}\int_s^t (\E^s[|L_r-L_s|^{\beta\wedge1}\wedge1]+(\E^s|\phi^{(n)}_r-\E^s\phi^{(n)}_r|)^{\beta\wedge1})\,dr\nn\\
		&\le 4 \|b\|_{\C^\beta}\Bigl(M^*|t-s|^{1+(\frac {\beta\wedge1\wedge\alpha}\alpha-\eps)}
		+ (8 \|b\|_{\C^\beta} M^*)^{\beta\wedge1}|t-s|^{1+(\beta\wedge1)m}\Bigr)\label{nontrivial}\\
		&\le 4 \|b\|_{\C^\beta}M^*\Bigl(|t-s|^{m}
		+ (8 \|b\|_{\C^\beta})^{\beta\wedge1}|t-s|^{m} T_0^{1+(\beta\wedge1)m-m}\Bigr)\label{semitrivial}\\
		&\le 8 \|b\|_{\C^\beta}M^*|t-s|^{m}\nn,
	\end{align}
	where in \eqref{nontrivial} we used the induction step and assumption \ref{A:3}, in \eqref{semitrivial} we used \eqref{thetacond}, and in the last inequality we used \eqref{Tcond}. This proves \eqref{statementA2} for $n+1$, and thus completes the proof.
\end{proof}

Now we can obtain the existence part of \cref{L:SEU}.
\begin{proof}[Proof of \cref{L:SEU}: strong existence and convergence of the Picard iterations] Fix $p\ge2$ and $\delta>0$ small enough such that \eqref{betacondsmall} holds.
	
	\textbf{Step 1}. 	 For any $n\in\Z_+$ we apply \cref{L:cboundSEU} to $\phi^{(n)}$ and $\phi^{(n+1)}$ with $\tau=1+(\frac{\beta\wedge\alpha\wedge1}{\alpha})-\delta$; as in the proof of uniqueness part it is clear that $\tau>(2-\beta)/\alpha$. We see that by \cref{L:A2} condition \eqref{phipsicond} holds for some $T_0$ not depending on $n$, $C_\phi=C_\psi=N$, where $N$ also does not depend on $n$. Thus, all the conditions of  \cref{L:cboundSEU} are met and we have for any 
	$n\in\Z_+$,  $T\in(0,T_0]$, $S\in[0,1-T]$,
	\begin{align*}
		\db{\phi^{(n+2)}-\phi^{(n+1)}}_{\scC^{\frac12}_p,[S,T]}&=	\|\cS\phi^{(n+1)}-\cS\phi^{(n)}\|_{{\scC^{\frac12}_p,[S,T]}}\\
		&\le N T^\eps \db{\phi^{(n+1)}-\phi^{(n)}}_{\scC^{\frac12}_p,[S,T]}+\|\phi^{(n+2)}(S)-\phi^{(n+1)}(S)\|_{L_p(\Omega)},
	\end{align*}
	where $N=N(\alpha, \beta,p,d,\|b\|_{\C^\beta})$, $\eps=\eps(\alpha,\beta)$. 
	
	Pick now $T\in(0,T_0]$ small enough so that $ N T^\eps\le 1/2$.
	We stress that the choice of $T$ does not depend on $n$.	
	Let $M:=\lceil \frac1T\rceil$ and let $0=S_0\le S_1\le\hdots\le  S_M=1$ be a partition of $[0,1]$ such that $S_{m+1}-S_m\le T$, $m=1,\hdots,M$.  Then for any $m=0,\hdots,M-1$ we get 
	\begin{align}\label{mainresstep1}
		\db{\phi^{(n+2)}-\phi^{(n+1)}}_{\scC^{\frac12}_p,[S_m,S_{m+1}]}&\le \frac12 \db{\phi^{(n+1)}-\phi^{(n)}}_{\scC^{\frac12}_p,[S_m,S_{m+1}]}\nn\\
		&\phantom{\le}+\|\phi^{(n+2)}(S_m)-\phi^{(n+1)}(S_m)\|_{L_p(\Omega)}.
	\end{align}
	
	\textbf{Step 2}. We claim that for any $m=0,\hdots,M-1$,
	\begin{equation}\label{claimP}
		\db{\phi^{(n+1)}-\phi^{(n)}}_{\scC^{\frac12}_p,[S_m,S_{m+1}]}\le (m+1)n^m2^{-n } \db{\phi^{(1)}-\phi^{(0)}}_{\scC^{\frac12}_p,[0,1]},\quad n\in\Z_+.
	\end{equation}
	We show \eqref{claimP} by induction over $m$. If $m=0$ then $S_m=0$ and  $\phi^{(n+1)}(0)=\phi^{(n)}(0)=\eta$ for any $n\in\Z_+$. Hence, we immediately get from \eqref{mainresstep1}
	\begin{equation*}
		\db{\phi^{(n+1)}-\phi^{(n)}}_{\scC^{\frac12}_p,[S_0,S_{1}]}\le 2^{-n} \db{\phi^{(1)}-\phi^{(0)}}_{\scC^{\frac12}_p,[S_0,S_{1}]},\quad n\in\Z_+,
	\end{equation*}
	which yields \eqref{claimP}.
	
	Assume now that \eqref{claimP} holds for  some $m-1\in\{0,1,\hdots,M-2\}$. We will show that \eqref{claimP} holds for $m$. We have from the induction assumption for any $n\in\Z_+$
	\begin{align*}
		\|\phi^{(n+1)}(S_m)-\phi^{(n)}(S_m)\|_{L_p(\Omega)}&\le
		\|\phi^{(n+1)}(S_m)-\phi^{(n)}(S_m)\|_{\scC^{\frac12}_p,[S_{m-1},S_{m}]}\\
		&\le mn^{m-1}2^{-n}\db{\phi^{(1)}-\phi^{(0)}}_{\scC^{\frac12}_p,[0,1]}.
	\end{align*}
	Therefore,  \eqref{mainresstep1} yields  for any $n\in\N$
	\begin{equation*}
		\db{\phi^{(n+1)}-\phi^{(n)}}_{\scC^{\frac12}_p,[S_m,S_{m+1}]}\le \frac12 \db{\phi^{(n)}-\phi^{(n-1)}}_{\scC^{\frac12}_p,[S_m,S_{m+1}]}+mn^{m-1}2^{-n}\db{\phi^{(1)}-\phi^{(0)}}_{\scC^{\frac12}_p,[0,1]}.
	\end{equation*}
	Iterating this inequality over $n$ ($m$ is fixed!)
	\begin{align}\label{verytedious}
		\db{\phi^{(n+1)}-\phi^{(n)}}_{\scC^{\frac12}_p,[S_m,S_{m+1}]}
		&\le  2^{-n} \db{\phi^{(1)}-\phi^{(0)}}_{\scC^{\frac12}_p,[0,1]}\nn\\
		&\phantom{\le}+m \sum_{i=1}^n 2^{-n+i}2^{-i}i^{m-1} \db{\phi^{(1)}-\phi^{(0)}}_{\scC^{\frac12}_p,[0,1]}\nn\\
		&=2^{-n}\db{\phi^{(1)}-\phi^{(0)}}_{\scC^{\frac12}_p,[0,1]} (1+ m \sum_{i=1}^n i^{m-1}).
	\end{align}
	Clearly,
	$$
	1+ m \sum_{i=1}^n i^{m-1}\le 1 +mn^m\le (m+1)n^m,
	$$ 
	which together with \eqref{verytedious} gives 
	\begin{equation*}
		\db{\phi^{(n+1)}-\phi^{(n)}}_{\scC^{\frac12}_p,[S_m,S_{m+1}]}\le 
		2^{-n}\db{\phi^{(1)}-\phi^{(0)}}_{\scC^{\frac12}_p,[0,1]} (m+1)n^m,
	\end{equation*}
	which is the desired claim \eqref{claimP}.
	
	\textbf{Step 3}. Note that by definition of $\phi^{(1)}$ and $\phi^{(0)}$
	\begin{equation*}
		\db{\phi^{(1)}-\phi^{(0)}}_{\scC^{\frac12}_p,[0,1]}\le \|b\|_{\C^0}.	
	\end{equation*}
	Therefore, using the  bound $(m+1)n^m2^{-n}\le N 2^{-\frac{n}2}$ whenever $m\le M$ for some $N=N(M)=N(\alpha, \beta,p,d,\|b\|_{\C^\beta})$, we rewrite  \eqref{claimP} as 
	\begin{equation*}
		\db{\phi^{(n+1)}-\phi^{(n)}}_{\scC^{\frac12}_p,[S_m,S_{m+1}]}\le N2^{-\frac{n}2 },\quad n\in\Z_+,
	\end{equation*}
	which yields
	\begin{equation*}
		\db{\phi^{(n+1)}-\phi^{(n)}}_{\scC^{\frac12}_p,[0,1]}\le N2^{-\frac{n}2 },\quad n\in\Z_+,
	\end{equation*}
	and eventually 
	\begin{equation}\label{lastbound}
		\db{\phi^{(k)}-\phi^{(n)}}_{\scC^{\frac12}_p,[0,1]}\le N2^{-\frac{n\wedge k}2 },\quad n,k\in\Z_+.
	\end{equation}

	\textbf{Step 4}. 	It follows by \eqref{lastbound}
	\begin{equation}\label{lb1}
		\sup_{t\in[0,1]}\|\phi^{(k)}_t-\phi^{(n)}_t\|_{L_p(\Omega)}\le \db{\phi^{(k)}-\phi^{(n)}}_{\scC^{\frac12}_p,[0,1]}\le N2^{-\frac{n\wedge k}2 },\quad n,k\in\Z_+.
	\end{equation}	
	
	Thus, for any $t\in[0,1]$ the sequence $(\phi^{(n)}_t-\eta)_{n\in\Z_+}
	=(\phi^{(n)}_t-\phi^{(0)}_t)_{n\in\Z_+}$ converges in $L_p(\Omega)$ as $n\to\infty$. Denote its limit by $\psi_t$. We have from \eqref{lastbound} by Fatou's lemma
	\begin{equation}\label{almostendequ}
		\sup_{t\in[0,1]}\|\psi_t-(\phi^{(n)}_t-\eta)\|_{L_p(\Omega)}\le
		\|\psi-(\phi^{(n)}-\eta)\|_{\scC^{\frac12}_p,[0,1]}\le N2^{-\frac{n}2}.
	\end{equation}
	
	Put $X_t:=\eta+\psi_t+L_t$. We claim that a version of $X$ is a strong solution to \eqref{eq:main-SDE} on $[0,1]$. Indeed, for any $t\in[0,1]$ random vector $X_t$ is clearly $\F_t$-measurable as a limit of $\F_t$-measurable random vectors; further, recalling \eqref{phistep} we have 
	\begin{align*}
		\Bigl\|X_t-\eta-\int_0^t b(X_s)\,ds-L_t\Bigr\|_{L_2(\Omega)}&=\Bigl\|\psi_t - \int_0^t b(L_s+\eta+\psi_s)\,ds\Bigr\|_{L_2(\Omega)}\\
		&\le \|\psi_t-(\phi_t^{(n+1)}-\eta)\|_{L_2(\Omega)}\\
		&\qquad+
		\Bigl\|\int_0^t (b(L_s+\eta+\psi_s)-b(L_s+\phi_s^{(n)}))\,ds\Bigr\|_{L_2(\Omega)}\\
		&\le \|\psi_t-(\phi_t^{(n+1)}-\eta)\|_{L_2(\Omega)}+\int_0^t \|\psi_s-(\phi_s^{(n)}-\eta)\|_{L_2(\Omega)}^{\beta\wedge1}\,ds.
	\end{align*}
	By \eqref{almostendequ}, the right-hand side of the above inequality tends to $0$ as $n\to\infty$. Thus, for any $t\in[0,1]$ we have $X_t=\eta+\int_0^t b(X_s)\,ds+L_t$ a.s.
	In particular, $\wt X_t:=\eta+\int_0^t b(X_s)\,ds+L_t$ coincides with $X$ for a.e. $\omega,t$ and therefore satisfies 
	$$\P\Bigl( \wt X_t=\eta+\int_0^tb(\wt X_s)\,ds+L_t \text{ for all }t\in[0,1]\Bigr)=1.
	$$
	Hence $X$ is a strong solution to SDE \eqref{eq:main-SDE} on $[0,1]$.
	
	\textbf{Step 5}. We see	from the definition of the Picard approximation $Y^{(n)}$ and $\phi^{(n)}$ that
	$Y^{(n)}=\phi^{(n)}+L$. By definition of $\wt X$, we have $X_t=\wt X_t$ a.s. for any fixed $t\in[0,1]$.  Therefore, we can rewrite \eqref{almostendequ} as
	\begin{equation*}
		\|\wt X-Y^{(n)}\|_{\scC^{\frac12}_p,[0,1]}=\|X-Y^{(n)}\|_{\scC^{\frac12}_p,[0,1]}\le N2^{-\frac{n}2}.
	\end{equation*} 
	Since the process $\wt X-Y^{(n)}$ is continuous for almost all $\omega$, the Kolmogorov continuity theorem implies  
	\begin{equation*}
		\| \| \wt X-Y^{(n)}\|_{\C^0([0,1])}\|_{L_p(\Omega)}\le N2^{-\frac{n}2}.
	\end{equation*} 
	which yields the desired convergence of $\| \wt X-Y^{(n)}\|_{\C^0([0,1])}$ to $0$ a.s. and in $L_p(\Omega)$.
\end{proof}

\section{Proofs of the auxiliary statements of the article}

\begin{proof}[Proof of \cref{P:RePhibound}]
	Recall that we denoted the generating triplet of $L$ by $(a,Q,\nu)$. 
	By \cite[Theorem 1.3]{SSW12} with $f(s)=s^{\alpha}$, assumption \eqref{phibound} implies that there exists  $t_0=t_0(\alpha)>0$ such that the gradient bound \eqref{eq:K1newA} holds for small enough $t\in(0,t_0]$. If $t\in(t_0,1]$, then 
	\begin{equation*}
		\|\nabla \cP_t f\|_{\C^0}\le \|\cP_{t-t_0}\nabla \cP_{t_0} f\|_{\C^0}\le N t_0^{-1/\alpha}\|f\|_{\C^0}\le
		\widetilde N t^{-1/\alpha}\|f\|_{\C^0},
	\end{equation*}
	where $\widetilde N=Nt_0^{-1/\alpha}$, which proves
	\eqref{eq:K1newA}.
	
	Now we move  on to \ref{A:2}. Fix $\eps>0$.   Note that by \cite[Lemma~A.2]{KS2019}, \eqref{eq:K1newA} implies that for any $\eps>0$
	\begin{equation*}
		\int_{|y|\le 1}|y|^{\alpha+\eps} \nu(dy)<\infty.
	\end{equation*}
	Therefore, by \cite[Lemma~A.3(i)]{KS2019}, $Q=0$. 
	
	If $\alpha\in[1,2)$, then   \cite[Theorem~3.2(iii)]{KS2019} implies for any $f\in\C^{\alpha+\eps}$ vanishing at infinity 
	\begin{equation*}
		\L f (x)= \langle a,\nabla f(x)\rangle +\int_{|y|\ge1} (f(x+y)-f(x))\,\nu(dy)+\int_{|y|\le1}\int_0^1 \langle \nabla f(x+\lambda y)-\nabla f(x),y\rangle\,d\lambda\nu(dy).
	\end{equation*}
	Hence 
	\begin{equation*}
		\|\L f\|_{\C^0}\le |a|\| f\|_{\C^1}+2\|f\|_{\C^0} \nu (\{|y|\ge1\})+\| f\|_{\C^{\alpha+\eps}} \int_{|y|\le1}|y|^{\alpha+\eps}\nu(dy)\le N\|f\|_{\C^{\alpha+\eps}}.
	\end{equation*}
	Very similarly, $\|\L f\|_{\C^1}\le  N\|f\|_{\C^{\alpha+1+\eps}}$.
	
	If $\alpha\in(0,1)$, then by above $ \int_{|y|\le 1}|y| \nu(dy)<\infty$. Consider now the process $\wt L_t:=L_t+\kappa t$, where $\kappa=-a+\int_{|y|\le1}y\nu(dy)$. Let $\wt \L$ be its generator.  It is immediate to see that 
	$$
	\E e^{i \langle \lambda,\wt L_t\rangle}=\exp\Bigl(t \int_{\R^d} (e^{i\langle \lambda,y\rangle}-1)\nu(dy)\Bigr),\quad \lambda\in\R^d,\,t\ge0.
	$$
	Therefore all the conditions of  \cite[Theorem~3.2(ii)]{KS2019} are satisfied and, thus, for any $f\in\C^{\alpha+\eps}$ vanishing at infinity 
	\begin{equation*}
		\wt \L f (x)= \int_{\R^d} (f(x+y)-f(y))\nu(dy).
	\end{equation*}
	Hence $	\|\wt \L f\|_{\C^0}\le \| f\|_{\C^{\alpha+\eps}} \int_{\R^d}(|y|^{\alpha+\eps}\wedge2)\nu(dy)\le N\|f\|_{\C^{\alpha+\eps}}$. The bound on $\|\wt \L f\|_{\C^1}$ is established by the same argument.
\end{proof}

\begin{proof}[Proofs that \cref{E:NDAS,E:SDAS,E:DCAS,E:Tempered,E:Truncated,E:relativistic,E:BM,E:sum,E:alphasttype} satisfy \ref{A:1}--\ref{A:3}]
	
	We begin with \cref{E:NDAS}. In this case, $\Re \Phi(\lambda)=c_\alpha \int_{\mathbb{S}}|\langle \lambda,\xi\rangle|^\alpha\mu(d\xi)$, for some $c_\alpha>0$ and the upper bound in \eqref{phibound} is immediate. The lower bound follows from the argument presented in \cite[p. 424--425 (after Hypothesis 2)]{Priola1}. By \cref{P:RePhibound}, this implies that $L$ (or its shifted version) satisfies \ref{A:1} and \ref{A:2}.
	
	It is easy to see that $\int_{|y|\ge1}|y|^{p}\nu (dy)<\infty$ for any  $p\in(0,\alpha)$, which implies \cite[Theorem~25.3]{Sato-san} that 
	\begin{equation}\label{ndaspm}
		\E |L_1|^p<\infty,\quad p\in(0,\alpha).
	\end{equation}
	If $\alpha\in(1,2)$, then by \cite[formula (14.15)]{Sato-san}, 
	$\Law (L_t-\kappa t)=\Law(t^{1/\alpha}(L_1-\kappa))$ for some $\kappa\in\R^d$. Thus, using \eqref{ndaspm}, we get for $p\in(0,\alpha)$, $t\in(0,1]$
	\begin{equation*}
		\E |L_t|^p\le N\big( t^{p/\alpha} \E |L_1|^{p}+t^p+t^{p/\alpha}\big)\le N t^{p/\alpha},
	\end{equation*}
	since $\alpha>1$. This shows \ref{A:3}.
	
	If $\alpha=1$, then by \cite[formula (14.16)]{Sato-san}, 
	$\Law L_t=\Law(t L_1+\kappa t \log t)$ for some $\kappa\in\R^d$. Thus, by \eqref{ndaspm}, for any $p\in(0,1)$
	$\E |L_t|^p\le N t^{p(1-\eps)}$, which is \ref{A:3}.
	
	Finally, if $\alpha\in(0,1)$, then consider the process $\wt L_t:=L_t+\kappa t$, for $\kappa=\int_{|y|\le1}|y|\nu(dy)$ (this integral is finite by  \cite[Proposition~14.5]{Sato-san}). By above, $\wt L$ satisfies \ref{A:1} and \ref{A:2}. Further, an easy direct calculation (see also \cite[Remark~14.6]{Sato-san}) shows that  $\wt L$ has a symbol  
	$$
	\wt \Phi(\lambda)=\int_{\mathbb{S}}\int_{0 }^\infty(1-e^{i\langle \lambda,r\xi \rangle} )r^{-1-\alpha}dr\mu(d\xi).
	$$
	Hence $\Law (\wt L_t)=t^{1/\alpha}\wt L_t$, and thus thanks to \eqref{ndaspm}, $\E |\wt L_t|^p\le N t^{p/\alpha}$ for $p\in(0,\alpha)$. Thus, $\wt L$ satisfies \ref{A:3}.
	
	\cref{E:SDAS} is a special case of \cref{E:NDAS} with $\mu$ being uniform measure on $\mathbb{S}$. Similarly, \cref{E:DCAS} is a special case of \cref{E:NDAS} with 
	$\mu=\sum_{k=1}^d (\delta_{e_k}+\delta_{-e_k})$, where $(e_k)$ is the standard basis in $\R^d$, see, e.g., \cite[p. 425]{Priola1}.
	
	Now let us move on to \cref{E:alphasttype}. 
	The L\'evy process $L$ now has the symbol
	$$
	\Phi(\lambda)=\int_0^\infty \int_{\mathbb{S}} \big(1-\cos(r \langle \lambda,\xi\rangle \big)
	r^{-1-\alpha}\rho(r)\,\mu(d\xi) dr,
	$$
	where we  used the fact that $\mu$ is symmetric. It is easy to see  that non-degeneracy of $\mu$ implies that $\int_{\mathbb{S}}|\langle\lambda, \xi\rangle|\,\mu(d\xi)>0$ for any $\lambda\in\mathbb{S}$ and thus
	\begin{equation}\label{nondegmu}
		\inf_{\lambda\in\mathbb{S}}\int_{\mathbb{S}}|\langle\lambda, \xi\rangle|\,\mu(d\xi)>0.
	\end{equation}
	Applying \eqref{rhoprop}, we get for any $\lambda\in\R^d$
	\begin{equation}\label{upperboundstabletype}
		\Phi(\lambda)\le N\int_{\mathbb{S}} \int_0^\infty  \big(1-\cos(r \langle \lambda,\xi\rangle \big)
		r^{-1-\alpha}\, dr\mu(d\xi)	\le N |\lambda|^\alpha\int_{\mathbb{S}}|\langle \frac{\lambda}{|\lambda|},\xi\rangle|^\alpha\,\mu(d\xi)\le  N |\lambda|^\alpha.
	\end{equation}
	Similarly, denoting $\overline \lambda:=\frac{\lambda}{|\lambda|}$, we get  for any $\lambda\in\R^d$, with $|\lambda|>1/C$
	\begin{align*}
		\Phi(\lambda)&\ge  \int_{\mathbb{S}}\int_0^C (1-\cos(r \langle \lambda,\xi\rangle )
		r^{-1-\alpha}\, dr\mu(d\xi)\\
		&= |\lambda|^\alpha \int_{\mathbb{S}}|\langle \overline \lambda,\xi\rangle|^\alpha\int_0^{C|\langle \lambda,\xi\rangle|} (1-\cos r)
		r^{-1-\alpha}\, dr\mu(d\xi)\\
		&\ge \frac1{10} |\lambda|^\alpha \int_{\mathbb{S}}|\langle \overline \lambda,\xi\rangle|^\alpha\int_0^{|\langle \overline \lambda,\xi\rangle|} 
		r^{1-\alpha}\, dr\mu(d\xi)\\
		&\ge N |\lambda|^\alpha \int_{\mathbb{S}}|\langle \overline \lambda,\xi\rangle|^2\mu(d\xi)\\
		&\ge N |\lambda|^\alpha \Bigl( \int_{\mathbb{S}}|\langle \overline \lambda,\xi\rangle|\mu(d\xi)\Bigr)^2\\
		&\ge N  |\lambda|^\alpha.
	\end{align*}
	where we used inequality $1-\cos r>\frac{r^2}{10}$ valid for $r\in[0,1]$ and \eqref{nondegmu}. Combining this with \eqref{upperboundstabletype}, we see that symbol $\Phi$ satisfies \eqref{phibound} for large $|\lambda|$ and \eqref{phiboundH3} for all $\lambda\in\R^d$. Thus, $L$ satisfies \ref{A:1}--\ref{A:3}.
	
	Examples \ref{E:Tempered}  and \ref{E:Truncated} are special cases of Example \ref{E:alphasttype}.
	
	In case of \cref{E:relativistic}, we have $N^{-1}|\lambda|^\alpha\le \Phi(\lambda)\le N|\lambda|^\alpha$ for large enough $|\lambda|$. Further, $\Phi(\lambda)\le N\lambda^2$ for small enough $|\lambda|$. Thus, by \cref{P:RePhibound} and \cref{P:ReH3}, assumptions \ref{A:1}--\ref{A:3} holds for the relativistic $\alpha$--stable process with $\nalpha=\alpha$.
	
	\cref{E:BM}, namely when $L$ is a Brownian motion, is obvious.
	
	Finally, let us treat \cref{E:sum}. We begin with part (i). Without loss of generality, assume $\alpha_1\le \alpha_2$. Put $L:=L^{(1)}+L^{(2)}$. Let $\cP$ (respectively $\cP^{(i)}$) be the semigroup associated with $L$  (respectively $L^{(i)}$, $i=1,2$). Similarly let $\L$, $\L^{(i)}$ be the generators of $L$, respectively $L^{(i)}$.  Then for any $f\in\C^0(\R^d)$, $x\in\R^d$ we have by independence of $L^{(1)}$ and $L^{(2)}$
	\begin{equation*}
		\cP_t f(x)=\E \cP_t^{(2)} f(x+L_t^{(1)}). 
	\end{equation*}
	Thus, since $L^{(2)}$ satisfies \ref{A:1},
	\begin{equation*}
		|\nabla	\cP_t f(x)|=|\E \nabla \cP_t^{(2)} f(x+L_t^{(1)})|\le N t^{-1/\alpha_2} \|f\|_{\C^0}. 
	\end{equation*}
	Therefore, $L$ satisfies \ref{A:1} with $\alpha=\alpha_2$. Since $\L=\L^{(1)}+\L^{(2)}$, we have for any $\eps>0$,  $f\in\C^{\alpha_2+\delta+\eps}(\R^d)$ vanishing at infinity, $\delta=0,1$  
	\begin{equation*}
		\|\L f \|_{\C^\delta(\R^d)}\le 
		\|\L^{(1)} f \|_{\C^\delta(\R^d)}+
		\|\L^{(2)} f \|_{\C^\delta(\R^d)}\le M \|f\|_{\C^{\alpha_1+\delta+\eps}}+
		M \|f\|_{\C^{\alpha_2+\delta+\eps}}\le 2 M \|f\|_{\C^{\alpha_2+\delta+\eps}},	\end{equation*}
	and thus  $L$ satisfies \ref{A:2} with $\alpha=\alpha_2$. Finally, to verify \ref{A:3} we fix $p\in(0,\alpha_2)$. Then If $p< \alpha_1$, then there is nothing to prove: 
	$$
	\E [|L_t^{(1)}+L_t^{(2)}|^p\wedge1]\le N
	\E |L_t^{(1)}|^p+N \E|L_t^{(2)}|^p\le N t^{p/\alpha_2-\eps},
	$$
	since $\alpha_2\ge\alpha_1$. Alternatively, if $\alpha_1\le p<\alpha_2$, then
	$$
	\E [|L_t^{(1)}+L_t^{(2)}|^p\wedge1]\le N
	\E |L_t^{(1)}|^{\alpha_1-\eps\alpha_1}+N \E|L_t^{(2)}|^p\le N (t^{1-2\eps}+t^{\frac{p}{\alpha_2}-\eps})\le N t^{\frac{p}{\alpha_2}-2\eps},
	$$
	where in the last inequality we used  that $\frac{p}{\alpha_2}<1$.
	
	Part (ii) of \cref{E:sum} is immediate.
	
	To treat part (iii) of \cref{E:sum}, we recall that by \cite[Theorem~2.4.16]{Iphone} L\'evy--It\^o decomposition holds. Namely, there exists a $d$-dimensional Brownian motion $W$ with covariance matrix $Q$ and an independent Poisson random measure $N$ such that 
	\begin{equation}\label{decompos}
		L_t=W_t+J_t=W_t+at+\int_0^t\int_{|x|\ge1} x N(dr,dx)+
		\int_0^t\int_{|x|\le1} x \wt N(dr,dx),	
	\end{equation}
	where $\wt N$ is the compensated Poisson measure: $\wt N(D_1,D_2):=N(D_1,D_2)-Leb(D_1) \nu(D_2)$, $D_1\in\mathcal{B}(\R_+)$, $D_2\in\mathcal{B}(\R^d)$, $Leb$ is the Lebesgue measure on $\R_+$.
	
	By noting, that the processes $W$ and $J$ are independent, it is easy to see that \ref{A:1} is satisfied with $\alpha=2$ by exactly the same argument as in the proof of part (i) above.
	
	\ref{A:2} is satisfied with $\alpha=2$ by the same argument as we used in the proof of \cref{P:RePhibound}; recall that $\int_{|y|\le 1} |y|^2 \,\nu(dy)<\infty$ thanks to the requirements to the jump measure.
	
	To show \ref{A:3}, we note that for any $p\in(0,2]$
	\begin{align}
		&\E |W_t|^p\le N t^{p/2};\label{estW}\\
		&\E \Bigl|\int_0^t\int_{|x|\le1} x \wt N(dr,dx)\Bigr|^p\le N t^{p/2}\Bigl(\int_{|x|\le1}x^2\nu(dx)\Bigr)^{p/2}=N t^{p/2}\label{estSJN},
	\end{align}
	where the inequality follows from, e.g., \cite[Lemma~2.4]{Kunita}. Further, by definition 
	\begin{align*}
		\E \Bigl|\int_0^t\int_{|x|\ge1} x N(dr,dx)\Bigr|^{\gamma\wedge1}&=	
		\E \Bigl|\sum_{0\le s\le t } \Delta L_s \one(|\Delta L_s|\ge1)\Bigr|^{\gamma\wedge1}\\
		&\le \E \sum_{0\le s\le t }|\Delta L_s|^{\gamma\wedge1} \one(|\Delta L_s|\ge1)\\
		&= t \int_{|x|\ge1} |x|^{\gamma\wedge1}\,\nu (dx)= Nt.
	\end{align*}
	If $p\le (\gamma\wedge1)$, then combining this with \eqref{estW}, \eqref{estSJN} and substituting into the decomposition \eqref{decompos}, we get
	$$
	\E|L_t|^p\le N t^{p/2}+N t^{p/(\gamma\wedge1)}\le N t^{p/2}.
	$$
	If $(\gamma\wedge1)\le p \le 2$, then similarly
	$$
	\E[|L_t|^p\wedge1]\le N t^{p/2}+N t\le N t^{p/2}.
	$$
	Thus $L$ satisfies \ref{A:3} with $\alpha=2$.
\end{proof}

\begin{proof}[Proof of \cref{lem:shiftedmodifiedSSL}]
	The proof is based on the method of \cite[proof of Lemma 2.2]{Mate20} (see also the proof of the original stochastic sewing lemma in \cite{Khoa}). The novelty here is that we apply the conditional Burkholder-Davis-Gundy inequality rather than the standard one. Let $(Z_i)_{i=2,\hdots,M}$ be a sequence of random vectors in $\R^d$ adapted to the filtration $\mathbb{G}:=(\mathcal{G}_i)_{i\ge0}$, and assume that $\G\subset \G_0$. Then
	\begin{equ}\label{eq:even-odd}
		\sum_{i=2}^M Z_i=\sum_{i=2}^M \E^{\mathcal{G}_{i-2}}Z_i +\sum_{i\,\text{\footnotesize even}}(Z_i-\E^{\mathcal{G}_{i-2}}Z_i)+\sum_{i\,\text{\footnotesize odd}}(Z_i-\E^{\mathcal{G}_{i-2}}Z_i).
	\end{equ}	
	The sequence $(Z_i-\E^{\mathcal{G}_{i-2}}Z_i)_{i\in\{2,\ldots,M\},\,i\,\text{\footnotesize even}}$ is a martingale difference sequence with respect to the filtration $(\mathcal{G}_i)_{i\in\{2,\ldots,M\},\,i\,\text{\footnotesize even}}$.
	We apply the conditional Burkholder-Davis-Gundy inequality \cite[Proposition~27]{cBDG}:
	\begin{equs}
		\Big\|\sum_{i\,\text{\footnotesize even}}(Z_i-\E^{\mathcal{G}_{i-2}}Z_i)\Big\|_{L^p(\Omega)|\G}&\le 
		N 
		\Big\|\sum_{i\,\text{\footnotesize even}}(Z_i-\E^{\mathcal{G}_{i-2}}Z_i)^2\Big\|_{L^{p/2}(\Omega)|\G}^{\frac12}
		\\
		&\leq N \Bigl(
		\sum_{i\,\text{\footnotesize even}}(\big\|Z_i\bigr\|_{L^{p}(\Omega)|\G}^2+\big\|\E^{\mathcal{G}_{i-2}}Z_i\bigr\|_{L^{p}(\Omega)|\G}^2)\Bigr)^{\frac12}\\
		&\le N \Bigl(
		\sum_{i\,\text{\footnotesize even}}\big\|Z_i\bigr\|_{L^{p}(\Omega)|\G}^2\Bigr)^{\frac12},
	\end{equs}
	where in the last step we used \eqref{standardbound} and the fact that $\G\subset\G_{i-2}$.
	Treating the term in \eqref{eq:even-odd} with the odd terms in the same way, we finally get
	\begin{equ}\label{eq:new-SSL-proof}
		\Big\|\sum_{i=2}^M Z_i\Big\|_{L^p(\Omega)|\G}\leq \sum_{i=2}^M \big\|\E^{\mathcal{G}_{i-2}}Z_i\big\|_{L^p(\Omega)|\G}+N\Bigl(\sum_{i=2}^M\| Z_i\|^2_{L_p(\Omega)|\G}\Bigr)^{1/2}.
	\end{equ}
	Now we proceed to the proof. For $0\le s \le t\le1$, $m\in \N$, consider the uniform partition of $[s,t]$: $t^m_i:=s+i(t-s)2^{-m}$, $i=0,\hdots, 2^m$. Put 
	\begin{equation*}
		A^m_{s,t}:=\sum_{i=1}^{2^m-1} A_{t^m_i,t^m_{i+1}},\,m\in\N.
	\end{equation*}
	Then it follows from \eqref{limpart} and the conditional Fatou's lemma that
	\begin{equation}\label{limLp}
		\|\A_t-\A_s\|_{L_p(\Omega)|\G}\le \liminf_{m\to\infty} \|A^m_{s,t}\|_{L_p(\Omega)|\G}.	
	\end{equation}
	For $m\in\N$ we apply \eqref{eq:new-SSL-proof} with $M:=2^m$, $Z_i:=\delta A_{t^m_{i-1},t^{m+1}_{2i-1},t^m_{i}}$, $\mathcal{G}_i:=\F_{t^m_{i}}$. We get 
	\begin{align}\label{threesplit}
		\|A^{m+1}_{s,t}-A^{m}_{s,t}\|_{L_p(\Omega)|\G}&\le\Bigl\|\sum_{i=2}^{2^m} \delta A_{t^m_{i-1},t^{m+1}_{2i-1},t^m_{i}}\Bigr\|_{L_p(\Omega)|\G}+\|A_{t^{m+1}_1,t^{m+1}_2}\|_{L_p(\Omega)|\G}\nn\\
		&\le N\sum_{i=2}^{2^m} \|\E^{\F^m_{t_{i-2}}}\delta A_{t^m_{i-1},t^{m+1}_{2i-1},t^m_{i}}\|_{L_p(\Omega)|\G}\nn\\
		&\phantom{\le}+N\Bigl(\sum_{i=2}^{2^m} \|\delta A_{t^m_{i-1},t^{m+1}_{2i-1},t^m_{i}}\|^2_{L_p(\Omega)|\G}\Bigr)^{1/2}+\|A_{t^{m+1}_1,t^{m+1}_2}\|_{L_p(\Omega)|\G}\nn\\
		&\le N 2^m \Gamma_2 2^{-m(1+\eps_2)}|t-s|^{1+\eps_2}+N 2^m \Gamma_3 2^{-m(1+\eps_3)}|t-s|^{1+\eps_3}\nn\\
		&\phantom{\le} + N 2^{m/2} \Gamma_1 2^{-m(1/2+\eps_1)}|t-s|^{1/2+\eps_1}+\Gamma_1 2^{-m(1/2+\eps_1)}|t-s|^{1/2+\eps_1}\nn\\
		&\le N \Gamma_2 2^{-m\eps_2}|t-s|^{1+\eps_2}+N \Gamma_3 2^{-m\eps_3}|t-s|^{1+\eps_3}\nn\\
		&\phantom{\le} +N  \Gamma_1 2^{-m\eps_1}|t-s|^{1/2+\eps_1}+\Gamma_1 2^{-m(1/2+\eps_1)}|t-s|^{1/2+\eps_1},
	\end{align}
	where in the penultimate inequality we used bounds \eqref{SSL1} and \eqref{SSL2}, the fact that 
	$t^m_{i-1}-(t^m_{i}-t^m_{i-1})=t^m_{i-2}$
	and inequality 
	$$\|\delta A_{s,u,t}\|_{L_p(\Omega)|\G}\le \| A_{s,t}\|_{L_p(\Omega)|\G}+
	\| A_{s,u}\|_{L_p(\Omega)|\G}+\| A_{u,t}\|_{L_p(\Omega)|\G}.$$
	
	Summing up inequalities \eqref{threesplit} over $m$ and using \eqref{SSL1} once again, we deduce for $m\in\N$
	\begin{align*}
		\|A^{m}_{s,t}\|_{L_p(\Omega)|\G}&\le \sum_{i=1}^{m-1}\|A^{i+1}_{s,t}-A^{i}_{s,t}\|_{L_p(\Omega)|\G}+
		\|A^{1}_{s,t}\|_{L_p(\Omega)|\G}\\
		&\le N \Gamma_2 |t-s|^{1+\eps_2}+N \Gamma_3 |t-s|^{1+\eps_3}+N  \Gamma_1|t-s|^{1/2+\eps_1}.
	\end{align*}
	This together with \eqref{limLp} yields the desired estimate \eqref{SSL3 cA}.
\end{proof}

The proof of \cref{prop:vmo_john_nirenberg} relies on the following  inequality which is a very minor modification of \cite[Lemma~2.1]{Le22}.
\begin{proposition}\label{p:khoaineq}
	Let $X$, $Y$ be nonnegative random variables. Assume that for any $\alpha>0$, $\theta\in(0,1)$
	one has 
	\begin{equation}\label{khoaineq}
		\P(X\ge2\alpha)\le \theta\P(X\ge\alpha )+\P(Y\ge\theta\alpha).
	\end{equation}
	Then for any $p\ge1$ there exists a constant $N=N(p)$ such that 
	\begin{equation}\label{rezmain}
		\|X\|_{L_p(\Omega)}\le N(p)\|Y\|_{L_p(\Omega)}.
	\end{equation}	
\end{proposition}	
\begin{proof} We follow the proof of \cite[Lemma~2.1]{Le22}. Using the identity $x^p=p\int_0^x \lambda^{p-1}\,d\lambda$ valid for any $x\ge0$, $p\ge1$, we derive for any $\theta\in(0,1)$, $M>0$
	\begin{align*}
		\E (X\wedge M)^p&=p\int_0^M \P(X\ge\lambda)\lambda^{p-1}\,d\lambda=2^{p}p\int_0^{\frac{M}2} \P(X\ge2\lambda)\lambda^{p-1}\,d\lambda\\
		&\le 
		2^pp\theta \int_0^{\frac{M}2} \P(X\ge\lambda)\lambda^{p-1}\,d\lambda
		+2^p p \int_0^{\frac{M}2} \P(Y\ge\theta\lambda)\lambda^{p-1}\,d\lambda\\
		&\le 2^p \theta \E (X\wedge M)^p+(2\theta^{-1})^p\E Y^p,
	\end{align*}
	where in the penultimate inequality we use our main assumption \eqref{khoaineq}. Take now $\theta:=2^{-p-1}$. Then, by above
	\begin{equation*}
		\frac12\E (X\wedge M)^p\le 2^{p^2+2p}\E Y^p.
	\end{equation*}
	Passing to the limit as $M \to \infty$ and using the monotone convergence theorem, we obtain \eqref{rezmain}.
\end{proof}

\begin{proof}[Proof of \cref{prop:vmo_john_nirenberg}]
	Fix $0\le S\le T$. Put 
	\begin{equation*}\xi^*(\omega):=\sup_{r\in[S,T]}\xi_r(\omega);\quad V^*(\omega):=\sup_{r\in[S,T]}|\A_r(\omega)-A_S(\omega)|,\qquad \omega\in\Omega.
	\end{equation*}
	We will assume that $\xi$ is bounded from below by a deterministic constant $\delta>0$.
	This is no loss of generality: if this is not the case, we can replace $\xi$ with $\xi':=\xi\vee \delta$, note that the condition \eqref{eq:vmo_gamma} still holds with $\xi'$, apply the claim with $\xi'$, and pass to the $\delta\to 0$ limit in the final bound \eqref{jnres}, which is possible thanks to the monotone convergence theorem.
	
	\textbf{Step~1}. We claim that for any stopping times $\tau\le\eta$ taking values in $[S,T]$ we have a.s.
	\begin{equation}\label{tauetast}
		\E^{\tau}\frac{|\A_\eta-\A_\tau|}{\xi^*}\le 2.
	\end{equation}
	First, we prove \eqref{tauetast} for the case when $\eta=T$ and $\tau$ takes finitely many values $S\le t_1<\hdots<t_n= T$.
	Recall the identity $\E^{\tau}(X\one_{\tau=t})=\E^t(X\one_{\tau=t})$ valid for any integrable random variable $X$, see, e.g., \cite[Problem~1.2.17(i)]{Karatzas}.
	Then, using  \eqref{eq:vmo_gamma} and the fact that $\xi^*\ge\xi_\tau$ we deduce 
	\begin{align}\label{mainthingjn1}
		\E^\tau\frac{|\A_T-\A_\tau|}{\xi^*}&\le\E^\tau\frac{|\A_T-\A_\tau|}{\xi_\tau}=\sum_{i=1}^n\E^{t_i}\Big(\frac{|\A_{T}-\A_{t_i}|}{\xi_{t_i}} \one(\tau=t_i)\Big)\nn\\
		&=
		\sum_{i=1}^n \one(\tau=t_i)\frac{1}{\xi_{t_i}} \E^{t_i}|\A_{T}-\A_{t_i}|\le1.
	\end{align}
	
	Next, for a general stopping time $\tau\in[S,T]$, we consider a sequence of stopping times $\tau_n$ taking finitely many values in $[S,T]$ and converging to $\tau$ from above. Then, for any $N>0$ we have using continuity of $\A$ and \eqref{mainthingjn1}:
	\begin{equation}\label{mainthingjn2}
		\E^\tau\frac{|\A_T-\A_{\tau}|}{\xi^*}\le \liminf_{n\to\infty} \E^\tau\frac{|\A_T-\A_{\tau_n}|}{\xi^*}=\liminf_{n\to\infty} \E^\tau\E^{\tau_n}\frac{|\A_T-\A_{\tau_n}|}{\xi^*}\le 1.
	\end{equation}
	Finally, let $\tau\le \eta$ be arbitrarily stopping times taking values in $[S,T]$. Then
	\eqref{mainthingjn2} yields
	\begin{equation*}
		\E^\tau\frac{|\A_\eta-\A_{\tau}|}{\xi^*}\le \E^\tau\frac{|\A_T-\A_{\tau}|}{\xi^*}+\E^\tau\frac{|\A_T-\A_{\eta}|}{\xi^*}\le1 +\E^\tau\E^\eta \frac{|\A_T-\A_{\eta}|}{\xi^*}\le 2,
	\end{equation*}
	which is \eqref{tauetast}.
	
	\textbf{Step~2}. Now we modify the corresponding part of the proof of  \cite[Theorem 1.3]{Le22} to adapt it to our new condition \eqref{tauetast}. For arbitrary $\alpha>0$ we consider two stopping times:
	\begin{equation*}
		\tau_\alpha:=T\wedge \inf\{r\in[S,T]: |\A_r-\A_S|\ge\alpha\};\qquad 
		\eta_\alpha:=T\wedge \inf\{r\in[S,T]: |\A_r-\A_S|\ge2\alpha\}.
	\end{equation*}
	Then for any $\theta\in(0,1)$ we clearly have
	\begin{align*}
		\{V^*\ge 2\alpha\}&\subset \{V^*\ge \alpha, |\A_{\eta_\alpha}-\A_{\tau_\alpha}|\ge\alpha\}\\
		&\subset 
		\{V^*\ge \alpha, |\A_{\eta_\alpha}-\A_{\tau_\alpha}|\ge 2\theta^{-1}\xi^*\}\cup
		\{2\theta^{-1}\xi^*\ge\alpha\}.
	\end{align*}
	Note  that $\{V^*\ge\alpha\}=\{|\A_{\tau_\alpha}-\A_S|\ge\alpha\}$ and hence the random variable $\one_ {\{V^*\ge\alpha\}}$ is $\F_{\tau_\alpha}$--measurable. 
	Therefore, using also  the Chebyshev inequality we get for any $\F_S$--measurable set $G$
	\begin{align*}
		\P(V^*\one_G\ge2\alpha)&=\P(V^*\ge2\alpha,G)\\
		&\le \P(V^*\ge \alpha, |\A_{\eta_\alpha}-\A_{\tau_\alpha}|\ge 2\theta^{-1}\xi^*,G)+ \P(
		2\xi^*\ge\alpha\theta,G)\\
		&\le \E \one_{V^*\ge \alpha}\one_G \E^{\tau_\alpha} \one_{\{|\A_{\eta_\alpha}-\A_{\tau_\alpha}|\ge 2\theta^{-1}\xi^*\}}+ \P(
		2\xi^*\one_G\ge\alpha\theta)\\
		&\le \theta\E \one_{V^*\ge \alpha}\one_G \E^{\tau_\alpha} \frac{|\A_{\eta_\alpha}-\A_{\tau_\alpha}|}{2\xi^*}+ \P(
		2\xi^*\one_G\ge\alpha\theta)\\
		&\le \theta\P (V^*\one_G\ge \alpha) + \P(2\xi^*\one_G\ge\alpha\theta),
	\end{align*}	
	where in the second inequality we used that $\F_S\subset\F_{\tau_\alpha}$ and in the last inequality we used \eqref{tauetast}. Since $\alpha>0$ and $\theta\in(0,1)$ were arbitrary, we see that condition \eqref{khoaineq} of \cref{p:khoaineq} holds. Hence, \eqref{rezmain} yields
	\begin{equation*}
		\|V^*\one_G\|_{L_p(\Omega)}\le N\|\xi^*\one_G\|_{L_p(\Omega)},
	\end{equation*}
	for $N=N(p)$ independent of $G$. Since $G$ was an arbitrary $\F_S$-measurable set, we get   
	\begin{equation*}
		\|V^*\|_{L_p(\Omega)|\F_S}\le N\|\xi^*\|_{L_p(\Omega)|\F_S},
	\end{equation*}
	which is the desired bound \eqref{jnres}.
\end{proof}

\begin{proof}[Proof of \cref{Prop:HKbounds}]
	We begin with the proof of \eqref{eq:K2A}. We proceed by induction. The case $\rho=0$ is obvious. Indeed $\|\cP_t f\|_{\C^0}\le \|f\|_{\C^0}\le \|f\|_{\C^\beta}$.
	
	Assume now that the statement is proved for $\rho\in[0,M]$, $M\in\Z_+$. Let us prove it for $\rho\in(M,M+1]$. 
	
	Let $\rho=M+1$. If $\beta\ge1$, then by the definition of the norm
	\begin{align}\label{betalargeMpl1}
		\|\cP_t f\|_{\C^{M+1}}&\le\|\cP_t f\|_{\C^0}+\sup_{i\in\{1,\hdots,d\}}\|\d_i \cP_t f\|_{\C^{M}}\nn\\
		&\le \|f\|_{\C^0}+\sup_{i\in\{1,\hdots,d\}}\|\cP_t \d_i f\|_{\C^{M}}\nn\\
		&\le \|f\|_{\C^0}+Nt^{(\beta-M-1)\wedge0/\alpha}\sup_{i\in\{1,\hdots,d\}}\|\d_i f\|_{\C^{\beta-1}}\nn\\
		&\le Nt^{(\beta-M-1)\wedge0/\alpha}\|f\|_{\C^{\beta}},
	\end{align}
	where in the third inequality we have used the induction step.
	If $\beta=0$, then 
	\begin{align}\label{betazeroMpl1}
		\|\cP_tf\|_{\C^{M+1}}&\le\|\cP_tf\|_{\C^{0}}+\sup_{i\in\{1,\hdots,d\}}\|\d_i\cP_t  f\|_{\C^{M}}\nn\\
		&\le\|f\|_{\C^{0}}+\sup_{i\in\{1,\hdots,d\}}\|\cP_{t/2}\d_i \cP_{t/2} f\|_{\C^{M}}\nn\\
		&\le\|f\|_{\C^{0}}+N t^{-M/\alpha}\sup_{i\in\{1,\hdots,d\}}\|\d_i \cP_{t/2} f\|_{\C^{0}}\nn\\
		&\le\|f\|_{\C^{0}}+N t^{-M/\alpha}t^{-1/\alpha}\|f\|_{\C^{0}}\nn\\
		&\le N t^{-(M+1)/\alpha}\|f\|_{\C^{0}},
	\end{align}
	where in the third inequality we used the induction step and in the fourth inequality we used \eqref{eq:K1newA}. Finally, if $\beta\in(0,1)$, $f\in\C^\beta$, then define for $\lambda>0$ the interpolation function
	\begin{equation*}
		K(\lambda,f):=\inf_{\substack{f=a+b\\a\in\C^0(\R^d), b\in\C^1(\R^d)}}	(\|a\|_{\C^0}+\lambda \|b\|_{\C^1}).
	\end{equation*}
	It is well-known, that if $f\in\C^\beta$, then 
	\begin{equation}\label{interpolationtheorem}
		N^{-1}\|f\|_{\C^\beta}\le \sup_{\lambda>0} \frac{K(\lambda,f)}{\lambda^{\beta}}\le N \|f\|_{\C^\beta},
	\end{equation}
	see, e.g., \cite[Example~1.8]{Lunardi}. Then for any $a\in\C^0(\R^d)$, $b\in\C^1(\R^d)$ such that $f=a+b$ we get using \eqref{betalargeMpl1} and \eqref{betazeroMpl1}
	\begin{align*}
		\|\cP_t f\|_{\C^{M+1}}&=\|\cP_t (a+b)\|_{\C^{M+1}}\le N t^{-(M+1)/\alpha}\|a\|_{\C^0}+N t^{-M/\alpha}\|b\|_{\C^1}\\
		&\le Nt^{-(M+1)/\alpha}(\|a\|_{\C^0}+t^{1/\alpha} \|b\|_{\C^1}).
	\end{align*}
	Taking infimum over all $a\in\C^0(\R^d)$, $b\in\C^1(\R^d)$ such that $f=a+b$ and using \eqref{interpolationtheorem}, we get
	\begin{equation*}
		\|\cP_t f\|_{\C^{M+1}}\le Nt^{-(M+1)/\alpha}K(t^{1/\alpha},f)\le N\|f\|_{\C^\beta} t^{-(M+1)/\alpha}t^{\beta/\alpha},
	\end{equation*}
	which is \eqref{eq:K2A}. Thus, the case $\rho=M+1$, $\beta\ge0$ is proven.
	
	Finally if $\rho\in(M,M+1)$, then by above and the standard interpolation inequality 
	\begin{align*}
		\|\cP_t f\|_{\C^{\rho}}&\le \|\cP_t f\|_{\C^{M}}^{M+1-\rho}\|\cP_t f\|_{\C^{M+1}}^{\rho-M}\le  N\|f\|_{\C^\beta} t^{\frac{(\beta-M)\wedge0}\alpha(M+1-\rho)}t^{\frac{(\beta-M-1)\wedge0}\alpha(\rho-M)}\\
		&\le N\|f\|_{\C^\beta} t^{\frac{(\beta-\rho)\wedge0}\alpha}.
	\end{align*}
	
	Now let us prove \eqref{eq:K1A}. We start by assuming additionally that $f$ vanishes at infinity. First note that it follows from assumptions of the proposition that $1+\eps\ge \mu\ge(\beta-\rho)/\alpha$. Thus
	\begin{equation}\label{betarhoalpha}
		\beta-\rho-\alpha-\eps\alpha\le0.
	\end{equation}
	We see that for any $0\le s <t$, $x\in\R^d$,
	\begin{equation*}
		\cP_tf(x)-\cP_sf(x)=\int_s^t \d_r\cP_rf(x)\,dr=\int_s^t \L \cP_r f(x)\,dr,
	\end{equation*}
	where we used the fact that for $r>0$ we have $\cP_r f\in\C^2$ by \eqref{eq:K2A} and we applied It\^o's formula for L\'evy processes, see, e.g., \cite[Theorem 4.4.7]{Iphone}. Therefore, by \eqref{eq:K2newA} and \eqref{eq:K2A}, taking also into account \eqref{betarhoalpha} and that $\cP_r f$ vanishes at infinity for any $r\ge0$, we deduce
	\begin{align*}
		\|\cP_tf-\cP_sf\|_{\C^\rho}&\le \int_s^t \|\L \cP_r f\|_{\C^\rho}\,dr\le 
		N\int_s^t \| \cP_r f\|_{\C^{\rho+\alpha+\eps\alpha}}\,dr\\
		&\le N \| f\|_{\C^\beta}\int_s^t r^{\frac{\beta-\rho-\alpha-\eps\alpha}\alpha}\,dr\\
		&\le N \| f\|_{\C^\beta}\int_s^t s^{\frac{\beta-\rho}\alpha-\mu} (r-s)^{\mu-1-\eps}\,dr\\
		&\le  N \| f\|_{\C^\beta} s^{\frac{\beta-\rho}\alpha-\mu} (t-s)^{\mu-\eps},
	\end{align*}
	where in the penultimate inequality we used obvious bounds $r\ge s$  and $r\ge r-s$ (we note that the corresponding exponents are nonpositive since  $\frac{\beta-\rho}{\alpha}\le\mu\le 1+\eps$), and in the last bound we used that $\mu>\eps$ and the singularity is integrable. This implies \eqref{eq:K1A} for the case when $f$ vanishes at infinity.
	
	In general case, take a smooth function $\chi\colon\R_+\to[0,1]$ such that $\chi(x)=1$ for $x\in[0,1]$ and 
	$\chi(x)=0$ for $x\ge2$. Then the function $f_n(x):=f(x)\chi(|x|/n)$ vanishes at infinity and $\|f_n\|_{\C^\beta}\le N\|f\|_{\C^\beta}$ for some $N$ depending only on the choice of $\chi$. By above,
	\begin{equation*}
		\|\cP_tf_n-\cP_sf_n\|_{\C^\rho}\le   N \| f_n\|_{\C^\beta} s^{\frac{\beta-\rho}\alpha-\mu} (t-s)^{\mu-\eps}\le  N \| f\|_{\C^\beta} s^{\frac{\beta-\rho}\alpha-\mu} (t-s)^{\mu-\eps}.
	\end{equation*}
	By the dominated convergence theorem, $(\cP_t-\cP_s)f_n\to (\cP_t-\cP_s)f$ everywhere. Therefore,
	\begin{equation*}
		\|(\cP_t-\cP_s)f\|_{\C^\rho}\le \liminf_{n\to\infty}	\|(\cP_t-\cP_s)f_n\|_{\C^\rho}\le   N \| f\|_{\C^\beta} s^{\frac{\beta-\rho}\alpha-\mu} (t-s)^{\mu-\eps}.\qedhere
	\end{equation*}
\end{proof}
\end{appendix}
\bibliographystyle{Martin}
\bibliography{Bounded_Drift}
\end{document}